\documentclass[11pt,a4paper, reqno]{amsart}
\usepackage{mathrsfs}

\usepackage{amsmath,amsfonts,verbatim}
\usepackage{latexsym}
\usepackage{amssymb,leftidx}
\usepackage{extarrows}
\usepackage{overpic}
\usepackage{color}
\usepackage{epsfig}
\usepackage{subfigure}
\usepackage{tikz}
\usepackage[backref=page]{hyperref}

\newcommand{\mk}{\mathfrak}
\newcommand{\CS}{Y}
\newcommand{\LS}{\mathfrak{L}(Y)}
\newcommand{\cone}{X}
\newcommand{\WF}{\mathrm{WF}}

\hypersetup{urlcolor=blue, citecolor=red}

\usepackage{amssymb}
\usepackage{mathrsfs}
\usepackage{amscd}
\usepackage{bbm}
\usepackage{cite}

\newcommand\R{\mathbb{R}}
\newcommand\C{\mathbb{C}}
\newcommand\Z{\mathbb{Z}}
\newcommand\N{\mathbb{N}}

\newcommand\A{\bf A}

\newcommand\CC{\mathcal{C}}

\newcommand\LL{\mathcal{L}}

\usepackage{graphicx}
\usepackage{float}

\numberwithin{equation}{section}
\newtheorem{proposition}{Proposition}[section]
\newtheorem{definition}{Definition}[section]
\newtheorem{lemma}{Lemma}[section]
\newtheorem{theorem}{Theorem}[section]

\newtheorem{remark}{Remark}[section]

\begin{document}
\title[Strichartz estimates for electromagnetic wave]{Decay and Strichartz estimates \\ for critical electromagnetic wave equations \\ on conic manifolds}

\author{Qiuye Jia}
\address{Mathematical Sciences Institute, the Australian National University; }
\email{Qiuye.Jia@anu.edu.au; }

\author[J. Zhang]{Junyong Zhang}
\address{Junyong Zhang\newline
  Department of Mathematics, Beijing Institute of Technology, Beijing 100081}
\email{zhang\_junyong@bit.edu.cn}

\begin{abstract}
We establish the decay and Strichartz estimates for the wave equation with large scaling-critical  electromagnetic potentials on a conical singular space $(X,g)$ with dimension $n\geq3$, where the metric $g=dr^2+r^2 h$ and $X=C(Y)=(0,\infty)\times Y$ is a product cone over the closed Riemannian manifold $(Y,h)$ with metric $h$. 
The decay assumption on the magnetic potentials is scaling critical and includes the decay of Coulomb type.
The main technical innovation lies in proving localized pointwise estimates for the half-wave propagator by constructing a localized spectral measure, which effectively separates contributions from conjugate point pairs on $\CS$. 
In particular, when $Y=\mathbb{S}^{n-1}$, our results, which address the case of large critical electromagnetic potentials, extend and improve upon those in \cite{CS}, which considered sufficiently decaying, and small potentials and that of \cite{DF}, which considered potentials decaying faster than scaling critical ones.
\end{abstract}

\maketitle

\begin{center}
 \begin{minipage}{120mm}
   { \small {\bf Key Words:  Decay estimates,  Strichartz estimates,  Singular electromagnetic potentials, wave equation, conical singular space}
      {}
   }\\
    { \small {\bf AMS Classification:}
      { 42B37, 35Q40, 35Q41.}
      }
 \end{minipage}
 \end{center}

%\maketitle

\tableofcontents

\section{Introduction and main results}

Strichartz and decay estimates are fundamental tools in the study of nonlinear Schrödinger equations, nonlinear wave equations, and other nonlinear dispersive equations. In particular, global-in-time Strichartz estimates play a crucial role in establishing global well-posedness and scattering results for such equations. In this paper, we continue the work initiated in our previous study \cite{JZ, JZ2}, where we mainly considered the Schrödinger equation, by investigating the validity of Strichartz estimates for the wave equation in high-dimensional conical singular spaces in the presence of a scaling-critical large magnetic potential. Magnetic potentials, which act as first-order differential operator with $O(|x|^{-1})$ coefficients, are critical with respect to the scaling of the operator, making their analysis particularly challenging even in Euclidean space. For example, in the special case $Y=\mathbb{S}^{n-1}$, the validity of Strichartz estimates for the dispersive equations with Coulomb type decay magnetic potential is an open problem posed in \cite{DFVV, Fanelli, FV}, and is a longstanding gap left by \cite{CS, DF, EGS1, EGS2}. \vspace{0.2cm}

\subsection{Background and motivation} 
By a conic manifold, we mean
\begin{equation}
X=C(Y)=(0,\infty) \times Y,
\end{equation}
which is an $n$-dimensional product cone over the closed Riemannian manifold $(Y, h)$. $X$ is equipped with the metric
\begin{equation}
g=dr^2+r^2 h.
\end{equation}
We will in fact view it as its completion\footnote{One can check, this `closure' is indeed complete.} $\overline{X}=[0,\infty) \times Y$.
 
The Euclidean space $\R^n$ is the simplest example of a cone, with the cross-section $Y=\mathbb{S}^{n-1}$ with its standard metric.
While the general product cones $X$ possess a dilation symmetry analogous to that of Euclidean space, but no other symmetries in general.

We consider the evolution of a charged particle (on the non-relativistic energy scale) in a potentially curved spacetime $\mathbb{R}_t \times X$ in an electromagnetic field.
Denoting the dual metric of $g$ by $\big(g^{jk}\big)$ and the determinant by $|g|=\mathrm{det}\big(g_{jk}\big)$, then such evolution is governed by the time-dependent Schr\"odinger or wave equation associated to the operator on the product cone $X=C(Y)$:
\begin{equation}\label{H-AV}
\begin{split}
H_{A, V}&=\frac1{\sqrt{|g|}}\sum_{j,k=1}^n\Big(i\frac{\partial}{\partial x_j}+A_j\Big)\sqrt{|g|} g^{jk} \Big(i\frac{\partial}{\partial x_k}+A_k\Big)+V(x)\\
&=-\Delta_g-2iA\cdot\nabla_g-i\mathrm{div}_g A+|A|^2+V(x)
\end{split}
\end{equation}
where $\cdot$ is the pairing of 1-forms and vector fields, $V: \cone\to \R$ is the electric scalar potential and 
\begin{equation} \label{eq:defn-A}
A = A_jdx^j
\end{equation}
is the one form representing the magnetic vector potential, and $|A|^2$ is defined by $\sum_{j,k=1}^{n}g^{jk}A_jA_k$. 
When we take divergence of one forms, we always mean the divergence of the vector field obtained from the one form using the metric.
 
An important special case is when $Y=\mathbb{S}^{n-1}$ is the unit sphere, then $X=\R^n$ and $H_{A, V}$ is an electromagnetic Schr\"odinger operator in the Euclidean space.
Schr\"odinger operators with electromagnetic potentials have been extensively studied in the context of spectral and scattering theory. 
Important physical potentials, such as those corresponding to constant magnetic fields and the Coulomb potential in Euclidean space $X=\R^n$,  were investigated by Avron, Herbst, and Simon \cite{AHS1, AHS2, AHS3} and Reed and Simon \cite{RS}. The quantum-mechanical scattering of a test particle in conical space was considered by t’Hooft \cite{Hooft} and Jackiw et al \cite{Jackiw1, Jackiw2}; later the consideration was extended to the case of a magnetic vortex placed along the axis of conical space, i.e. to the quantum-mechanical scattering on a cosmic string \cite{Yu}.

Building on these foundational works, as well as recent studies \cite{DFVV, EGS1, EGS2, FFFP, FFFP1, GYZZ, FV, FZZ}, this paper investigates—from a mathematical perspective—how electric/magnetic potentials and conical geometry influence the dynamics of solutions to dispersive equations, particularly the wave equation. \vspace{0.2cm}

Despite those aforementioned progress, pointwise decay estimates and Strichartz estimates in general conical geometries—especially under scaling-critical electromagnetic potentials—remain largely unexplored due to their sensitivity to geometric perturbations. This motivates our work to quantify the influence of conjugate points and scaling-critical potentials on dispersive decay rates, bridging gaps between existing results.

Concretely, in Definition~\ref{defn:non-focusing-Lagrangian}, we propose a natural non-focusing condition (NFC) on the Lagrangian submanifold $\mathscr{L}_\pm$ of $T^*(\mathbb{R} \times \CS \times \CS)$ (see \eqref{eq: propagating lagrangians}) governing the wave propagation on conic manifolds. We expect this to be the sharp condition for localized dispersive estimates to hold for the full range of indices. This condition is satisfied by a large class of closed manifolds, see Proposition~\ref{prop:non-focusing-curvature-condition} for some examples where we use some curvature conditions to deduce this non-focusing condition.

Roughly speaking, for each fixed time in $(0,\pi]$, we require that the part of this Lagrangian submanifold near the diagonal of $Y \times Y$ is a finite sheet cover and each sheet projects down to $Y \times Y$ diffeomorphically. This corresponds to requiring the geodesic flow, which governs the wave propagation, is non-focusing there.
This generalizes the aforementioned Euclidean case to a much larger class of cones that is robust under perturbations. In particular, any cone with cross section having sectional curvature less than $2$ is in this class (see Proposition~\ref{prop:non-focusing-curvature-condition}). For example, cones over flat tori, negatively curved (closed) manifold, spheres with radius larger than $\frac{1}{2}$ are in this class.   

We expect this condition to be sharp for localized dispersive estimates. 
This has been confirmed when $Y=\rho \mathbb{S}^{n-1}$ is a sphere with radius $\rho$ by Taira \cite{Taira} using special functions and our condition perhaps reveals the reason behind this from an more general geometric perspective. 

\subsection{The setup}

In this paper, we consider the Schrödinger operator in form of \eqref{H-AV}, featuring Coulomb-type magnetic potentials and inverse-square electric potentials, that is
\begin{equation} \label{eq:A-form}
A = {\A}(\hat{x}), \quad V = \frac{a(\hat{x})}{r^2},\quad \hat{x} \in Y,
\end{equation}
where $\A$ is an 1-form on $\CS$ and the first equality is interpreted as viewing ${\A}$ as a form on $X$ that is translation (in $r$) invariant.
This imposes a condition that $A$ has no radial component: $g(A,dr)=0$, which is the so-called Cronstr\"om or transversal gauge. 
One should think of $\eqref{eq:A-form}$ as requiring 
\begin{equation}
A \sim O(r^{-1}), \quad r\to 0, +\infty
\end{equation}
in the language of existing literature in PDEs. 
Because components there were (implicitly) in terms of basis of 1-forms like $dx^j$, which are equivalent to $dr,rd\hat{x}$, with $\hat{x}$ being a coordinate system on $Y$.

More concretely, in the special case of the Euclidean space, if we use standard Cartesian coordinates on $\mathbb{R}^n$ to write $A$ and write ${\A}$ also as a vector with $n$-components via $\mathbb{S}^{n-1} \hookrightarrow \mathbb{R}^n$, then \eqref{eq:A-form} means 
\begin{equation}
A_jdx^j = {\A}_j d\hat{x}_j,
\end{equation}
so their components in two local coordinates respectively satisfies $A_j = \frac{{\A}_j}{r}$.

Using conventions above, the transversal gauge in the Euclidean setting reads
\begin{equation}\label{eq:transversal,Euclidean}
{\A}(\hat{x})\cdot\hat{x}=0,
\qquad
\text{for all }\hat{x} \in \mathbb{S}^{n-1} \hookrightarrow \mathbb{R}^n.
\end{equation}
Putting \eqref{eq:A-form} into \eqref{H-AV}, we set
\begin{equation}\label{eq:LAa-general}
\begin{split}
\mathcal{L}_{{\A},a}&=-\partial_r^2-\frac{n-1}r\partial_r+\frac1{\sqrt{|h|}}\sum_{j,k=2}^n\Big(i\frac{\partial}{\partial \hat{x}_j}+A_j\Big)\sqrt{|h|} h^{jk} \Big(i\frac{\partial}{\partial \hat{x}_k}+A_k\Big)+\frac{a(\hat{x})}{r^2}.
\end{split}
\end{equation}

Due to the radial-tangential decomposition of our metric, we have
\begin{equation}
\sqrt{\det (g_{jk})} = r^{n-1} \sqrt{\det (h_{jk})}.
\end{equation}
Combining with \eqref{eq:A-form}, we have
\begin{equation}
\mathrm{div}_gA = \mathrm{div}_{r^2h}({\A}) = r^{-2}\mathrm{div}_{h}({\A}),
\end{equation}
where the $r^{-2}$-factor arises since we are dualizing ${\A}$ to vectors using different metrics.
In addition, by the transversal gauge encoded in \eqref{eq:A-form}, we have
\begin{equation}
\langle A, \nabla_g f \rangle = {\A}( \nabla_{r^2h} f) = r^{-2}{\A}(\nabla_{h} f),
\end{equation}
where we used the invariance of ${\A}$ and the tangential part of $\nabla_gf$ (which is the only part having contribution to the pairing due to our transversal gauge condition) in the first step, and the pairing in the middle and right hand side are on $Y$ only, with $r$ being the scaling parameter. So in our setting the operator in \eqref{eq:LAa-general} becomes
\begin{equation}\label{eq:LAa-general'}
\begin{split}
\mathcal{L}_{{\A},a}&=-\Delta_g+\frac{|{\A}(\hat{x})|^2+i\,\mathrm{div}_{h}{\A}(\hat{x})+a(\hat{x})}{r^2}+2i {\A}(\hat{x})\cdot \nabla_g\\
&=-\partial_r^2-\frac{n-1}r\partial_r+ \frac{L_{{\A},a}}{r^2},
\end{split}
\end{equation}
where $\hat{x}\in \CS$, $|{\A}(\hat{x})|^2 = \sum_{i,j=1}^{n-1} h^{jk}{\A}_j{\A}_k$, and
\begin{equation}\label{L-angle}
\begin{split}
L_{{\A},a}
=-\Delta_{h}+\big(|{\A}(\hat{x})|^2+i\,\mathrm{div}_{h}{\A}(\hat{x})+a(\hat{x})\big)+2i {\A}(\hat{x})\cdot\nabla_{h}.
\end{split}
\end{equation}

This is symmetric on $L^2(Y)$ with pairing using the density $\sqrt{\det (h_{ij})}$ and we can apply the standard spectral theory.
In the special case of the Euclidean space, we have
\begin{equation}\label{eq:LAa-Euclidean}
\mathcal{L}_{{\A},a}=\Big(i\nabla+\frac{{\A}(\hat{x})}{|x|}\Big)^2+\frac{a(\hat{x})}{|x|^2},
\end{equation}
which is the exact  Schr\"odinger operator studied in \cite{FFT, FFFP, FFFP1, FZZ}.
Of course, \eqref{eq:LAa-general}\eqref{L-angle} are still valid in this special case, and with divergence, dot product interpreted as in the standard vector calculus.

The Euclidean Schrödinger operator \eqref{eq:LAa-Euclidean}, which includes the critical Coulomb-type decay potential, has garnered significant attention from both the mathematical and physical communities. For instance, the Aharonov-Bohm effect \cite{ES1949, AB59} arises from the critical Coulomb-type decay potential, and the diffractive behavior of waves in such potentials has been investigated in \cite{Yang1,Yang2}. The asymptotic behavior of time-independent Schrödinger solutions was analyzed in \cite{FFT}. Motivated by these developments and to address open problems left in \cite{DFVV, EGS1, EGS2, FV}, the authors \cite{JZ2} introduced a new method to prove Strichartz estimates for Schrödinger equations associated with the electromagnetic Schrödinger operator \eqref{eq:LAa-Euclidean} in $\R^n$ with dimension $n\geq3$. In this work, we study the Schrödinger operator \eqref{eq:LAa-general}, which requires simultaneous consideration of both electromagnetic potentials and conical geometry in its analysis.

%%%%%%%%%%%%%%%%%%%%%%%%%%%%%%
As discussed earlier, the dispersive and Strichartz estimates for this operator have been specifically investigated when $Y=\mathbb{S}^1$ in \cite{FFFP, FFFP1, FZZ, GYZZ}.
Furthermore, the spectral properties of this operator were analyzed in  \cite{FFT} for the more general case where $Y = \mathbb{S}^{n-1}$.
The aim of this paper is twofold.
First of all, the special case $Y = \mathbb{S}^{n-1}$, which corresponds to $X = \mathbb{R}^n$, fills the gap left by \cite{CS, DF} by studying the wave equation associated with the electromagnetic Schrödinger operator \eqref{eq:LAa-Euclidean}, where both the electric and magnetic potentials are singular at origin and scaling critical.

\subsection{Main results and assumptions on $Y$} 
Our main results are microlocalized decay estimates and global Strichartz estimates for wave equation. 
Before stating our main results, we need to introduce a few notions.

First of all, we introduce:
\begin{align}  \label{eq: propagating lagrangians}
\begin{split}
\mathscr{L}_{\pm}:= & \{ (s,\hat{y}, \hat{x},\tau,\mu_2,-\mu_1) \in T^*(\mathbb{R} \times \CS \times \CS): 
\\& \tau = \mp |\mu_1|_h, (\hat{y},\mu_2) = \exp(\pm s\mathsf{H}_{p})(\hat{x},\mu_1) \},
\end{split}
\end{align} 
where we use $p = |\mu|^2_h$ to denote the homogeneous principal symbol of $P$, and
\begin{equation*}
\mathsf{H}_{p} = (2|\mu|_h)^{-1}H_{p}
\end{equation*}
is the rescaled Hamilton vector field.
These are the Lagrangian submanifolds of $T^*(\mathbb{R} \times \CS \times \CS)$ governing the forward and backward wave propagations since $e^{\mp is \sqrt{P}}$ are Fourier integral operators associated to them. We will use this in Section~\ref{subsec:parametrix} to derive our desired form of parametrices.

Now we introduce the non-focusing condition (NFC) on $\mathscr{L}_\pm$ near the diagonal of $Y \times Y$:
\begin{definition}[The non-focusing condition (NFC)] \label{defn:non-focusing-Lagrangian}
Let $(Y,h)$ be a compact Riemanian manifold without boundary. We say that $Y$ has locally non-focusing wave propagation relation within time $\pi$ if there exists a neighborhood $\mathcal{U}$ of the diagonal of $Y \times Y$ and $\epsilon>0$ such that the projection
\begin{equation}
\mathscr{P}_\pm : \mathscr{L}_\pm \cap \big( T_{ (0,\pi+\epsilon) \times \mathcal{U} }^*(\R \times Y \times Y) \big)  \to Y \times Y 
\end{equation}
is a locally ray bundle. 
By a locally ray bundle, we mean for each $(\hat{x},\hat{y}) \in \mathcal{U}$, there is a neighborhood $U$ of it contained in $\mathcal{U}$ such that 
\begin{equation}
\mathscr{P}_{\pm}^{-1}(U)= \bigcup_{i}  U_i,
\end{equation}
where $i$ runs over a finite index set, $U_i$ are pairwise disjoint, and $\mathscr{P}_{\pm}|_{U_i \cap (S^*\R \times S^*Y \times S^*Y)}$ is a diffeomorphism. Since $\mathscr{L}_\pm$ is conic (i.e., invariant under the fiber dilation), this is equivalent to requiring that $\mathscr{P}_\pm|_{U_i}$ is a fiber bundle with fiber being the ray $\mathbb{R}_+$ corresponding to the fiber dilation on $T^*(\R \times Y \times Y)$.

If the conditions above are satisfied, we also say that $\mathscr{L}_\pm$ is non-focusing near the diagonal.
\end{definition}

The following proposition gives a large class of $(Y,h)$ satisfying our  non-focusing condition condition.
\begin{proposition} \label{prop:non-focusing-curvature-condition}
Let $(Y,h)$ be a closed Riemanian manifold satisfying one of the following:
\begin{itemize}
\item $(Y,h)$ has sectional curvature (not necessarily constant) $K<1$;
\item $(Y,h)$ is simply connected and has sectional curvature $\frac{1}{2} \leq K <2$;
\end{itemize}
then it has non-focusing wave propagation relation within time $\pi$ in the sense of Definition~\ref{defn:non-focusing-Lagrangian}. 
\end{proposition}

We postpone the proof of this proposition in the appendix. 

\begin{remark}
In particular, by the first case, any non-positively curved compact manifolds is in this class. 
For spheres, combining two cases, any sphere (of dimension at least $2$) with radius larger than $\frac{1}{2}$ is in this class.
\end{remark}

\begin{remark}
In the first case, without the requirement $\pi_1(Y)=0$, the upper bound $1$ is sharp since the real projective space equipped with metrics induced from the unit sphere has sectional curvature one and does not satisfy the property in Definition~\ref{defn:non-focusing-Lagrangian}. In the second case, the upper bound $2$ is sharp as one can see that the sphere with radius $\frac{1}{2}$ again fails to satisfy Definition~\ref{defn:non-focusing-Lagrangian}. 

On the other hand, the question that what is the sharp lower bound when one only imposes such type conditions seems to be a deep one. 
Notice that the sphere theorem breaks down in even dimensions if one includes the endpoint $\frac{1}{4}$ with the complex projective space with Fubini-Study metric being the counterexample. In fact, according to \cite{berger1983}, this together with some other symmetric spaces are the only counterexamples.
However, \cite{abresch1996sphere} showed that one can break the threshold $\frac{1}{2}$ (corresponding to $\frac{1}{4}$ when we normalized the upper bound to $1$) and lower bounds of the injective radius keep to hold. And in the odd dimensional setting, even the sphere theorem keep to hold with pinching slightly below $\frac{1}{4}$.

On the other hand, if one want such results through such a pinching condition, then a lower bound that is at least $\frac{1}{9}$ (if the upper bound is $1$) has to be imposed. This can be seen from the Bergers sphere $\mathbb{S}^3 = \mathrm{SU}(2)$ with left-invariant metric $\epsilon^2\sigma_1^2+\sigma_2^2+\sigma_3^2$ in terms of the Pauli matrices, which forms a basis of $\mk{su}(2)$. It has pinching that could tends to $\frac{1}{9}$ and arbitrarily small injective radius.
See \cite[Section~6.6.2]{petersen2006} for a more detailed discussion.
\end{remark}

Let $(Y,h)$ be a closed manifold with locally non-focusing propagation relation in the sense of Definition~\ref{defn:non-focusing-Lagrangian}. 
We consider a covering of $\CS$:
\begin{align} \label{eq: Uj covering}
\CS = \cup_{j=1}^{\mathsf{F}} \mathcal{U}_j
\end{align}
such that point pairs in each $\mathcal{U}_j$ are not conjugate to each other. In addition, let $\mathcal{U}$ be a neighborhood of the diagonal of $Y \times Y$ as in Definition~\ref{defn:non-focusing-Lagrangian}, we can choose each $\mathcal{U}_j$ small so that
\begin{equation} \label{eq: Uj condition}
\mathcal{U}_j \times \mathcal{U}_j \subset \mathcal{U}.
\end{equation}
In statements using non-resonant endpoint condition (NREC) in Definition~\ref{defn:NREC}, we further assume the diameter of $\mathcal{U}_j$ is less than $d_0>0$ in Proposition~\ref{prop: NREC-thicken-diagonal}.
Let $\{Q_j\}_{j=1}^{\mathsf{F}}$ be a partition of unity subordinate to this covering, which means
\begin{equation}\label{Id-p-Q}
\mathrm{Id}=\sum_{j=1}^{\mathsf{F}} Q_j,
\end{equation}
and each $Q_j$ is supported in $\mathcal{U}_j$. Then we have the following localized dispersive estimate.

\begin{theorem}[Microlocalized pointwise estimates for half-wave propagator]\label{thm:dispersive}
Let  $\LL_{{\A},a}$ be the Friedrichs self-adjoint extension operator given in \eqref{eq:LAa-general} on $n$-dimensional product cone $\cone=C(Y)$ ($n\geq 3$), where the cross-section  $\CS$  satisfies the NFC in the sense of Definition~\ref{defn:non-focusing-Lagrangian}, and where $a(\hat{x})\in C^{\infty}(\CS,\mathbb{R})$  and $1$-form ${\A}(\hat{x})={\A}_jd\hat{x}^j$, ${\A}_j\in C^{\infty}(\CS,\mathbb{R})$ are in \eqref{eq:A-form} such that the operator $P_{{\A}, a}:=L_{{\A},a}+(n-2)^2/4$ in \eqref{L-angle} is strictly positive on the cross section $\CS$. 
Let $\nu_0$ be the positive square roof of the smallest eigenvalue of the positive operator
$P_{{\A}, a}$.
Let $x=(r_1,\hat{x})$ and $y=(r_2,\hat{y})$ be in $(0,+\infty)\times\CS$ and let $Q_j$ be defined in \eqref{Id-p-Q}. Then, for any $k\in\Z$ and $\varphi \in C_c^\infty([1,2])$,
the kernel of half-wave propagator satisfies the following properties:

$\bullet$ if  either $2^{k}r_1\lesssim 1$ or $2^{k}r_2\lesssim 1$, then the frequency localized estimates
\begin{equation}\label{est:dispersive<}
 \begin{split}
\big| \varphi(2^{-k}\sqrt{\LL_{{\A},a}}) e^{it\sqrt{\LL_{{\A},a}}} (x,y)\big|\leq &C2^{kn}\big(1+2^k|t|\big)^{-\frac{n-1}2} \\
&\times\begin{cases}(2^{2k}r_1r_2)^{\nu_0-\frac{n-2}2},\quad &2^k r_1, 2^k r_2\lesssim 1;\\
(2^k r_1)^{\nu_0-\frac{n-2}2},\quad &2^k r_1\lesssim 1\ll  2^k r_2;\\
(2^k r_2)^{\nu_0-\frac{n-2}2},\quad &2^k r_2\lesssim 1\ll  2^k r_1;
 \end{cases}
\end{split}
\end{equation}

$\bullet$ if both $2^{k}r_1\gg 1$ and $2^{k}r_2\gg 1$,  then for all integers
$1\leq j\leq \mathsf{F}$, the microlocalized estimates 
\begin{equation}\label{est:dispersive}
 \begin{split}
\big| \big[Q_j \varphi(2^{-k}\sqrt{\LL_{{\A},a}}) e^{it\sqrt{\LL_{{\A},a}}} Q_j\big] (x,y)\big|\leq C2^{kn}\big(1+2^k|t|\big)^{-\frac{n-1}2}
\end{split}
\end{equation}
hold, where  $C$ is a constant independent of points $x,y\in(0,+\infty)\times\CS$. 
\end{theorem}

\begin{remark}
The special case $Y=\mathbb{S}^{n-1}$ satisfies a stronger condition NREC stated in \eqref{eq:NREC-quantitative} below since close geodesic loops has length $2\pi$.
\end{remark}

\begin{remark} The method for wave requires new ingredients compared with our previous works \cite{JZ, JZ2}, even though it is in the spirit of them. This is due to the differences in the construction of the wave and Schr\"odinger propagators. For example, in \cite{JZ}, we derived the decay estimates for the wave equation from the Schr\"odinger's dispersive estimates using the Littlewood-Paley theory and the subordination formula. However, in our current case, as in \cite{JZ2}, we have to use a localizer $Q_j$ to separate conjugate points on the cross section $\CS$, which poses an obstruction to deriving the wave decay estimates directly from the Schr\"odinger's decay estimates proved in  \cite{JZ2}. Unlike the method outlined in \cite{JZ}, this necessity compels us to independently derive the spectral properties and prove the wave decay estimates without relying on pre-existing frameworks.
\end{remark}

To state the second main result on the Strichartz estimates, we need the following definition of admissible pairs.
\begin{definition}\label{def:Lambda} 
In dimension $n\geq3$, for $s\in\R$, we define the set $\Lambda_s$ of $s$-admissible pairs to be the set of pairs $(q,p)\in [2,\infty]\times [2,\infty]$ satisfying the wave-admissible condition
\begin{equation}\label{adm}
2/q\leq (n-1)(1/2-1/p),\quad (q,p, n)\neq(2,\infty, 3),
\end{equation}
and the scaling condition
\begin{equation}\label{scaling}
s=n\big(1/2-1/p\big)-1/q.
\end{equation}
\end{definition}

For the solution to the wave equation without potentials in Euclidean space $\R^n$,
\begin{equation}\label{wave}
\partial_{t}^2u-\Delta u=F, \quad u(0)=u_0,
~\partial_tu(0)=u_1,
\end{equation}
it is well known that the following sharp Strichartz estimates hold for $(q,p)\in\Lambda_{s}$ and $(\tilde{q},\tilde{p})\in \Lambda_{1-s}$
\begin{equation}\label{est:Stri-f}
\|u(t,x)\|_{L^q(\R;L^{p}(\R^n))}\leq C\left(\|u_0\|_{\dot H^{s}(\R^n)}+\|u_1\|_{\dot H^{s-1}(\R^n)}+\|F\|_{L^{\tilde{q}'}(\R;L^{\tilde{p}'}(\R^n))}\right),
\end{equation}
see \cite{KT}. 
In the following, we will denote the Sobolev spaces associated to $\mathcal{L}_{{\A},a}$ by
\begin{align}\label{def:sobolev}
& \dot H^{s}_{{\A},a}(\cone):=\mathcal L_{{\A},a}^{-\frac s2}L^2(\cone),
\qquad
H^s_{{\A},a}(\cone) :=L^2(\cone)\cap\dot H^{s}_{{\A},a}(\cone).
\nonumber
\end{align}
Equivalently, the homogeneous Sobolev norm of $\|\cdot\|_{\dot H^{s}_{{\A},a}(\cone)}$ can be defined by
\begin{equation}\label{Sobolev-n}
\|f\|_{ \dot H^{s}_{{\A},a}(\cone)}=\Big(\sum_{j\in\Z}2^{2js}\|\varphi_{j}(\sqrt{\mathcal L_{{\A},a}})f\|_{L^2(\cone)}^2\Big)^{1/2},
\end{equation}
where $s\in\R$ and $\varphi_{j}(\sqrt{\mathcal L_{{\A},a}})$ is the Littlewood-Paley operator; see Section \ref{subsec:LP} for details.
For $n\geq3$ and $-1\leq s\leq 1$, we have  
\begin{equation}\label{nor-eq}
 \dot H^{s}_{{\A},a}(\cone)\sim \dot H^s(\cone)
 \end{equation}
similarly by \cite[cf. Lemma 2.3]{FFT} in combination with duality and interpolation.

\begin{theorem}\label{thm:stri-wave} Let  $\LL_{{\A},a}$ and $\nu_0$ be the same as in Theorem \ref{thm:dispersive} and $\Lambda_s$ be given in Definition \ref{def:Lambda}. 
For
\begin{equation}\label{def:alpha}
\alpha=-(n-2)/2+ \nu_0,
\end{equation}
define
\begin{equation}\label{def:q-alpha}
p(\alpha)=
\begin{cases}
\infty,\quad \alpha\geq 0;\\
\frac{n}{|\alpha|}, \quad -(n-2)/2<\alpha< 0,
\end{cases}
\end{equation}
and 
\begin{equation}\label{Ls}
\Lambda_{s,\alpha(\nu_0)}:=\{(q,p)\in\Lambda_s: p<p(\alpha) \}.
\end{equation}
For any $s\in\R$ and let $u$ be the solution to
\begin{equation}\label{eq:wave}
\partial_{t}^2u+\mathcal{L}_{{\A},a} u=0, \quad u(0)=u_0,
\quad \partial_tu(0)=u_1.
\end{equation}
Then, the following homogeneous Strichartz estimate holds:
\begin{equation}\label{est:Stri}
\|u(t,x)\|_{L^q(\R;L^{p}(\cone))}\leq C\left(\|u_0\|_{\dot H^{s}_{{\A},a}(\cone)}+\|u_1\|_{\dot H^{s-1}_{{\A},a}(\cone)}\right),
\end{equation}
provided $(q,p)\in \Lambda_{s,\alpha(\nu_0)}$. Moreover, if $(q,p)\in \Lambda_{s,\alpha(\nu_0)}$ and $(\tilde{q},\tilde{p})\in \Lambda_{1-s,\alpha(\nu_0)}$, the inhomogeneous Strichartz estimates hold:
\begin{equation}\label{est:in-Stri}
\Big\|\int_{\tau<t}\frac{\sin{\big((t-\tau)\sqrt{\LL_{{\A},a}}\big)}}
{\sqrt{\LL_{{\A},a}}}F(\tau)d\tau\Big\|_{L^q_t(\R;L^p(\cone))}\lesssim\|F\|_{L^{\tilde{q}'}_t(\R;L^{\tilde{p}'}(\cone))}.
\end{equation}
The restriction $p<p(\alpha)$ is necessary in the sense that the Strichartz estimates \eqref{est:Stri} may fail even if $(q,p)\in \Lambda_s$, but  $p\geq p(\alpha)$.
\end{theorem}

\begin{remark}\label{rem:set}
The set $\Lambda_{s,\alpha(\nu_0)}$ is non-empty only if $s\in [0,1+\nu_0)$. Moreover, $\Lambda_{s,\alpha(\nu_0)}=\Lambda_s$ when $s\in [0,1/2+\nu_0)$, in which case the condition $p<p(\alpha)$ automatically disappears; 
while $\Lambda_{s,\alpha(\nu_0)} \subsetneq \Lambda_s$ when $s \in [1/2+\nu_0, 1+\nu_0)$ in general. 
On the other hand, if $a(\hat{x})\geq 0$, then $\alpha>0$,
so that $\Lambda_{s,\alpha(\nu_0)}=\Lambda_s$ for $s\in\R$.
\end{remark}

\begin{remark}\label{rem:sobolev}
Due to \eqref{nor-eq}, the homogeneous Sobolev norm of $\|\cdot\|_{\dot H^{s}_{{\A},a}(\cone)}$ in  \eqref{est:Stri} can be replaced by the standard Sobolev norm $\|\cdot\|_{\dot H^{s}(\cone)}$ when $0\leq s\leq 1$.
\end{remark}

\begin{remark} The smallest eigenvalue  of $P_{{\A}, a}$
plays a crucial role in the Strichartz estimates. In particular, when $0<\nu_0<(n-2)/2$, there is a `red line' in diagrammatic picture of the range of $(q,p)$, which differs from the case without potentials. In the case without potentials,
$\nu_0 \geq (n-2)/2$, this red line disappears.
\end{remark}

\begin{center}
 \begin{tikzpicture}[scale=1]
\draw[->] (0,0) -- (4,0) node[anchor=north] {$\frac{1}{q}$};
\draw[->] (0,0) -- (0,4)  node[anchor=east] {$\frac{1}{p}$};
\draw (0,0) node[anchor=north] {O}
(3,0) node[anchor=north] {$\frac12$};
\draw  (0, 3) node[anchor=east] {$\frac12$}
       (0, 0.6) node[anchor=east] {$\frac12-\frac{1+\nu_0}{n}$}
       (0, 1.2) node[anchor=east] {$\frac12-\frac{1}{n-1}$};

\draw[thick] (3,0) -- (3,1.2)  %thick line 
              (3,1.2) -- (0,3);

  \filldraw[fill=gray!30](0,3)-- (0,0.5)-- (3,0.5)--(3,1.2);      
  
   \draw[red, dashed,thick](3,0.5) -- (0,0.5);       

\draw[dashed,thick] (0,1.2) -- (3,1.2); %dashed line
                   % (3,0.6) --  (0,1.8)
                    %(3,1.2) -- (0,2.4);
\draw (-0.1,3.2) node[anchor=west] {A};
\draw (2.9,1.2) node[anchor=west] {B};
\draw (2.9,0.6) node[anchor=west] {C};
\draw (2.9,0.15) node[anchor=west] {D};
\draw (-0.1,0.4) node[anchor=west] {E};
\draw (-0.1,1.0) node[anchor=west] {F};

\draw (1.65,2.88) node[anchor=west] {$\frac2q+\frac{n-1}{p}=\frac{n-1}{2}$};

\draw[<-] (1.6,2.1) -- (2,2.6) node[anchor=south]{$~$};

\path (2,-1) node(caption){Fig 1. $n\geq4$};  %caption

\draw[->] (8,0) -- (12,0) node[anchor=north] {$\frac{1}{q}$};
\draw[->] (8,0) -- (8,4)  node[anchor=east] {$\frac{1}{p}$};
\path (9.6,-1) node(caption){Fig 2. $n=3$};  %caption

\draw  (8.1, -0.1) node[anchor=east] {O};
\draw  (11, 0) node[anchor=north] {$\frac12$};
\draw  (8, 3) node[anchor=east] {$\frac12$}
(8, 1) node[anchor=east] {$\frac16$};

 \filldraw[fill=gray!30](8,0)-- (11,0)-- (8,3); 
 
 \draw[red, dashed,thick] (8,0) -- (11,0); 

\draw[thick] (8,3) -- (11,0);  %thick line
\draw[dashed,thick] (8,1) -- (11,0); %dashed line
\draw (7.9,3.15) node[anchor=west] {A};
\draw (10.9,0.2) node[anchor=west] {B};
\draw (7.9,1.15) node[anchor=west] {C};
\draw (10,2.6) node[anchor=west] {$\frac{2}{q}+\frac{2}{p}=1$};
\draw (10.7,1.3) node[anchor=west] {$\frac{1}{q}+\frac{3}{p}=\frac32-1$};

\draw (11,0) circle (0.06);

\draw[<-] (9,2.1) -- (10,2.6) node[anchor=south]{$~$};
\draw[<-] (10,0.5) -- (10.7,1.3) node[anchor=south]{$~$};

\path (6,-1.5) node(caption){Diagrammatic picture of the admissible range of $(q,p)$, when $\nu_0>1/(n-1)$.};  %caption

\end{tikzpicture}

\end{center}

\subsection{Literature review}  

The study of decay estimates and Strichartz estimates for dispersive equations has a long and rich history, owing to their central importance in both mathematical analysis and the theory of partial differential equations (PDEs). For the classical Schrödinger and wave equations with electromagnetic potentials, these estimates have been extensively investigated in both mathematical and physical contexts. We refer to \cite{BPSS, BPST, BG, CS, DF, DFVV, EGS1, EGS2, S} and the references therein for a comprehensive overview of these results. However, due to the diverse effects of different potentials, it is challenging to develop a universal framework that applies to all types of potentials. As a result, the overall understanding of this program remains incomplete, particularly for scaling critical physical potentials. In this direction, a substantial body of literature has focused on the decay behavior of dispersive equations under perturbations by various potentials. For subcritical magnetic potentials, several works (see \cite{CS, DF, DFVV, EGS1, EGS2, S, SchlagSurvey} and the references therein) have established time-decay and Strichartz estimates. In the case of the scaling-critical purely inverse-square electric potential, pioneering results were obtained by Burq, Planchon, Stalker, and Tahvildar-Zadeh \cite{BPSS, BPST}, who proved Strichartz estimates for the Schrödinger and wave equations in space dimensions $n\geq2$. However, when a magnetic field is introduced, the situation becomes significantly more complex. This is because the scaling-critical magnetic potential acts as first-order differential operator with $O(|x|^{-1})$ coefficients and induces a long-range perturbation, which complicates the analysis. Recently, Beceanu and Kwon \cite{BK} studied decay estimates for the three-dimensional Schrödinger equation with magnetic potentials, but their results require the potential to decay at infinity strictly faster than the Coulomb potential $|x|^{-1}$. \vspace{0.1cm}

For the Coulomb type decay magnetic Schr\"odinger operators,  these estimates (decay and Strihcartz estimates) are known only in $\R^2$ and only for Aharonov-Bohm type magnetic fields, which is the special cases of \eqref{H-AV} that $Y=\mathbb{S}^1$, $A(x)={\A}(\hat{x})r^{-1}$ (if we write one forms in local coordinates on $X$ and $Y$ respectively and compare components) and $V=a(\hat{x})r^{-2}$, see  \cite{FFFP, FFFP1, FZZ, GYZZ}. We note that the aforementioned results for the scaling-critical long-range magnetic case are currently limited to two dimensions. The success of the argument relies on the simple structure of the cross section  $\mathbb{S}^1$  and the potentials, where the absence of conjugate points and the explicit knowledge of the eigenfunctions and eigenvalues of the magnetic Laplacian on $\mathbb{S}^1$ play a crucial role. Even dispersive estimates for a scaling-critical class of electric potentials (with respect to the global Kato norm) were established in \cite{BG} recently,
however, to the best of our knowledge, no Strichartz estimates have been established for scaling-critical magnetic Schrödinger operators in higher dimensions $n\geq3$ in the presence of both singular magnetic and electric potentials. For higher dimensions, we refer to  \cite{DFVV, BK, FV, EGS1, EGS2} for results on almost-critical magnetic Strichartz estimates, though the critical Coulomb case remains unresolved. \vspace{0.2cm}

While the aforementioned results concern Euclidean space, significant attention has also been given to the case of the cone, where the interplay between conical geometry and dispersive equations reveals rich phenomena. For the Laplacian without potentials, Cheeger and Taylor \cite{CT1, CT2} studied wave diffraction on conical manifolds, later extended to multi-cone settings by Ford and Wunsch \cite{FW}. Müller and Seeger \cite{MS1} analyzed wave propagation regularity, while Schrödinger dynamics with inverse-square potentials were examined in \cite{Carron, wang}, and Riesz transforms in \cite{HL}. 
Schlag, Soffer and Staubach further developed angular-momentum-dependent decay estimates for Schrödinger and wave equations on manifolds with conical ends \cite{SSS1, SSS2}. Pointwise dispersive estimates for Schrödinger equations on product cones were proved by Keeler and Marzuola \cite{KM}, and local-in-time dispersive/Strichartz estimates for general conic manifolds without conjugate points were proved by Chen \cite{Chen}. The special case of the flat 2D Euclidean cone $X=C(\mathbb{S}^1_{\sigma})$has been particularly fruitful, with Ford establishing Schrödinger dispersive estimates \cite{Ford}, and Blair-Ford-Marzuola \cite{BFM} obtaining decay estimates for $\sin(t\sqrt{\Delta_g})/\sqrt{\Delta_g}$ while conjecturing pointwise decay for $\cos(t\sqrt{\Delta_g})$. Most recently, \cite{Z} resolved this conjecture through resolvent kernel constructions. The Strichartz estimates for Schr\"odinger and wave in a general conical setting (without assumption on the conjugate radius of $Y$) have been proved by Zheng and the last author in \cite{ZZ1, ZZ2}.

The study of decay estimates for the Schrödinger and wave equations in the presence of potentials has undergone significant developments over the years. In the case of a pure electric potential (i.e., \( A = 0 \)), decay estimates for the Schrödinger equation were first established in \cite{JSS}. For the wave equation, analogous results were derived in \cite{Bea}, albeit under the assumption that the potential \( V \) belongs to the Schwartz class. These foundational works were subsequently refined and extended in a series of studies, including those in \cite{BG, DFVV, DF, EGS1, EGS2}.

For Schrödinger equations involving magnetic potentials, Strichartz and smoothing estimates have been developed by D'Ancona, Fanelli, Vega, Visciglia in \cite{DFVV} and Erdogan, Goldberg, Schlag in \cite{EGS1, EGS2}, where the assumption on the decay of potential is $O(\langle x \rangle^{-(1+\epsilon})$. 
For wave equations with magnetic potential, the Strichartz estimates were shown by Cuccagna, Schirmer \cite{CS} and D'Ancona, Fanelli \cite{DF}. 
In particular, in \cite{DF}, \( t^{-1} \) decay estimates were established for the wave equation with a small magnetic potential, assuming initial data in weighted \( H^s \) spaces. 
Very recently, \cite{BK} demonstrated that this statement holds even on \(\mathbb{R}^3\) when \( A \) is a short-range perturbation, meaning the potential decays at infinity strictly faster than \( |x|^{-1} \).
Despite these advancements, a significant gap remains in the literature when our focus lies on the scaling critical Coulomb potential, which decays at infinity exactly as \( |x|^{-1} \). 
As mentioned above,  for 2D the {\it Aharonov-Bohm} potential, \cite{FFFP, FFFP1} studied the Schr\"odinger equation and \cite{FZZ, GYZZ} studied the wave and Klein-Gordon equations. 
For higher dimensions, the problem was open until \cite{JZ2} for the Schrödinger equation.
Apart from dispersive and Strichartz estimates, the asymptotics (or, Price's law) in very similar geometric setting is considered in \cite{BM} and that for Dirac equation with similar potential (but being constant on the cross section) is considered \cite{BGRM-price}.
Our approach for wave here builds on the methodology introduced in \cite{JZ, JZ2}, which relies on the parametrix construction. Consequently, the wave results presented here are entirely novel and address a previously unexplored aspect of the problem.  In summary, our results are novel and still under active development, with progress in the following settings:  
\vspace{0.2cm}

\noindent$\bullet$ {\bf Euclidean space (special case):} When the cross section is spherical (\( Y = \mathbb{S}^{n-1} \)), the magnetic potential exhibits Coulomb decay, scaling as \( O(|x|^{-1}) \).  \vspace{0.1cm}

\noindent$\bullet$ {\bf Product cone (general case):} For a cone with a general cross section \( Y \), the analysis extends naturally, preserving key qualitative features.  

\subsection{The structure of this paper}

In Section~\ref{sec:tools}, we recall and introduce some ingredients of the harmonic analysis and microlocal analysis adapted to our current setting. Specifically, in Section~\ref{subsec:spectral-analysis}, we recall basic facts for reducing the problem on $X$ to that on $Y$.
In Section~\ref{subsec:geometric-pre}, we introduce the concept of distance spectrum, which will be used in our expression of the wave propagators. Also, we introduce the technical assumption: non-resonant endpoint condition (NREC) in Definition~\ref{defn:NREC}. Roughly speaking, this condition says that geodesics on $Y$ won't come back to the starting point around time $\pi$. 
We emphasize that this NREC is merely a technical assumption that allows one to simplify the proof in certain cases (which includes the most classical case $\CS = \mathbb{S}^{n-1}$), but the non-focusing condition is instead very intrinsically crucial for the dispersive and Strichartz estimate.
This in turn corresponds to the requirement that geodesics on $X$ won't come back to the starting point when they go back to infinity.
In Section~\ref{subsec:parametrix}, we give the concrete form of the parametrix of wave propagators that we will use.
Then in Section~\ref{subsec:LP}, we introduce the Littlewood-Paley theory adapted to our Schr\"odinger operator.

In Section~\ref{sec:kernel-spectral-measure}, we give the expression of the spectral measure associated with the operator $\mathcal{L}_{{\A},a}$. The key point is that we can write the spectral measure, hence also the propagator $e^{- i t \mathcal{L}_{{\A},a}}$ as an oscillatory integral in terms of $\cos(s\sqrt{P})$ and $\sin(\pi \sqrt{P})e^{-s \sqrt{P}}$, allowing us to apply the parametrix we obtained on $Y$.

%%%%%%%%%%%%%%%%%%%%%%%%%%%%%%%%%%%%%%%% 
In Section~\ref{sec:proof-with-NREC}, we will give the proof of our results under the aforementioned NREC condition. The  NREC condition is satisfied when $Y=\mathbb{S}^{n-1}$ due to the fact that the injective unit sphere is $\pi$. On a more technical level, this allows us to say that the contribution from $\cos(s\sqrt{P})$ with $s$ close to $\pi$  and all contribution from $\sin(\pi \sqrt{P})e^{-s \sqrt{P}}$, which corresponds to endpoint part of propagation and the diffraction part and are collected in the $I_R$-term in Section~\ref{sec:proof-with-NREC}, are almost negligible and the proof of our main results (in particular, Theorem~\ref{thm:dispersive}) will be simplified drastically and become conceptually more transparent. 

In Section~\ref{sec:striW}, we prove the Strichartz estimate (Theorem~\ref{thm:stri-wave})  using the localized dispersive estimate (Theorem~\ref{thm:dispersive}).

Finally, we remove the NREC assumption in Section~\ref{sec:without-NREC}.
This is achieved by analyzing the kernel including the part near the endpoint directly by exploiting the cancellation of the propagating part at the endpoint and the diffraction part at the starting point. See
Lemma~\ref{lem:in-parts} and Proposition~\ref{prop:localized-GD-IBP}. 
In the language of \cite{CT1,CT2}, this corresponds to analyzing the interface between `region II' and `region III' there.

{\bf Acknowledgments:} \quad  The authors would like to thank Andrew Hassell for his helpful discussions and thank Jiqiang Zheng for reading the manuscript carefully. 
J. Zhang is grateful for the hospitality of the Australian National University when he was visiting Andrew Hassell at ANU. 
J. Zhang was supported by National key R\&D program of China: 2022YFA1005700, National Natural Science Foundation of China(12171031) and Beijing Natural Science Foundation(1242011);
Q. Jia was supported by the Australian Research Council through grant FL220100072.

\vspace{0.2cm}
\section{Some microlocal and harmonic analysis tools} \label{sec:tools}
In this section, we recall some microlocal and harmonic analysis tools about the parametrix and the Littlewood-Paley theory, which were studied in \cite{JZ2,JZ}.
We highlight that our condition on the cross-section $Y$ strictly generalizes the settings considered in  \cite{JZ2,JZ}.

\subsection{Spectral analysis for $\mathcal{L}_{{\A},a}$} 
\label{subsec:spectral-analysis}
In local coordinates $x=(r, \hat{x})$, the metric reads $g=\big(g_{jk}\big)=\left( \begin{smallmatrix}
1 & {\bf 0} \\
{\bf 0} & r^2 h(\hat{x})
\end{smallmatrix} \right)$. We recall \eqref{eq:LAa-general'}
\begin{equation}\label{LAa-r}
\begin{split}
\mathcal{L}_{{\A},a}&=-\Delta_g+\frac{|{\A}(\hat{x})|^2+i\,\mathrm{div}_{h}{\A}(\hat{x})+a(\hat{x})}{r^2}+2i\frac{{\A}(\hat{x})}{r}\cdot\nabla_g\\
&=-\partial_r^2-\frac{n-1}r\partial_r+\frac{L_{{\A},a}}{r^2},
\end{split}
\end{equation}
where $\hat{x}\in \CS$ and
\begin{equation} 
\begin{split}
L_{{\A},a}=-\Delta_{h}+\big(|{\A}(\hat{x})|^2+i\,\mathrm{div}_{h}{\A}(\hat{x})+a(\hat{x})\big)+2i {\A}(\hat{x})\cdot\nabla_{h}.
\end{split}
\end{equation}
From the classical spectral theory, the spectrum of $L_{{\A},a}$ is formed by a countable family of real eigenvalues with finite multiplicity $\{\mu_k({\A},a)\}_{k=0}^\infty$ enumerated such that
\begin{equation}\label{eig-Aa}
\mu_0({\A},a)\leq \mu_1({\A},a)\leq \cdots 
\end{equation}
where we repeat each eigenvalue as many times as its multiplicity, and $\lim\limits_{k\to\infty}\mu_k({\A},a)=+\infty$. We refer to \cite[Lemma A.5]{FFT} for the special case $Y=\mathbb{S}^{n-1}$ .
For each $k\in\N$, let $\psi_k(\hat{x})\in L^2(\CS)$ be the normalized eigenfunction of the operator $L_{{\A},a}$ corresponding to the $k$-th eigenvalue $\mu_k({\A},a)$, i.e. satisfying that
\begin{equation}\label{equ:eig-Aa}
\begin{cases}
L_{{\A},a}\psi_k(\hat{x})=\mu_k({\A},a)\psi_k(\hat{x}) \quad \hat{x} \in  \CS,\\
\int_{\CS}|\psi_k(\hat{x})|^2 d\hat{x}=1.
\end{cases}\end{equation}
Notice that the operator $P_{{\A},a}$ defined in Theorem \ref{thm:dispersive} is related to $L_{{\A},a}$ by
$$P=P_{{\A},a}=L_{{\A},a}+(n-2)^2/4,$$
thus they have the same eigenfunctions and the difference of their eigenvalues is the constant $(n-2)^2/4$.
\vspace{0.2cm}

We have the orthogonal decomposition $$L^2(\CS)=\bigoplus_{k\in\N}h_{k}(\CS),$$
where 
\begin{equation}\label{hk}
h_{k}(\CS)=\text{span}\{\psi_k(\hat{x})\}.
\end{equation}
For $f\in L^2(\cone)$, we can write $f$ in the following form by separation of variables:
\begin{equation}\label{sep.v}
f(x)=\sum_{k\in\N} c_{k}(r)\psi_k(\hat{x}),
\end{equation}
where
\begin{equation*}
 c_{k}(r)=\int_{\CS}f(r,\hat{x})
\overline{\psi_k(\hat{x})} \, \sqrt{|\det h|} d\hat{x}.
\end{equation*}
Hence, on each space $\mathcal{H}^{k}=\text{span}\{\psi_k\}$, from \eqref{LAa-r}, we have
\begin{equation*}
\begin{split}
\LL_{{\A},a}=-\partial_r^2-\frac{n-1}r\partial_r+\frac{\mu_k}{r^2}.
\end{split}
\end{equation*}
For $\nu >-\frac{1}{2}$, $f\in L^2(\cone)$ we define the Hankel transform by
\begin{equation}\label{hankel}
(\mathcal{H}_{\nu}f)(\rho,\hat{x})=\int_0^\infty (r\rho)^{-\frac{n-2}2}J_{\nu}(r\rho)f(r,\hat{x}) \,r^{n-1}dr,
\end{equation}
where the Bessel function of order $\nu$ is given by
\begin{equation}\label{Bessel}
J_{\nu}(r)=\frac{(r/2)^{\nu}}{\Gamma\left(\nu+\frac12\right)\Gamma(1/2)}\int_{-1}^{1}e^{isr}(1-s^2)^{(2\nu-1)/2} ds, \quad \nu>-1/2, r>0.
\end{equation}

Using the functional calculus, for a Borel measurable function $F$ (see \cite{Taylor}), we define $F(\mathcal{L}_{{\A},a})$ by
\begin{equation}\label{funct}
F(\mathcal{L}_{{\A},a}) f(r_1,\hat{x})=\int_0^\infty \int_{\CS} K(r_1,\hat{x},r_2,\hat{y}) f(r_2,\hat{y})\; r^{n-1}_2\;dr_2\; \sqrt{|\det h|}d\hat{y}, 
\end{equation}
where
$$\nu_k=\sqrt{\mu_k+(n-2)^2/4}, \quad K(r_1,\hat{x},r_2,\hat{y})=\sum_{k\in\N}\psi_{k}(\hat{x})\overline{\psi_{k}(\hat{y})}K_{\nu_k}(r_1,r_2),$$
and
\begin{equation}\label{equ:knukdef}
  K_{\nu_k}(r_1,r_2)=(r_1r_2)^{-\frac{n-2}2}\int_0^\infty F(\rho^2) J_{\nu_k}(r_1\rho)J_{\nu_k}(r_2\rho) \,\rho d\rho.
\end{equation}
Similarly, recalling $\nu=\nu_k=\sqrt{\mu_k+(n-2)^2/4}$, the square root of the eigenvalue of the operator $P=L_{{\A},a}+(n-2)^2/4$,  and using the spectral theory, as well as \cite{Taylor}, we have
\begin{equation}\label{FA}
\tilde{F}(\sqrt{P})=\sum_{k\in\N}\psi_{k}(\hat{x})\overline{\psi_{k}(\hat{y})} \tilde{F}(\nu_k),
\end{equation}
where $\tilde{F}$ is a Borel measurable function. 

\subsection{Geometric preliminaries}
\label{subsec:geometric-pre}

We recall some geometric facts that we need in our parametrix construction from \cite{JZ}.
We take the symplectic (instead of Riemannian) perspective to view the geodesic flow of $Y$ as a flow on $T^*Y$, which is the Hamilton flow associated to (the principal symbol of) $\Delta_h$. Also, we consider the exponential map as a map defined on $T^*Y$.

%%%%%%%%%%%%%%%%%%%%%
Now we introduce the notion of the distance spectrum, which generalizes the distance function and will be the key ingredient of phase functions parametrizing $\mathscr{L}_\pm$. For $(\hat{x},\hat{y}) \in Y \times Y$ and $\hat{\mu}=\mu|\mu|_h^{-1}$, we define the forward/backward distance spectrum associated to $(\hat{x},\hat{y}) \in Y \times Y$ to be
\begin{align} \label{eq:defn-distance-spectrum}
\begin{split}
\mathfrak{D}_\pm (\hat{x},\hat{y}) = & \{ \mathfrak{d} \in [0,\pi+\epsilon): 
\exists \hat{\mu}_2 \in S^*_{\hat{y}}Y, \;\hat{\mu}_1 \in S^*_{\hat{x}}Y \text{ such that }
\\ & \exp(\mp \mathfrak{d}\mathsf{H}_{p})(\hat{y},\hat{\mu}_2) = (\hat{x},\hat{\mu}_1) \},
\end{split}
\end{align}
which is a collection of smooth functions $\mathfrak{d} (\hat{x},\hat{y})$ of $\hat{x},\hat{y}$. Here we count $\mathfrak{d}$ with multiplicity for different $\hat{\mu}_2$ (and corresponding $\hat{\mu}_1$). 
\begin{remark} \label{remark:D+- and D}
Since we allow the momentum run over the entire $S^*_{\hat{y}}Y$, and the forward $\mathsf{H}_{p}$-flow starting at $(\hat{y},\hat{\mu}_2)$ is the same as the backward flow starting at $(\hat{y},-\hat{\mu}_2)$, so $\mathfrak{D}_+(\hat{x},\hat{y})$ is actually the same as $\mathfrak{D}_-(\hat{x},\hat{y})$. We only keep the $\pm$ sign to emphasize which of $e^{\mp is\sqrt{P}}$ are we considering and we will denote it by $\mathfrak{D}(\hat{x},\hat{y})$ when we consider $\cos(s\sqrt{P})$.
\end{remark}

Equivalently, that is all those $\mathfrak{d} \in [0,\pi+\epsilon)$ such that there is a (unit speed) geodesic $\gamma$ (with loops counted with multiplicity) starting at $\hat{y}$ with $\gamma(\mathfrak{d}) = \hat{x}$.
In particular, when $d_h(\hat{x},\hat{y})<\mathrm{inj}(Y)$ (hence $d_h(\hat{x},\hat{y})$ is smooth and realized by the unique distance minimizing geodesic), $d_h(\hat{x},\hat{y}) \in \mathfrak{D}_\pm(\hat{x},\hat{y})$.

By our assumption that $\mathscr{L}_\pm$ is non-focusing in the sense of Definition~\ref{defn:non-focusing-Lagrangian}, we know that for each $(\hat{x},\hat{y}) \in \mathcal{U}$, there are only finitely many geodesics within length $\pi$ connecting $\hat{x},\hat{y}$ and they corresponds to sheets $U_i$ on $\mathscr{L}_\pm$ in Definition~\ref{defn:non-focusing-Lagrangian} in a one to one manner. After replacing $\mathcal{U}$ by a smaller region whose closure (which is compact) is contained in $\mathcal{U}$, we only need to take finitely many such neighborhoods to cover it and we know that the size of $\mk{D}(\hat{x},\hat{y})$ for $(\hat{x},\hat{y})$ in $\mathcal{U}$ is uniformly bounded.

Next we introduce a technical condition.
\begin{definition}[non-resonant endpoint condition(NREC)] \label{defn:NREC}
Let $\LS \subseteq \mathbb{R}$ denote \emph{the length spectrum} of $Y$, which is the set of lengths of closed geodesics (allowed to go through the same trajectory more than one times and allow the tangent vector at the endpoint to be different from that at the starting point) on $Y$. Then we say $(Y,h)$ satisfies the \emph{non-resonant endpoint condition (NREC)} on (in fact on $X$, so that $\pi$ is actually the `endpoint' of geodesic flows) if
\begin{center}  
\emph{$\pi \notin \overline{\LS}$.}
\end{center}
This implies that there exists $\delta_0>0$ such that
\begin{equation} \label{eq:NREC-quantitative}
[\pi-\delta_0,\pi+\delta_0] \cap \LS = \emptyset.
\end{equation}
\end{definition}

\begin{remark}
This NREC is satisfied if the injective radius of the section cross $Y$ is strictly greater than $\frac{\pi}2$. In particular, this NREC is satisfied for our special example $Y=\mathbb{S}^{n-1}$ whose injective radius is $\pi$. Notice that using the injective radius as the criterion puts restrictions on the global obstruction for the exponential to be a diffeomorphism, while our non-focusing condition in Definition~\ref{defn:non-focusing-Lagrangian} concerns only the local obstruction (i.e., its degeneracy). For example, a torus with radius of one of its circles being $\frac{1}{2k}, k \in \mathbb{Z}_+$ will satisfy Definition~\ref{defn:non-focusing-Lagrangian} but not this NREC, and we need to use the proof in Section~\ref{sec:without-NREC} for those cases.  
\end{remark}

Now we show that when this NREC is satisfied, we have control not only on the length of geodesic returning to the same point, but also geodesics returning to a point that is very close nearby.

\begin{proposition}  \label{prop: NREC-thicken-diagonal}
Suppose $(Y,h)$ satisfies NREC and let $\delta_0$ be as in \eqref{eq:NREC-quantitative}, then there exists $d_0>0$ such that whenever $d_h(\hat{x},\hat{y})<d_0$, then geodesics connecting $\hat{x}$ and $\hat{y}$ won't have length in $[\pi-\delta_0,\pi+\delta_0]$.
\end{proposition}

\begin{proof} We prove this by using the contradiction argument.
Otherwise, there is a sequence of point pairs $(y_{1,i},y_{2,i})$ such that 
\begin{equation} \label{eq: y1y2-distance-0}
d_h(y_{1,i},y_{2,i}) \to 0,
\end{equation}
and there are some $s_i \in [\pi-\delta_0,\pi+\delta_0]$ and $\hat{\mu}_{1,i} \in S^*_{ y_{1,i} } Y,\hat{\mu}_{2,i} \in S^*_{y_{2,i}}Y$ such that 
\begin{equation}
\exp(s_i\mathsf{H}_{\frac{1}{2}p})(y_{2,i},\hat{\mu}_{2,i}) = (y_{1,i},\hat{\mu}_{1,i}).
\end{equation}
By the compactness of $[\pi-\delta_0,\pi+\delta_0] \times S^*Y \times S^*Y$, we know that after passing to a subsequence we can assume that 
$s_i \to s_{\infty} \in [\pi-\delta_0,\pi+\delta_0]$, $(y_{1,i},y_{2,i}) \to (y,y)$ (recalling \eqref{eq: y1y2-distance-0}) and $\hat{\mu}_1, \hat{\mu}_2 \in S^*_{y}Y$ such that  
\begin{equation}
\exp(s_\infty \mathsf{H}_{\frac{1}{2}p})(y,\hat{\mu}_2) = (y,\hat{\mu}_1), 
\end{equation}
which contradicts \eqref{eq:NREC-quantitative}.
\end{proof}

\subsection{The Hadamard parametrix} 
\label{subsec:parametrix}

In this subsection, we recall the localized parametrices of the half-wave propagators $e^{\pm is\sqrt{P}}$, which gives the same type of expressions of the even wave propagator $\cos(s\sqrt{P})$ and the Poisson wave propagator $e^{(-s\pm i\pi)\sqrt{P}}$, where $P=P_{{\A}, a}$ is the operator on $\CS$. Similar results have been proved in \cite[Section~3]{JZ2} when $\CS = \mathbb{S}^{n-1}$, which in turn used the global parametrix construction in \cite[Section~3]{JZ}.

We first give the oscillatory integral expression of the half wave propagator $Q_je^{\pm is\sqrt{P}}Q_j$. 
\begin{lemma}[Half-wave propagator]\label{lemma: localized-halfwave}
Let $\mk{D}_{\pm}(\hat{x},\hat{y})$ be as in \eqref{eq:defn-distance-spectrum} and $P=P_{{\A},a}$ be as in Theorem \ref{thm:dispersive}, then, for each $Q_j$ with $1\leq j\leq \mathsf{F}$ given in \eqref{Id-p-Q},  the kernel of $Q_j e^{\pm is\sqrt{P}} Q_j$ can be written as 
\begin{equation}\label{eq:half-wave-kernel}
Q_je^{\pm is\sqrt{P}}(\hat{x},\hat{y})Q_j = K_{\pm,N}(s; \hat{x},\hat{y})+R_{\pm,N}(s; \hat{x},\hat{y}),
\end{equation}
where $R_{\pm,N}(s; \hat{x},\hat{y}) \in C^{N-n-2}([0,\pi] \times Y \times Y)$ and
\begin{equation} \label{K-pm-N}
\begin{split}
K_{\pm,N}(s; \hat{x},\hat{y})&=(2\pi)^{n-1} \sum_{\mk{d} \in \mk{D}_\pm (\hat{x},\hat{y})} \int_{\R^{n-1}} e^{i \mk{d}(\hat{x},\hat{y}){\bf 1}\cdot\xi} a_{\pm,\mk{d}}(s, \hat{x},\hat{y}; |\xi|) e^{\pm  is|\xi|} d\xi\\
&= \sum_{\mk{d} \in \mk{D}_\pm (\hat{x},\hat{y})} \sum_{\varsigma = \pm}\int_0^\infty b_{\varsigma }(\rho \mk{d}) e^{\varsigma i \rho \mk{d}} a_{\pm,\mk{d}}(s, \hat{x},\hat{y}; \rho) e^{\pm  is\rho} \rho^{n-2} d\rho.
\end{split}
\end{equation}
In addition, one may choose $a_{+,\mk{d}}=a_{-,\mk{d}}$ to satisfy
 \begin{equation}\label{a}
 |\partial^\alpha_{s,\hat{x},\hat{y}}\partial_\rho^k a_{\pm, \mk{d}}(s,\hat{x},\hat{y};\rho)|\leq C_{\alpha,k}(1+\rho)^{-k},
 \end{equation}
 and
 \begin{equation}\label{b+-}
\begin{split}
| \partial_r^k b_\pm(r)|\leq C_k(1+r)^{-\frac{n-2}2-k},\quad k\geq 0.
\end{split}
\end{equation}
Furthermore, under the NREC condition and let $\delta_0>0$ be as in \eqref{eq:NREC-quantitative}, 
then one can choose $a_{\pm,\mk{d}}(s,\hat{x},\hat{y},\rho)=0$ for $s>\pi-\frac{\delta_0}{2}$.
\end{lemma}
For the rest of this paper, we keep the $\pm$ sub-indices to indicate which operator those amplitudes are associated to and make distinction between $a$, but one should keep in mind that they are actually the same function.

\begin{remark}
The conclusion that one can put the part of the propagator with $s$ large into the residual part and choose $a_{\mk{d}}(s,\hat{x},\hat{y},\rho)=0$ there should not be surprising, as it can be interpreted as a formulation of the finite speed of propagation (of singularities) for the half-wave propagators $e^{\pm is\sqrt{P}}$.
\end{remark}

\begin{proof}
The form of the parametrix in \eqref{KN1} follows from the standard parametrix construction using the calculus of Fourier integral operators, as explained in more detail in \cite[Section~3]{JZ}. We mention here that the key point is to verify the given phase function parametrizes (in the sense of H\"ormander \cite{FIO1}) the propagating Lagrangians in \eqref{eq: propagating lagrangians}, at least when the exponential map is non-degenerate, hence $\mathscr{L}_{\pm}$ at fixed $s$ project down to $Y \times Y$ diffeomorphically.

Now we justify the claim that one can choose $a_{\pm,\mk{d}}(s,\hat{x},\hat{y},\rho)=0$ for $s > \pi-\frac{\delta_0}{2}$ under NREC. 
Denote the amplitude of the parametrix obtained from the parametrix construction by $\overline{a}_{\pm,\mk{d}}(s,\hat{x},\hat{y},\rho)$ (in fact, one needs to change the parametrization of the Lagrangian $\mathscr{L}_\pm$ in \eqref{eq: propagating lagrangians} when we encounter conjugate point pairs, but this part only contributes a residual term after inserting those $Q_j$-factors, for the same reason as we explain below). 
Now one can insert a factor $\chi_0(s)$ such that $\chi_0(s)=1$ on $[0,\pi - \delta_0]$ and is supported in $[0,\pi - \frac{\delta_0}{2}]$. That is, consider the same oscillatory integral, but with decomposing it into to parts $I_1,I_2$, with
\begin{equation}
\chi_0(s)\overline{a}_{\pm,\mk{d}}  \text{ and }  (1-\chi_0(s))\overline{a}_{\pm,\mk{d}}
\end{equation} 
being amplitudes respectively.
$\pi - \delta_0$  is larger than the diameter of $\mathrm{supp} Q_j$, the first part $I_1$ satisfies the property of parametrix (i.e., solve the wave equation up to a $C^\infty$-error) on 
\begin{equation}
\mathrm{supp} Q_j \times \mathrm{supp} Q_j.
\end{equation}
For the $I_2$-term, we have
\begin{equation} \label{eq: QjI2Qj-smooth}
Q_jI_2Q_j \in C^\infty([0,\pi] \times \CS \times \CS).
\end{equation}

This is because by the support of the fact $(1-\chi_0(s))$, we know $\WF(I_2)$ is over point pairs $(\hat{x},\hat{y})$ that are connected by geodesics of length $\ell \in [\pi-\delta_0,\pi]$. 
But no geodesic having such length connects point pairs in $\mathrm{supp} Q_j \times \mathrm{supp} Q_j$ by our choice after \eqref{eq: Uj condition} and Proposition~\ref{prop: NREC-thicken-diagonal}.
Thus after composing $Q_j$ (viewing it as a $0$-th order Fourier integral operator (FIO), its wavefront set is contained in the conormal bundle of the diagonal restricted over $\mathcal{U}_j$) from left and right, its wavefront set is empty (because the corresponding composition of canonical relations is empty) and we have \eqref{eq: QjI2Qj-smooth}. Consequently, this part can be collected into the $R_N$-term.

%%%%%%%%%%%%%%%%%%%%%%%%%%%%%%%%%%%
We can take $a_{-,\mk{d}}=a_{+,\mk{d}}$ by the following observation: $\mathscr{L}_\pm$ are actually the same Lagrangian (ignoring time related components) just with the flow with initial condition $(\hat{y},\mu_2)$ in $\mathscr{L}_+$ replaced by $(\hat{y},-\mu_2)$ in $\mathscr{L}_-$. So the corresponding amplitude $a_{-,\mk{d}}(s,\hat{x}, \hat{y};\xi)$ should satisfy
\begin{equation} \label{eq: a+,- relationship}
a_{-,\mk{d}}(s,\hat{x}, \hat{y};-\xi) = a_{+,\mk{d}}(s,\hat{x}, \hat{y};\xi),
\end{equation}
since they are obtained through solving the same transport equation along the same (lifted) geodesic. But in the step of the reduction to a function of $|\xi|$, the critical point for $a_{-,\mk{d}}(s,\hat{x}, \hat{y};-\xi)$ is now at $\xi = |\xi|(-1,0,...,0)$. So after the reduction which makes the symbol depends only on $|\xi|$, (see \cite[Section~3]{JZ} for more details). We still use $a_{\pm,\mk{d}}$ to denote the amplitude) and we obtain 
\begin{equation} \label{eq: a+,- relationship}
a_{-,\mk{d}}(s,\hat{x}, \hat{y};|\xi|) = a_{+,\mk{d}}(s,\hat{x}, \hat{y};|\xi|).
\end{equation}
\end{proof}

 \begin{lemma}[Hadamard parametrix I] 
\label{lemma: parametrix 1}
Let $\mk{D}(\hat{x},\hat{y})=\mk{D}_\pm(\hat{x},\hat{y})$ (see Remark~\ref{remark:D+- and D}) and $P=P_{{\A},a}$ be as in Theorem \ref{thm:dispersive}, then for each $Q_j$ with $1\leq j\leq \mathsf{F}$ given in \eqref{Id-p-Q}, 
and $\forall N>n+2$, the kernel of  $ Q_j\cos(s \sqrt{P}) Q_j$ can be written as
 \begin{equation}\label{KR}
\big[Q_j \cos(s \sqrt{P})Q_j \big](\hat{x},\hat{y})=K_N(s; \hat{x},\hat{y})+R_N(s; \hat{x},\hat{y}),
 \end{equation}
 where $R_N(s; \hat{x},\hat{y})\in C^{N-n-2} ([0,\pi]\times \CS \times \CS)$ and  
  \begin{equation}\label{KN1}
  \begin{split}
K_N(s; \hat{x},\hat{y})&=(2\pi)^{n-1} \sum_{\mk{d} \in \mk{D}(\hat{x},\hat{y})} \int_{\R^{n-1}} e^{i \mk{d}(\hat{x},\hat{y}){\bf 1}\cdot\xi}  a_{\mk{d}}(s, \hat{x},\hat{y}; |\xi|) \cos(s |\xi|) d\xi\\
&= \sum_{\mk{d} \in \mk{D}(\hat{x},\hat{y})} \sum_{\pm}\int_0^\infty b_{\pm}(\rho \mk{d}) e^{\pm i \rho \mk{d}} a_{\mk{d}}(s, \hat{x},\hat{y}; \rho) \cos(s \rho) \rho^{n-2} d\rho
\end{split}
 \end{equation}
 with ${\bf 1}=(1,0,\ldots,0)$ and $a_{\mk{d}}\in S^0$ satisfies \eqref{a} and $b_\pm$ obeys \eqref{b+-}.
In addition, under the NREC condition and let $\delta_0>0$ be as in \eqref{eq:NREC-quantitative}, then we can choose $a_{\mk{d}}(s,\hat{x},\hat{y},\rho)=0$ for $s>\pi-\frac{\delta_0}{2}$.
\end{lemma}

\begin{lemma}[Localized Poisson-wave operator]
\label{lemma: localized-poisson} 
With $\mk{D}_\pm(\hat{x},\hat{y}),P$ as above, the localized Schwartz kernel of $e^{(-s\pm i\pi)\sqrt{P}}$ can be decomposed as
\begin{equation}\label{eq:Poisson-kernel-1}
Q_je^{(-s\pm i\pi))\sqrt{P}}Q_j=\tilde{K}_{\pm, N}(s; \hat{x},\hat{y})+\tilde{R}_N(s; \hat{x},\hat{y}),
\end{equation}
where $\tilde{R}_N(s; \hat{x},\hat{y})\in C^{N-n-2} ([0,+\infty)\times Y\times Y)$ and  
\begin{equation}\label{eq:Poisson-kernel-2}
\begin{split}
\tilde{K}_{\pm, N}(s; \hat{x},\hat{y})&=(2\pi)^{n-1} \sum_{\mk{d} \in \mk{D}_\pm (\hat{x},\hat{y})} \int_{\R^{n-1}} e^{i \mk{d}(\hat{x},\hat{y}){\bf 1}\cdot\xi} \tilde{a}_{\pm,\mk{d}}(s, \hat{x},\hat{y}; |\xi|) e^{-(s\mp i\pi)|\xi| } d\xi\\
&=\sum_{\mk{d} \in \mk{D}_\pm(\hat{x},\hat{y})} \sum_{\varsigma = \pm}\int_0^\infty b_{\varsigma}(\rho \mk{d}) e^{\varsigma i \rho \mk{d}} \tilde{a}_{\pm,\mk{d}} (s, \hat{x},\hat{y}; \rho) e^{-(s \mp i\pi)\rho } \rho^{n-2} d\rho
\end{split}
 \end{equation}
where ${\bf 1}=(1,0,\ldots,0)$ and $\tilde{a}_{\pm,\mk{d}} \in S^{0}$, which means it satisfies 
\begin{equation} \label{a'}
|\partial^\alpha_{s,\hat{x},\hat{y}}\partial_\rho^k \tilde{a}_{\pm,\mk{d}}(s,\hat{x},\hat{y};\rho)|\leq C_{\alpha,k}(1+\rho)^{-k}.
\end{equation}

In addition, we can choose $\tilde{a}_{\pm,\mk{d}}$ such that the jet of $\tilde{a}_{\pm,\mk{d}}$ at $s = 0$ coincide with that of $a_\pm$ (given in \eqref{K-pm-N}) at $s=\pi$ in the sense that
\begin{equation} \label{eq:wave-poisson-jet-match}
(\partial_{s}^k \tilde{a}_{\pm,\mk{d}})(0,\hat{x},\hat{y};\rho)
= i^k(\partial_s^k a_{\pm,\mk{d}})(\pi,\hat{x},\hat{y};\rho).
\end{equation}
Under NREC, the localized version of it is residual in the sense that:
\begin{equation}
Q_j(\hat{x})\tilde{K}_{\pm, N} Q_j(\hat{y}) \in C^\infty([0,\infty) \times \CS \times \CS).
\end{equation}
\end{lemma}

\begin{proof}
The form of the oscillatory integral follows from the same argument as in \cite[Section~3]{JZ} and we discuss the smoothness under NREC.

The smoothness in $s$ follows directly if one shows that it is $C^\infty$ with respect to $\hat{x},\hat{y}$, since differentiating in $s$ only adds a $\sqrt{P}$ factor. 
For $s=0$, this follows from the property of
\begin{equation}
Q_je^{i\pi\sqrt{P}}Q_j,
\end{equation}
by the same discussion as above: $e^{i\pi\sqrt{P}}$ has canonical relation with points $\mk{d}(\hat{x},\hat{y})=\pi$ for $\mk{d} \in \mk{D}(\hat{x},\hat{y})$.
But there are no such point pairs according to our choice of $\mathcal{U}_j$ stated after \ref{eq: Uj condition}, in combination with Proposition~\ref{prop: NREC-thicken-diagonal}.
For $s>0$, the phase function has exponential decay in $\rho$ (or $|\xi|$) already, which means that one can view this as an oscillatory integral with amplitude in $S^{-\infty}$ and this gives the desired smoothness. 
\end{proof}

\subsection{The Littlewood-Paley theory}\label{subsec:LP}
In this subsection, we recall the the Bernstein inequalities and the square function inequalities associated with the Schr\"odinger operator $\LL_{{\A},a}$ which have been proved in \cite{JZ, JZ2} based on 
the Gaussian upper bounds of heat kernel in \cite{HZ23}.

We introduce $\varphi \in C_c^\infty(\mathbb{R}\setminus\{0\})$, with $0\leq\varphi\leq 1$, $\text{supp}\,\varphi\subset[1/2,2]$, and
\begin{equation}\label{LP-dp}
\sum_{j\in\Z}\varphi(2^{-j}\lambda)=1,\quad \varphi_{j}(\lambda):=\varphi(2^{-j}\lambda), \, j\in\Z,\quad \phi_0(\lambda):=\sum_{j\leq0}\varphi(2^{-j}\lambda).
\end{equation}
More precisely, we have the following propositions.

\begin{proposition}[Bernstein inequalities]\label{prop:Bern}
Let $\varphi(\lambda)$ be a $C^\infty_c$ bump function on $\R$  with support in $[\frac{1}{2},2]$ and let $\alpha$ and $p(\alpha)$ be given in \eqref{def:alpha} and \eqref{def:q-alpha} respectively, then it holds for any $f\in L^q(\cone)$ and $j\in\mathbb{Z}$
\begin{equation}\label{est:Bern}
\|\varphi(2^{-j}\sqrt{\LL_{{\A},a}})f\|_{L^p(\cone)}\lesssim2^{nj\big(\frac{1}{q}-\frac{1}{p}\big)}\|\varphi(2^{-j}\sqrt{\LL_{{\A},a}}) f\|_{L^q(\cone)}, \,
\end{equation}
provided $p'(\alpha)<q\leq p<p(\alpha)$.
In addition, if $\alpha\geq0$, the range can be extended to $1\leq q< p\leq +\infty$ including the endpoints.
\end{proposition}

\begin{proposition}[The square function inequality]\label{prop:squarefun} Let $\{\varphi_{j}\}_{j\in\mathbb Z}$ be a Littlewood-Paley sequence given by \eqref{LP-dp} and let $\alpha$ and $p(\alpha)$ be given in \eqref{def:alpha} and \eqref{def:q-alpha} respectively.
Then, for $p'(\alpha)<p<p(\alpha)$,
there exist constants $c_p$ and $C_p$ depending on $p$ such that
\begin{equation}\label{square}
c_p\|f\|_{L^p(\cone)}\leq
\Big\|\Big(\sum_{j\in\Z}|\varphi_j(\sqrt{\LL_{{\A},a}})f|^2\Big)^{\frac12}\Big\|_{L^p(\cone)}\leq
C_p\|f\|_{L^p(\cone)}.
\end{equation}

\end{proposition}

\section{The kernel of the spectral measure }\label{sec:kernel-spectral-measure}

In this section, we first construct the representation of the spectral measure associated with the operator $\mathcal{L}_{{\A},a}$ and then prove properties of the kernel. 

Before stating the results, we introduce a few notations.
For $r_1,r_2,s \in [0,+\infty)$, we define ${\bf m}_s$ and ${\bf n}_s\in\R^2$ by
\begin{equation}\label{bf-m}
{\bf m}_s=(r_1-r_2, \sqrt{2(1-\cos s)r_1r_2}),
\end{equation}
and
\begin{equation}\label{bf-n}
{\bf n}_s=\big(r_1+r_2, \sqrt{2(\cosh s-1)r_1r_2}\big).
\end{equation} 
And we further denote their lengths by
\begin{equation}\label{d-1}
d(s; r_1,r_2)=|{\bf m}_s|=\sqrt{r_1^2+r_2^2-2r_1r_2\cos s},\quad s\in[0,\pi],
\end{equation}
and
\begin{equation} \label{d-2}
\tilde{d}(s; r_1,r_2)=|{\bf n}_s|=\sqrt{r_1^2+r_2^2+2 r_1r_2\,\cosh s},\quad s\in [0,+\infty).
\end{equation}
Geometrically, one can think of $s$ as the propagating time along the flow with unit speed on the cross section $Y$. And the singularity is carried by the propagating Lagrangian defined by \eqref{eq: propagating lagrangians}, on which $s=d_h(\hat{x},\hat{y})$ (where $d_h$ is the distance function on $Y$), and $d(s; r_1,r_2)=|{\bf m}_s|$ becomes the distance between $x$ and $y$ in $X$.

\begin{theorem} [Spectral measure kernel] \label{thm:spect}
Suppose 
$x=(r_1, \hat{x}), \; y=(r_2,\hat{y}) \in  (0,+\infty)\times\CS,$
and ${\bf m}_s, {\bf n}_s, d(s; r_1,r_2), \tilde{d}(s; r_1,r_2)$ are as above, then the Schwartz kernel of the spectral measure  $dE_{\sqrt{\mathcal{L}_{{\A},a}}}(\lambda; x, y)$ associated with $\mathcal{L}_{{\A},a}$ can be written as:
\begin{equation}\label{spect}
 \begin{split}
& dE_{\sqrt{\mathcal{L}_{{\A},a}}}(\lambda; x, y)=G(\lambda; x, y)+D(\lambda; x , y)\\& =
\frac{\lambda}{4\pi^2} (r_1r_2)^{-\frac{n-2}2}\Big(\frac1{\pi}
\int_0^\pi \int_{\mathbb{S}^1} e^{-i\lambda {\bf m}_s \cdot\omega} d\sigma_\omega  \cos(s\sqrt{P})(\hat{x},\hat{y}) ds
 \\&\qquad\qquad\qquad
-\frac{\sin(\pi\sqrt{P})}{ \pi }\int_0^\infty \int_{\mathbb{S}^1} e^{-i\lambda {\bf n}_s \cdot\omega} d\sigma_\omega
\, e^{-s\sqrt{P}}(\hat{x},\hat{y}) ds\Big)\\
&=
\frac{\lambda}{4\pi^2} (r_1r_2)^{-\frac{n-2}2}\sum_{\pm}\Big(\frac1{\pi}
\int_0^\pi \mathfrak{a}_\pm(\lambda d(s; r_1,r_2))e^{\pm i\lambda d(s; r_1,r_2)} \cos(s\sqrt{P}) ds
 \\&\qquad\qquad\quad
 -\frac{\sin(\pi\sqrt{P})}{ \pi }\int_0^\infty \mathfrak{a}_\pm(\lambda \tilde{d}(s; r_1,r_2))e^{\pm i\lambda \tilde{d}(s; r_1,r_2)}
e^{-s\sqrt{P}} ds\Big),
\end{split}
\end{equation}
where  $\mathfrak{a}_\pm\in C^\infty([0,+\infty))$ satisfies
\begin{equation}\label{bean}
\begin{split}
| \partial_r^k \mathfrak{a}_\pm(r)|\leq C_k(1+r)^{-\frac{1}2-k},\quad k\geq 0,
\end{split}
\end{equation}
and $P=P_{{\A},a}$ and $\cos(s\sqrt{P})$ with $s\in[0,\pi]$ denotes the even-wave propagator and  $e^{-s\sqrt{P}}$ with $s\in[0,+\infty)$ is the Poisson operator on the unit sphere $\CS$. Furthermore, if either $\lambda r_1 \lesssim 1$ or $\lambda r_2 \lesssim 1$, then  the spectral measure  $dE_{\sqrt{\mathcal{L}_{{\A},a}}}(\lambda; x, y)$ can be written as
\begin{equation}\label{Spect<1}
 \begin{split}
dE_{\sqrt{\mathcal{L}_{{\A},a}}}(\lambda; x, y)= \lambda^{n-1}\Big(\sum_{\pm}A_{\pm}(\lambda; x, y) e^{\pm i\lambda \max\{r_1, r_2\}}
+B(\lambda; x, y)\Big),
\end{split}
\end{equation}
where
\begin{equation}\label{bean-A}\begin{split}
\big|\partial_\lambda^\alpha A_\pm(\lambda,x, y) \big|\leq C_\alpha \lambda^{-\alpha} &(1+\lambda \max\{r_1, r_2\})^{-\frac{n-1}2}\\
&\times \begin{cases}0,\quad &\lambda r_1, \lambda r_2\lesssim 1;\\
(\lambda r_1)^{\nu_0-\frac{n-2}2},\quad &\lambda r_1\lesssim 1\ll   \lambda r_2;\\
(\lambda r_2)^{\nu_0-\frac{n-2}2},\quad &\lambda r_2\lesssim 1\ll   \lambda r_1;
 \end{cases},
\end{split}\end{equation}
\begin{equation}\label{bean-B}\begin{split}
\big|\partial_\lambda^\alpha B(\lambda,x, y) \big|\leq C_\alpha \lambda^{-\alpha} &(1+\lambda \max\{r_1, r_2\})^{-\frac{n+1}2}\\
&\times \begin{cases}(\lambda^2r_1r_2)^{\nu_0-\frac{n-2}2},\quad &\lambda r_1, \lambda r_2\lesssim 1;\\
(\lambda r_1)^{\nu_0-\frac{n-2}2},\quad &\lambda r_1\lesssim 1\ll   \lambda r_2;\\
(\lambda r_2)^{\nu_0-\frac{n-2}2},\quad &\lambda r_2\lesssim 1\ll   \lambda r_1.
 \end{cases}
\end{split}\end{equation}
Here $C$, $C_\alpha$ are constants independent of $\lambda$ and $x, y$. 
\end{theorem} 

The proof of this theorem is divided into two subsections. We will discuss each region in separately, and they are robust under enlarging those regions by changing the implicit constants in inequalities above and can be glued together to obtain expressions above by using a partition of the unity.

\subsection{The abstract construction of the spectral measure}
In this subsection, we prove \eqref{spect}. Our starting point is the propagator of Schr\"odinger equation constructed in \cite{JZ2, JZ}, inspired by Cheeger-Taylor \cite{CT1,CT2}.
The kernel of Schr\"odinger propagator can be written as
\begin{equation}\label{S-kernel} 
\begin{split}
e^{-it\LL_{{\A},a}}(x,y)&=e^{-it\LL_{{\A},a}}(r_1, \hat{x}, r_2, \hat{y})\\
 &=\big(r_1 r_2\big)^{-\frac{n-2}2}\frac{e^{-\frac{r_1^2+r_2^2}{4it}}}{2it}
  \Big(\frac1{\pi}\int_0^\pi e^{\frac{r_1r_2}{2it} \cos(s)} \cos(s\sqrt{P})(\hat{x}, \hat{y}) ds\\
  &-\frac{\sin(\pi\sqrt{P})}{\pi}\int_0^\infty e^{-\frac{r_1r_2}{2it} \cosh s} e^{-s\sqrt{P}}(\hat{x}, \hat{y}) ds\Big),
\end{split}
\end{equation}
where $P=P_{{\A},a}=L_{{\A},a}+(n-2)^2/4$ with $L_{{\A},a}$ in \eqref{L-angle}.

We next construct the resolvent kernel from the above Schr\"odinger kernel. More precisely, we prove:
\begin{theorem}[Resolvent kernel]\label{thm:res}
Let $x=(r_1, \hat{x})$ and $y=(r_2,\hat{y})$ be in $(0,+\infty)\times\CS$ and let $d(s, r_1, r_2)$ and $\tilde{d}(s, r_1, r_2)$
be in \eqref{d-1} and \eqref{d-2} respectively. 
Then the Schwartz kernel of the resolvent $$(\LL_{{\A},a}-(\lambda^2\pm i0))^{-1}:=\lim_{\epsilon\searrow 0}(\LL_{{\A},a}-(\lambda^2\pm i\epsilon))^{-1}$$ can be written as
\begin{equation}\label{r-g}
\begin{split}
\frac{\pm i}{8\pi} \big(r_1 r_2\big)^{-\frac{n-2}2}\Big(\frac1{\pi}
 \int_0^{\pi} H_0^{\pm}\big(\lambda d(s, r_1, r_2)\big) \cos(s\sqrt{P})(\hat{x},\hat{y}) ds
\end{split}
\end{equation}
\begin{equation}\label{r-d}
\begin{split}
\\&
-\frac{\sin(\pi\sqrt{P})}{\pi}
 \int_0^\infty H_0^{\pm}\big(\lambda \tilde{d}(s, r_1, r_2)\big)\,e^{-s\sqrt{P}}(\hat{x},\hat{y}) ds\Big),
\end{split}
\end{equation}
where $P=P_{{\A},a}$
and
 $H_0^\pm$ are the Hankel functions of order zero with $H_0^-=\overline{H_0^+}$ and, for $y>0$
\begin{equation} \label{eq:hankel-zero-order}
H_0^+(y)= C\times \begin{cases} y^{-\frac12} e^{i(y+\frac\pi 4)}\Big(1+O(y^{-1})\Big),\qquad y\to+\infty\\
\log(\frac{y}{2}) +O(1) ,\qquad \quad y\to 0.
\end{cases}
\end{equation}
\end{theorem} 

\begin{proof}
We first note that
when $\sigma\in \{\sigma\in\C: \Im(\sigma)>0\}$, then
$$(s-\sigma)^{-1}=\frac1{i}\int_0^\infty e^{-ist} e^{i\sigma t}dt,\quad \forall s\in\R,$$
thus we obtain, 
\begin{equation}\label{res+}
\begin{split}
(\LL_{{\A},a}-(\lambda^2+i0))^{-1}&=\frac1{i}\lim_{\epsilon\to 0^+}\int_0^\infty e^{-it\LL_{{\A},a}} e^{it\sigma_\epsilon}dt,
\end{split}
\end{equation}
where $\sigma_\epsilon=\lambda^2+i\epsilon$ with $\epsilon>0$ .
From \eqref{S-kernel}, we obtain
\begin{equation}\label{res+}
\begin{split}
&(\LL_{{\A},a}-(\lambda^2+i0))^{-1}\\
&=\frac1{i}\lim_{\epsilon\to 0^+}\int_0^\infty 
\big(r_1 r_2\big)^{-\frac{n-2}2}\frac{e^{-\frac{r_1^2+r_2^2}{4it}}}{2it}
  \Big(\frac1{\pi}\int_0^\pi e^{\frac{r_1r_2}{2it} \cos s} \cos(s\sqrt{P})(\hat{x}, \hat{y}) ds\\
  &+\frac{i}{2\pi}\int_0^\infty e^{-\frac{r_1r_2}{2it} \cosh s} \Big(e^{(-s+ i\pi)\sqrt{P}}-e^{(-s- i\pi)\sqrt{P}}\Big)(\hat{x}, \hat{y}) ds\Big)
 e^{it(\lambda^2+i\epsilon)}dt.
\end{split}
\end{equation}
To prove Theorem \ref{thm:res}, we first need the following lemma.
\begin{lemma} For fixed  $\sigma_\epsilon=\lambda^2+i\epsilon$ with $\epsilon>0$ 
\begin{equation}\label{equ:m}
\begin{split}
\int_0^\infty \frac{e^{-\frac{r_1^2+r_2^2}{4it}} }{it} e^{\frac{r_1r_2}{2it} \cos s} e^{i\sigma_\epsilon t}dt=\frac{i}\pi\int_{\R^2} \frac{e^{-i{\bf m}_s\cdot {\xi}}}{|\xi|^2-\sigma_\epsilon} \, d{ \xi},\quad s\in[0,\pi]
\end{split}
\end{equation}
and
\begin{equation}\label{equ:n}
\begin{split}
\int_0^\infty \frac{e^{-\frac{r_1^2+r_2^2}{4it}} }{it} e^{-\frac{r_1r_2}{2it}\cosh s} e^{i\sigma_\epsilon t}dt=\frac i \pi\int_{\R^2}\frac{e^{-i{\bf n}_s\cdot {\xi}}}{|\xi|^2-\sigma_\epsilon} \, d{ \xi},\quad s\in[0,+\infty),
\end{split}
\end{equation}
where $\xi=(\xi_1,\xi_2)\in\R^2$ and ${\bf m}_s, {\bf n}_s\in\R^2$ are in \eqref{bf-m} and \eqref{bf-n}.
\end{lemma}

\begin{proof} We first prove \eqref{equ:m}. Let
\begin{equation}\label{s-r}
s_1=r_1-r_2, \quad s_2=\sqrt{2r_1r_2}.
\end{equation}
 We write
\begin{equation}
\begin{split}
&\frac{e^{-\frac{r_1^2+r_2^2}{4it}} }{it} e^{\frac{r_1r_2}{2it}\cos s} =
\frac{e^{-\frac{(r_1-r_2)^2}{4it}} }{\sqrt{it}} \frac{e^{-\frac{r_1r_2}{2it}(1-\cos s)}}{\sqrt{it}}\\
&=\frac{e^{-\frac{s_1^2}{4it}} }{\sqrt{it}} \frac{e^{-\frac{s_2^2}{4it}(1-\cos s)}}{\sqrt{it}}.
\end{split}
\end{equation} 
By using the formula
\begin{equation}\label{F-Gaussian}
\int_{-\infty}^\infty e^{-it\eta^2} e^{-i \eta r} d\eta=\sqrt{\frac{\pi}{it}}e^{-\frac{ r^2}{4it}},
\end{equation}
we have 
\begin{equation}
\begin{split}
&\frac{e^{-\frac{r_1^2+r_2^2}{4it}} }{it} e^{\frac{r_1r_2}{2it}\cos s} \\
&=\frac1\pi\int_{-\infty}^\infty e^{-it\xi_1^2} e^{-is_1\xi_1} d\xi_1 \int_{-\infty}^\infty e^{-it\xi_2^2} e^{-i\sqrt{1-\cos s}s_2\xi_2} d\xi_2\\
&=\frac1\pi\int_{\R^2} e^{-it(\xi_1^2+\xi_2^2)} e^{-i{\bf m}_s\cdot(\xi_1,\xi_2)} d\xi_1 d\xi_2,
\end{split}
\end{equation} 
where ${\bf m}_s$ is given by \eqref{bf-m}. Let $\xi=(\xi_1,\xi_2)$, then 
\begin{equation}
\begin{split}
&\int_0^\infty \frac{e^{-\frac{r_1^2+r_2^2}{4it}} }{it} e^{\frac{r_1r_2}{2it} \cos s} e^{i\sigma_\epsilon t}dt\\
&=\frac1 \pi\int_{\R^2} e^{-i{\bf m}_s\cdot\xi} \int_0^\infty e^{-it|\xi|^2} e^{i\sigma_\epsilon t}dt\, d\xi\\
&=\frac i \pi\int_{\R^2} \frac{e^{-i{\bf m}_s\cdot\xi}}{|\xi|^2-\sigma_\epsilon} \, d\xi
\end{split}
\end{equation}
which shows \eqref{equ:m}. Next we prove \eqref{equ:n}. We consider 
\begin{equation}
\begin{split}
\int_0^\infty \frac{e^{-\frac{r_1^2+r_2^2}{4it}} }{it} e^{-\frac{r_1r_2}{2it}\cosh s} e^{i\sigma_\epsilon t}dt.
\end{split}
\end{equation}
Instead of \eqref{s-r}, using the change of variables $$s_1=r_1+r_2, \quad s_2=\sqrt{2r_1r_2},$$ we similarly write
\begin{equation}
\begin{split}
\frac{e^{-\frac{r_1^2+r_2^2}{4it}} }{it} e^{-\frac{r_1r_2}{2it}\cosh s} &=
\frac{e^{-\frac{(r_1+r_2)^2}{4it}} }{\sqrt{it}} \frac{e^{-\frac{r_1r_2}{2it}(\cosh s-1)}}{\sqrt{it}}\\
&=\frac{e^{-\frac{s_1^2}{4it}} }{\sqrt{it}} \frac{e^{-\frac{s_2^2}{4it}(\cosh s-1)}}{\sqrt{it}}.
\end{split}
\end{equation} 
By \eqref{F-Gaussian}, we obtain that 
\begin{equation}
\begin{split}
&\frac{e^{-\frac{s_1^2}{4it}} }{\sqrt{it}} \frac{e^{-\frac{s_2^2}{4it}(\cosh s-1)}}{\sqrt{it}}\\
&=\frac 1\pi\int_{-\infty}^\infty e^{-it\xi_1^2}e^{-is_1\xi_1} d\xi_1 \int_{-\infty}^\infty e^{-it\xi_2^2}e^{-i\sqrt{\cosh s-1}s_2\xi_2} d\xi_2\\
&=\frac 1\pi\int_{-\infty}^\infty \int_{-\infty}^\infty e^{-it(\xi_1^2+\xi_2^2)}e^{-i(s_1\xi_1+\sqrt{\cosh s-1}s_2\xi_2)} d\xi_1 d\xi_2.
\end{split}
\end{equation} 
Similarly as above, we show
\begin{equation}
\begin{split}
\int_0^\infty \frac{e^{-\frac{r_1^2+r_2^2}{4it}} }{it} e^{-\frac{r_1r_2}{2it}\cosh s} e^{i\sigma_\epsilon t}dt
=\frac {i}\pi\int_{\R^2} \frac{e^{-i{\bf n}_s\cdot\xi}}{|\xi|^2-\sigma_\epsilon} d\xi
\end{split}
\end{equation}
where ${\bf n}_s=(n_1,n_2)=(r_1+r_2, \sqrt{2(\cosh s-1)r_1r_2})$.

\end{proof}

From \eqref{equ:m} and \eqref{equ:n}, it follows 
\begin{equation}\label{equ:res+}
\begin{split}
&(\LL_{{\A},a}-(\lambda^2+i0))^{-1}\\&=
\frac1{2\pi}\big(r_1 r_2\big)^{-\frac{n-2}2}
 \int_0^\infty \lim_{\epsilon\to 0^+}\frac1\pi\int_{\R^2} \frac{e^{-i{\bf m}_s\cdot {\xi}}}{|\xi|^2-(\lambda^2+i\epsilon)} \, d{ \xi} \cos(s\sqrt{P})(\hat{x}, \hat{y}) ds\\&
-\frac{\sin(\pi\sqrt{P})}{2\pi } \big(r_1 r_2\big)^{-\frac{n-2}2}
 \int_0^\infty \lim_{\epsilon\to 0^+}\frac1\pi\int_{\R^2} \frac{e^{-i{\bf n}_s\cdot {\xi}}}{|\xi|^2-(\lambda^2+i\epsilon)} \, d{ \xi} \,e^{-s\sqrt{P}}(\hat{x}, \hat{y}) ds.
 \end{split}
\end{equation}

Next, we need the following lemma about the resolvent for Laplacian in $\R^2$ which is well-known, for example, see \cite[Chapter 3.4]{CK1} or \cite[Eq. (5.16.3)]{Le}.   
\begin{lemma} Let  $\sigma_\epsilon=\lambda^2+i\epsilon$ with $\epsilon>0$. Then
\begin{equation}\label{equ:res-free}
\begin{split}
\lim_{\epsilon\to0^+}\frac1\pi\int_{\R^2} \frac{e^{-ix\cdot {\xi}}}{|\xi|^2-\sigma_\epsilon} \, d{ \xi}=\frac i {4\pi} H_0^+(\lambda |x|),\quad x\in\R^2\setminus\{0\},
\end{split}
\end{equation}
where $H_0^+$ is the Hankel function of order zero as in \eqref{eq:hankel-zero-order}.

\end{lemma}

From \eqref{bf-m} and \eqref{bf-n}, we have that
$$|{\bf m}_s|=\sqrt{r_1^2+r_2^2-2\cos s r_1r_2}=d(s, r_1,r_2)$$
and
$$|{\bf n}_s|=\sqrt{r_1^2+r_2^2+2\cosh s\, r_1r_2}=\tilde{d}(s, r_1,r_2).$$
Combining \eqref{equ:res+} and \eqref{equ:res-free}, we have
\begin{equation} \label{eq: outgoing-resolvent-1}
\begin{split}
&(\LL_{{\A},a}-(\lambda^2+i0))^{-1}\\&=
\frac{i}{8\pi^2} \big(r_1 r_2\big)^{-\frac{n-2}2}
 \int_0^\infty H_0^{+}\big(\lambda |{\bf m}_s|\big) \cos(s\sqrt{P})(\hat{x}, \hat{y}) ds\\&
-\frac{i}{8\pi^2}\big(r_1 r_2\big)^{-\frac{n-2}2}
 \int_0^\infty H_0^{+}\big(\lambda |{\bf n}_s|\big)\,
  \sin(\pi \sqrt{P})e^{-s\sqrt{P}}(\hat{x}, \hat{y})ds,
\end{split}
\end{equation}
which proves Theorem \ref{thm:res} for $(\LL_{{\A},a}-(\lambda^2+i0))^{-1}$.
Next we consider the incoming resolvent $(\LL_{{\A},a}-(\lambda^2-i0))^{-1}$. For an operator $\mathfrak{P}$ with Schwartz kernel $K_{\mathfrak{P}}$, we denote the operator acting with kernel $K_{\mathfrak{P}}$ by $\overline{\mathfrak{P}}$. Then we know
\begin{equation}\label{out-inc}
\begin{split}
(\LL_{{\A},a}-(\lambda^2-i0))^{-1}=\overline{(\overline{\LL_{{\A},a}}-(\lambda^2+i0))^{-1}}.
\end{split}
\end{equation}
Using the same derivation for \eqref{equ:res+}, we have
\begin{equation}\label{eq: outgoing-resolvent-2}
\begin{split}
&(\overline{\LL_{{\A},a}}-(\lambda^2+i0))^{-1}\\&=
\frac1{2\pi}\big(r_1 r_2\big)^{-\frac{n-2}2}
 \int_0^\infty \lim_{\epsilon\to 0^+}\frac1\pi\int_{\R^2} \frac{e^{-i{\bf m}_s\cdot {\xi}}}{|\xi|^2-(\lambda^2+i\epsilon)} \, d{ \xi} \cos(s\sqrt{\overline{P}})(\hat{x}, \hat{y}) ds\\&
-\frac{\sin(\pi\sqrt{\overline{P}})}{2\pi } \big(r_1 r_2\big)^{-\frac{n-2}2}
 \int_0^\infty \lim_{\epsilon\to 0^+}\frac1\pi\int_{\R^2} \frac{e^{-i{\bf n}_s\cdot {\xi}}}{|\xi|^2-(\lambda^2+i\epsilon)} \, d{ \xi} \,e^{-s\sqrt{\overline{P}}}(\hat{x}, \hat{y}) ds.
 \end{split}
\end{equation} 

Notice that for any holomorphic $F$ that take real values on $\R$, we have
\begin{equation}
\overline{F(\overline{z})}=F(z), z \in \mathbb{C},
\end{equation}
which extends to the operator case using the functional calculus, with $F(z) = \cos(s z)$, and $e^{-sz}\sin(\pi z)$ and $z = \sqrt{P}$, where $P=-\Delta_{h}+\big(|{\A}(\hat{x})|^2+i\,\mathrm{div}_{h}{\A}(\hat{x})+a(\hat{x})\big)+2i {\A}(\hat{x})\cdot\nabla_{h}+(n-2)^2/4$. 
Therefore taking complex conjugate (notice that we can change the sign of phases $i{\bf n}_s\cdot {\xi}$ and $i{\bf m}_s\cdot {\xi}$ after conjugating to the form we want by using the symmetry in $\xi$) on both sides of \eqref{eq: outgoing-resolvent-2} completes the proof of Theorem \ref{thm:res}.
\end{proof}

Finally, we prove the representation of the spectral measure \eqref{spect}. According to Stone’s formula, the spectral measure is related to the resolvent by
 \begin{equation}\label{stone}
 dE_{\sqrt{\LL_{{\A},a}}}(\lambda)=\frac{d}{d\lambda}dE_{\sqrt{\LL_{{\A},a}}}(\lambda)\,d\lambda=\frac{\lambda}{\pi i}\big(R(\lambda+i0)-R(\lambda-i0)\big)\, d\lambda
 \end{equation}
 where the resolvent $$R(\lambda\pm i0)=\lim_{\epsilon\searrow 0}(\LL_{{\A},a}-(\lambda^2\pm i\epsilon))^{-1}.$$
 
From \eqref{stone}, \eqref{out-inc} and \eqref{equ:res+}, we have
 \begin{equation*}
 \begin{split}
 &(r_1r_2)^{\frac{n-2}2}dE_{\sqrt{\LL_{{\A},a}}}(\lambda;x,y)\\
& =
\frac1{8\pi^2} \frac{\lambda}{\pi i}
\int_0^\infty \int_{\R^2} e^{-i{\bf m}_s\cdot {\xi}}\Big(\frac1{|\xi|^2-(\lambda^2+i0)}-\frac1{|\xi|^2-(\lambda^2-i0)}\Big) \, d{ \xi}
\cos(s\sqrt{P})(\hat{x}, \hat{y}) ds
 \\&
-\frac{1}{8\pi^2} \frac{\lambda}{\pi i}
 \int_0^\infty \int_{\R^2} e^{-i{\bf n}_s\cdot {\xi}}\Big(\frac1{|\xi|^2-(\lambda^2+i0)}-\frac1{|\xi|^2-(\lambda^2-i0)}\Big) \, d{ \xi}
\Big(e^{-s\sqrt{P}}\sin(\pi \sqrt{P})\Big)(\hat{x}, \hat{y}) ds.
\end{split}
\end{equation*}
On the one hand, we note that
\begin{equation}\label{id-spect}
\begin{split}
&\lim_{\epsilon\to 0^+}\frac{\lambda}{\pi i}\int_{\R^2} e^{-ix\cdot\xi}\Big(\frac{1}{|\xi|^2-(\lambda^2+i\epsilon)}-\frac{1}{|\xi|^2-(\lambda^2-i\epsilon)}\Big) d\xi\\
&=\lim_{\epsilon\to 0^+} 2 \frac{\lambda}{\pi }\int_{\R^2} e^{-ix\cdot\xi}\Im\Big(\frac{1}{|\xi|^2-(\lambda^2+i\epsilon)}\Big)d\xi\\
&=2 \lim_{\epsilon\to 0^+} \frac{\lambda}{\pi }\int_{0}^\infty \frac{\epsilon}{(\rho^2-\lambda^2)^2+\epsilon^2} \int_{|\omega|=1} e^{-i\rho x\cdot\omega} d\sigma_\omega  \, \rho d\rho\\
&= 2\lambda \int_{|\omega|=1} e^{-i\lambda x\cdot\omega} d\sigma_\omega, 
\end{split}
\end{equation}
where we used the fact the Poisson kernel is an approximation to the identity, which implies that, for any continuous bounded function $m(x)$:
\begin{equation}
\begin{split}
m(x) &=\lim_{\epsilon\to 0^+}\frac1\pi \int_{\R} \frac{\epsilon}{(x-y)^2+\epsilon^2} m(y)dy 
\\ &=\lim_{\epsilon\to 0^+}\frac1\pi \int_{\R} \Im\big(\frac{1}{x-(y+i\epsilon)}\big) m(y)dy.
\end{split}
\end{equation}

Therefore we obtain that
\begin{equation*}
 \begin{split}
& dE_{\sqrt{\LL_{{\A},a}}}(\lambda;x,y)\\& =
\frac{\lambda}{4\pi^2} (r_1r_2)^{-\frac{n-2}2}\Big(\frac1{\pi}
\int_0^\pi \int_{\mathbb{S}^1}e^{-i\lambda {\bf m}_s \cdot\omega} d\sigma_\omega  \cos(s\sqrt{P})(\hat{x},\hat{y}) ds
 \\&
-\frac{\sin(\pi\sqrt{P})}{ \pi }\int_0^\infty \int_{\mathbb{S}^1} e^{-i\lambda {\bf n}_s \cdot\omega} d\sigma_\omega
e^{-s\sqrt{P}}(\hat{x},\hat{y}) ds\Big).
\end{split}
\end{equation*}
On the other hand, for example \cite[Theorem 1.2.1]{sogge}, we also note that
\begin{equation}\label{a-pm1}
\begin{split}
\int_{\mathbb{S}^{1}} e^{-i x\cdot\omega} d\sigma(\omega)=\sum_{\pm}  \mathfrak{a}_\pm(|x|) e^{\pm i|x|}
\end{split}
\end{equation}
where $\mathfrak{a}_\pm$ satisfies \eqref{bean}.
The expression above can be simplied to be
\begin{equation*}
 \begin{split}
& dE_{\sqrt{\LL_{{\A},a}}}(\lambda;x,y)\\& =
\frac{\lambda}{4\pi^2} (r_1r_2)^{-\frac{n-2}2}\sum_{\pm}\Big(\frac1{\pi}
\int_0^\pi \mathfrak{a}_\pm(\lambda |{\bf{m}}_s|)e^{\pm i\lambda |{\bf m}_s|} \cos(s\sqrt{P}) ds
 \\&
-\frac{\sin(\pi\sqrt{P})}{ \pi }\int_0^\infty \mathfrak{a}_\pm(\lambda |{\bf{n}}_s|)e^{\pm i\lambda |{\bf n}_s|}
e^{-s\sqrt{P}} ds\Big).
\end{split}
\end{equation*}
Recalling $|{\bf m}_s|=d(s, r_1, r_2)$ and $|{\bf n}_s|=\tilde{d}(s, r_1, r_2)$ in \eqref{d-1} and \eqref{d-2}, we have proved \eqref{spect} of Theorem \ref{thm:spect}.\vspace{0.2cm}

\subsection{The proof of \eqref{Spect<1} } In this subsection, we study the property of the spectral measure when either  $\lambda r_1\lesssim 1$ or $\lambda r_2\lesssim 1$ holds.
Our starting point is the following lemma.

\begin{lemma}\label{lem:spect} Let $\nu_j^2$ be the eigenvalues of the positive operator
$P=P_{{\A}, a}$ and let
$\psi_{j}(\hat{x})$ be the corresponding $L^2$-normalized eigenfunction, then the spectral measure can be written as
\begin{equation}\label{spect1}
\begin{split}
dE_{\sqrt{\LL_{{\A},a}}}(\lambda;x,y)=\frac{\pi}2 \lambda
(r_1r_2)^{-\frac{n-2}2}\sum_{j}\psi_{j}(\hat{x})\overline{\psi_{j}(\hat{y})}J_{\nu_j}(\lambda r_1)J_{\nu_j}(\lambda r_2),
\end{split}
\end{equation}
where $J_{\nu}$ is the Bessel function of order $\nu$.
\end{lemma}

This lemma can be proven by a minor modification of the construction of
the resolvent \cite[Section 4]{HL} and the spectral measure \cite[Lemma 3.1]{ZZ1} on a metric cone. Here we sketch the proof. 
\begin{proof}
Recall $L_{{\A},a}$ in \eqref{L-angle}, we write
\begin{equation}\label{LvPb}
\begin{split}
\LL_{{\A},a}&=-\partial_r^2-\frac{n-1}r\partial_r+\frac{L_{{\A},a}}{r^2}\\&=r^{-1-\frac
n2}\Big(-(r\partial_r)^2+P\Big)r^{\frac
n2-1}=r^{-1-\frac n2}P_br^{\frac n2-1}
\end{split}
\end{equation}
where $P_b=-(r\partial_r)^2+P$. 
Recalling that $\nu_j^2=\mu_j+(n-2)^2/4$ in \eqref{equ:eig-Aa} are the eigenvalues of the positive operator $P=P_{{\A},a}$ corresponding to the $L^2$-normalized eigenfunction $\psi_{j}(\hat{x})$:
\begin{equation*}
P\psi_{j}=\nu_j^2\psi_{j}.
\end{equation*}
Let $\Pi_j$ be projection onto the eigenspace spanned by $\psi_{j}$.  Then for $k \in \R$ (we will extend to larger range later), we have the decomposition
\begin{equation*}
P_b+k^2r^2=\sum_{j}\big(-(r\partial_r)^2+k^2r^2+\nu_j^2\big)\Pi_j,
\end{equation*}
which gives
\begin{equation} \label{eq: Pb+k2r2,inverse}
\big(P_b+k^2r^2\big)^{-1}=\sum_{j}\Pi_j\big(-(r\partial_r)^2+k^2r^2+\nu_j^2\big)^{-1}.
\end{equation}
Let $T_j=-(r\partial_r)^2+k^2r^2+\nu_j^2$ and let $T_j^{-1}(r_1,r_2)$ be
the kernel of the inverse $T^{-1}_j$. Therefore, using the same argument as in \cite[Section~4.3]{HL}, which in turn is using \cite[Section~3.4]{GH1}, we have
\begin{equation}
T_j^{-1}(r_1,r_2)=\begin{cases}
I_{\nu_j}(kr_1)K_{\nu_j}(kr_2)\Big|\frac{dr_1}{r_1}\frac{dr_2}{r_2}\Big|^{\frac12},\quad
r_1<r_2;\\
K_{\nu_j}(kr_1)I_{\nu_j}(kr_2)\Big|\frac{dr_1}{r_1}\frac{dr_2}{r_2}\Big|^{\frac12},\quad
r_1>r_2. \end{cases}
\end{equation}
where the modified Bessel functions
\begin{equation}
\begin{split}
I_{\nu_j}(r)&=\frac{2^{-\nu_j}r^{\nu_j}}{\sqrt{\pi}\Gamma(\nu_j+\frac12)}\int_{-1}^1(1-t^2)^{\nu_j-\frac12}e^{-rt}dt,\\
K_{\nu_j}(r)&=\frac{\sqrt{\pi}2^{-\nu_j}r^{\nu_j}}{\Gamma(\nu_j+\frac12)}\int_{1}^\infty(t^2-1)^{\nu_j-\frac12}e^{-rt}dt.
\end{split}
\end{equation}
Combining with \eqref{eq: Pb+k2r2,inverse}, we have
\begin{equation}
\big(P_b+k^2r^2\big)^{-1}=\begin{cases}
\sum_{j}\psi_{j}(\hat{x})\overline{\psi_{j}(\hat{y})}I_{\nu_j}(kr_1)K_{\nu_j}(kr_2)\Big|\frac{dr_1}{r_1}\frac{dr_2}{r_2} d\hat{x} d\hat{y}\Big|^{\frac12},\quad
r_1<r_2;\\
\sum_{j}\psi_{j}(\hat{x})\overline{\psi_{j}(\hat{y})}
K_{\nu_j}(kr_1)I_{\nu_j}(kr_2)\Big|\frac{dr_1}{r_1}\frac{dr_2}{r_2} d\hat{x} d\hat{y}\Big|^{\frac12},\quad
r_1>r_2.
\end{cases}
\end{equation}
This formula continues analytically to the imaginary axis, by
setting $k=-i\lambda$ and using the following formulae
$$I_{\nu}(-iz)=e^{-\nu\pi i/2}J_{\nu}(z), \quad K_{\nu}(-iz)=\frac{\pi i}2e^{\nu\pi
i/2}H^{(1)}_{\nu}(z),$$ we obtain
\begin{equation*}
\big(\LL_{{\A},a}-(\lambda^2+i0)\big)^{-1}= \frac{\pi
ir_1r_2}2\begin{cases}\sum_{j}\psi_{j}(\hat{x})\overline{\psi_{j}(\hat{y})}J_{\nu_j}(r_1\lambda)H^{(1)}_{\nu_j}(r_2\lambda)\Big|\frac{dr_1}{r_1}\frac{dr_2}{r_2} d\hat{x} d\hat{y}\Big|^{\frac12},\quad
r_1<r_2;\\
\sum_{j}\psi_{j}(\hat{x})\overline{\psi_{j}(\hat{y})}J_{\nu_j}(r_2\lambda)H^{(1)}_{\nu_j}(r_1\lambda)\Big|\frac{dr_1}{r_1}\frac{dr_2}{r_2} d\hat{x} d\hat{y}\Big|^{\frac12},\quad
r_1>r_2,
\end{cases}
\end{equation*}
where $J_\nu$ and $H^{(1)}_{\nu}$ are Bessel and Hankel functions of the first kind, see e.g. \cite{Watson}.  Let $H$ be the Heaviside function, since $\mathrm{Im}(i
H_{\nu}^{(1)})(r)=J_\nu(r)$, we have
\begin{equation*}
\begin{split}
&dE_{\sqrt{\LL_{{\A},a}}}(\lambda; x,y)=\lambda
\mathrm{Im}\Big(\big(\LL_{{\A},a}-(\lambda^2+i0)\big)^{-1}\Big)\\&= \frac{\pi \lambda
r_1r_2}2\sum_{j}\psi_{j}(\hat{x})\overline{\psi_{j}(\hat{y})}\\
&\times\Big(J_{\nu_j}(\lambda r_1)iH^{(1)}_{\nu_j}(\lambda r_2)H(r_2-r_1)
+J_{\nu_j}(\lambda r_2)iH^{(1)}_{\nu_j}(\lambda r_1)H(r_1-r_2)\Big)\Big|\frac{dr_1}{r_1}\frac{dr_2}{r_2} d\hat{x} d\hat{y}\Big|^{\frac12}\\&=\frac{\pi \lambda
r_1r_2}2\sum_{j}\psi_{j}(\hat{x})\overline{\psi_{j}(\hat{y})}J_{\nu_j}(\lambda r_1)J_{\nu_j}(\lambda r_2)\Big|\frac{dr_1}{r_1}\frac{dr_2}{r_2} d\hat{x} d\hat{y}\Big|^{\frac12}.
\end{split}
\end{equation*}
Now we return to half-density $|dx|^{1/2}$ and obtain
\begin{equation*}
\begin{split}
&dE_{\sqrt{\LL_{{\A},a}}}(\lambda; x,y)=\frac{\pi \lambda}2
(r_1r_2)^{-\frac{n-2}2}\sum_{j}\psi_{j}(\hat{x})\overline{\psi_{j}(\hat{y})}J_{\nu_j}(\lambda r_1)J_{\nu_j}(\lambda r_2)\big|dx dy\big|^{\frac12},
\end{split}
\end{equation*}
which shows \eqref{spect1}.
\end{proof}

\begin{proof}[The proof of \eqref{Spect<1}]
Without loss of generality, we may assume that $\lambda r_1\lesssim 1$ always holds. Recall that if $\mathrm{Re} \nu>-1/2$, then one has the recursion relation (see for example \cite[Page 45, 3.2-(4)]{Watson}): 
\begin{equation} \label{eq:bessel-derivative-recursion}
\frac{d}{dt} \left(t^{-\nu}J_{\nu}(t)\right)=-t^{-\nu} J_{\nu+1}(t).
\end{equation}
Using this, we have 
\begin{equation} \label{eq:bessel-lambda-derivative}
\frac{d}{d\lambda} \left( (\lambda r_i)^{-\nu_j}J_{\nu_j}(\lambda r_i)\right) =- r_i(\lambda r_i)^{-\nu_j} J_{\nu_j+1}(\lambda r_i),
\end{equation}
and consequently
\begin{align*}
\begin{split}
 |(\frac{d}{d\lambda})^{\alpha_1} \big( (\lambda r_1)^{- \frac{n-2}{2} } J_{\nu_j}(\lambda_1r_1) \big)|
 & =  |(\frac{d}{d\lambda})^{\alpha_1} \big( (\lambda r_1)^{- \frac{n-2}{2}+\nu_j }  (\lambda r_1)^{-\nu_j}J_{\nu_j}(\lambda_1r_1) \big)|
 \\ & \lesssim 
 \sum_{k=0}^{\alpha_1} |r_1^k (\lambda r_1)^{- \frac{n-2}{2}+\nu_j -k}
 r_1^{\alpha_1-k} (\lambda r_1)^{-\nu_j} J_{\nu_j+\alpha_1-k}(\lambda r_1) |
 \\ & \lesssim \lambda^{-\alpha_1} \sum_{k=0}^{\alpha_1} |   (\lambda r_1)^{- \frac{n-2}{2} +\alpha_1 -k}   J_{\nu_j+\alpha_1-k}(\lambda r_1) |.
\end{split}
\end{align*}
This gives 
\begin{equation}\label{de-bes}
\begin{split}
&\Big|\Big(\frac{d}{d\lambda}\Big)^{\alpha}\Big(
(\lambda^2r_1r_2)^{-\frac{n-2}2}\sum_{j \geq 0} \psi_{j}(\hat{x})\overline{\psi_{j}(\hat{y})}J_{\nu_j}(\lambda r_1)J_{\nu_j}(\lambda r_2)\Big)\Big|
\\&\lesssim \lambda^{-\alpha}(\lambda^2r_1r_2)^{-\frac{n-2}2}
\sum_{j\geq 0}  \sum_{\alpha_1+\alpha_2=\alpha } (\lambda r_1)^{\alpha_1} (\lambda r_2)^{\alpha_2}
\\&\quad \times 
 \sum_{k=0}^{\alpha_1} \sum_{k'=0}^{\alpha_2} \Big| \psi_{j}(\hat{x})\overline{\psi_{j}(\hat{y})} 
(\lambda r_1)^{-k} J_{\nu_j+\alpha_1-k}(\lambda r_1)  (\lambda r_2)^{-k'}J_{\nu_j+\alpha_2-k'}(\lambda r_2)\Big|.
\end{split}
\end{equation}
We will estimate it by considering the two cases that $\lambda r_2\lesssim 1$ or $\lambda r_2\gg1$.

{\bf Case 1: When $\lambda r_2\lesssim 1$.} 
We first recall a few well-known estimates. Firstly, the eigenfunction estimate (see \cite[Eq~(3.2.5)-(3.2.6)]{Sogge-H}) 
 \begin{equation} \label{eq:eigen-L-infinity}
\|\psi_{k}(\hat{x})\|_{L^\infty(Y)}\leq  C (1+\nu^2_k)^{\frac{n-2}4},
 \end{equation}
and the Weyl’s asymptotic formula (see \cite[Eq~(0.6)]{Garding53})
 \begin{equation}\label{est:eig}
\nu^2_k\sim (1+k)^{\frac 2{n-1}},\quad k\geq 1,\implies\|\psi_{k}(\hat{x})\|^2_{L^\infty(Y)}\leq  C (1+k)^{\frac{n-2}{n-1}}.
 \end{equation}
 In addition, the classical bound on Bessel functions (see e.g. \cite[Section~8.1.4]{Stein}):
\begin{equation} \label{eq:bessel-bound-near-zero}
|J_\nu(t)|\leq
\frac{Ct^\nu}{2^\nu\Gamma(\nu+\frac12)\Gamma(1/2)}\big(1+\frac1{\nu+1/2}\big).
\end{equation}
Combining bounds above (in particular, \eqref{eq:bessel-bound-near-zero} shows that terms in \eqref{de-bes} with difference $k,k'$ can be bound by the same power of $\lambda r_i$, which depends only on $\alpha_1,\alpha_2$) with \eqref{de-bes}, we have 
\begin{equation}
\begin{split}
&\Big|\Big(\frac{d}{d\lambda}\Big)^{\alpha}\Big(
(\lambda^2r_1r_2)^{-\frac{n-2}2}\sum_{j\geq 0}\psi_{j}(\hat{x})\overline{\psi_{j}(\hat{y})}J_{\nu_j}(\lambda r_1)J_{\nu_j}(\lambda r_2)\Big)\Big|
\\&\lesssim \lambda^{-\alpha}\sum_{\alpha_1+\alpha_2=\alpha; } \sum_{j\geq0} \nu_j^{n-1} \frac{(\lambda r_1)^{\nu_j-\frac{n-2}2+2\alpha_1}}{2^{\nu_j}\Gamma(\nu_j+\frac12)\Gamma(1/2)}
\frac{(\lambda r_2)^{\nu_j-\frac{n-2}2+2\alpha_2}}{2^{\nu_j}\Gamma(\nu_j+\frac12)\Gamma(1/2)}.
\end{split}
\end{equation}
In the current case ($\lambda r_2\lesssim  1$),  we obtain
\begin{equation}
\begin{split}
&\Big|\Big(\frac{d}{d\lambda}\Big)^{\alpha}\Big(
(\lambda^2r_1r_2)^{-\frac{n-2}2}\sum_{j\geq 0}\psi_{j}(\hat{x})\overline{\psi_{j}(\hat{y})}J_{\nu_j}(\lambda r_1)J_{\nu_j}(\lambda r_2)\Big)\Big|
\\&\lesssim \lambda^{-\alpha}(\lambda^2r_1r_2)^{\nu_0-\frac{n-2}2}\sum_{j} \frac{\nu_j^{n-1}}{2^{2\nu_j}\Gamma(\nu_j+\frac12)\Gamma(\nu_j+\frac12)}\lesssim  \lambda^{-\alpha}(\lambda^2r_1r_2)^{\nu_0-\frac{n-2}2}.
\end{split}
\end{equation}
Therefore, when $\lambda r_1, \lambda r_2\lesssim 1$ (hence $\lambda\max\{r_1, r_2\} \lesssim 1$), for any $K\geq1$, we have
\begin{equation}
\begin{split}
&\Big|\Big(\frac{d}{d\lambda}\Big)^{\alpha}\Big(
(\lambda^2r_1r_2)^{-\frac{n-2}2}\sum_{j\geq 0}\psi_{j}(\hat{x})\overline{\psi_{j}(\hat{y})}J_{\nu_j}(\lambda r_1)J_{\nu_j}(\lambda r_2)\Big)\Big|
\\&\lesssim  \lambda^{-\alpha}(\lambda^2r_1r_2)^{\nu_0-\frac{n-2}2}(1+\lambda\max\{r_1, r_2\})^{-K},\quad \forall K\geq 1.
\end{split}
\end{equation}
So if $\lambda r_2\lesssim 1$, we can collect all terms into $B(\lambda;  x, y)$ and it satisfies \eqref{bean-B}.

{\bf Case 2: When $\lambda r_2\gg 1$.} For any $K\geq 1$, we define a constant $J:=\max\{\alpha+K,\nu_0\}$ and we assume $\lambda r_2\geq 4J$. 
In this case, we split the right hand side of \eqref{de-bes} into two parts
\begin{equation}\label{z< 2}
\begin{split}
&\lambda^{-\alpha}(\lambda^2r_1r_2)^{-\frac{n-2}2}\sum_{\alpha_1+\alpha_2=\alpha}(\lambda r_1)^{\alpha_1} (\lambda r_2)^{\alpha_2}
\\& \times \Big(\Big|\sum_{\{j: \nu_j\geq \lambda r_2/2\}}
 \sum_{k=0}^{\alpha_1} \sum_{k'=0}^{\alpha_2} \Big| \psi_{j}(\hat{x})\overline{\psi_{j}(\hat{y})} 
(\lambda r_1)^{-k} J_{\nu_j+\alpha_1-k}(\lambda r_1)  (\lambda r_2)^{-k'}J_{\nu_j+\alpha_2-k'}(\lambda r_2)\Big|
\\& +\Big|\sum_{\{j: \nu_j < \lambda r_2/2\}}
 \sum_{k=0}^{\alpha_1} \sum_{k'=0}^{\alpha_2} \Big| \psi_{j}(\hat{x})\overline{\psi_{j}(\hat{y})} 
(\lambda r_1)^{-k} J_{\nu_j+\alpha_1-k}(\lambda r_1)  (\lambda r_2)^{-k'}J_{\nu_j+\alpha_2-k'}(\lambda r_2)\Big| \Big).
\end{split}
\end{equation}

For the summation $\{j: \nu_j\geq \lambda r_2/2\}$, we use the fact that 
\begin{equation} \label{eq:Bessel-uniform-nu}
|J_{\nu}(x)|\leq C x^{-\frac{1}{3}},
\end{equation}
where $x\geq 1$ and the constant $C$ is independent of $\nu$ (see e.g. \cite[Section~VIII.5.2(b)]{Stein}). 
Recall that $\lambda r_1\lesssim 1$ and $\lambda r_2\gg 1$, we obtain that (the way to bound those factors involving $\lambda r_1$ is the same as above, using \eqref{eq:bessel-bound-near-zero}; while for $\lambda r_2$, we only keep the term with highest power in $\lambda r_2$ after using \eqref{eq:Bessel-uniform-nu} now):
\begin{equation*}
\begin{split}
&\lambda^{-\alpha}(\lambda^2r_1r_2)^{-\frac{n-2}2} \sum_{\{j: \nu_j\geq \lambda r_2/2\}} \sum_{k=0}^{\alpha_1} \sum_{k'=0}^{\alpha_2} (\lambda r_1)^{\alpha_1} (\lambda r_2)^{\alpha_2}
\\&  \times 
 \Big| \psi_{j}(\hat{x})\overline{\psi_{j}(\hat{y})} 
(\lambda r_1)^{-k} J_{\nu_j+\alpha_1-k}(\lambda r_1)  (\lambda r_2)^{-k'}J_{\nu_j+\alpha_2-k'}(\lambda r_2)\Big|
\\&\lesssim \lambda^{-\alpha}(\lambda r_1)^{-\frac{n-2}2} 
\sum_{\{j: \nu_j\geq \lambda r_2/2\}} \sum_{\alpha_1+\alpha_2=\alpha} \nu_j^{n-1} \frac{(\lambda r_1)^{\nu_j+2\alpha_1}}{2^{\nu_j}\Gamma(\nu_j+\frac12)\Gamma(1/2)} (\lambda r_2)^{\alpha_2-\frac13-(n-2)/2}.
\end{split}
\end{equation*}

For those terms with $\nu_j\geq \lambda r_2/2\geq 2J$, we estimate them
\begin{equation*}
\begin{split}
&\lesssim \lambda^{-\alpha} (\lambda r_1)^{K} \sum_{\{j: \nu_j\geq \lambda r_2/2\}}  \frac{\nu_j^{n-1}}{2^{\nu_j}\Gamma(\nu_j+\frac12)\Gamma(1/2)}
(\lambda r_2)^{\alpha-\frac13-(n-2)/2}
\\&\lesssim \lambda^{-\alpha}(\lambda r_1)^{K} (\lambda r_2)^{-K} \sum_{j} \frac{\nu_j^{n-1+\alpha+K}}{2^{\nu_j}\Gamma(\nu_j+\frac12)}\\
&\lesssim  \lambda^{-\alpha}(\lambda r_1)^{K}(1+\lambda \max\{r_1, r_2\})^{-K}.
\end{split}
\end{equation*}
Thus, this term can be included in the term $B(\lambda;  x, y)$ satisfying \eqref{bean-B} again.
\vspace{0.2cm}

Note that we can write the sum over terms with $\lambda r_2\geq 2\nu_j$, by using \eqref{eq:bess3} and \eqref{eq:bess5} in Lemma~\ref{lem:bessel}, as 
\begin{equation}
\begin{split}
&(\lambda^2r_1r_2)^{-\frac{n-2}2}\sum_{\{j: 2\nu_j\leq \lambda r_2\}}\psi_{j}(\hat{x})\overline{\psi_{j}(\hat{y})}J_{\nu_j}(\lambda r_1)J_{\nu_j}(\lambda r_2)\\
&=(\lambda^2r_1r_2)^{-\frac{n-2}2}\sum_{\{j: 2\nu_j\leq \lambda r_2\}}\psi_{j}(\hat{x})\overline{\psi_{j}(\hat{y})}J_{\nu_j}(\lambda r_1)(\lambda r_2)^{-\frac12}\sum_{\pm} e^{\pm i\lambda r_2} j_\pm (\nu_j; \lambda r_2)\\
&=\sum_{\pm} e^{\pm i\lambda r_2} A_\pm(\lambda; r_1, r_2; \hat{x}, \hat{y}), \end{split}
\end{equation}
where
\begin{equation}\label{A12}
A_\pm(\lambda; r_1, r_2; \hat{x}, \hat{y})=(\lambda^2r_1r_2)^{-\frac{n-2}2}\sum_{\{j: 2\nu_j\leq \lambda r_2\}}\psi_{j}(\hat{x})\overline{\psi_{j}(\hat{y})}J_{\nu_j}(\lambda r_1)(\lambda r_2)^{-\frac12}j_\pm (\nu_j; \lambda r_2).
\end{equation}
Now we verify that \eqref{A12} satisfy \eqref{bean-A}. Computing derivatives as in \eqref{eq:bessel-lambda-derivative} and applying \eqref{eq:bess5} of Lemma~\ref{lem:bessel} to factors involving $\lambda r_2$ and \eqref{eq:bessel-bound-near-zero} to factors involving $\lambda r_1$, we obtain
\begin{equation}
\begin{split}
&\Big|\Big(\frac{d}{d\lambda}\Big)^{\alpha}
A_\pm(\lambda; r_1, r_2; \hat{x}, \hat{y})\Big|
\\&\lesssim \lambda^{-\alpha}\sup_{0\leq \alpha_1+ \alpha_2\leq \alpha}\sum_{\{j: 2\nu_j\leq \lambda r_2\}} \nu_j^{n-1} \frac{(\lambda r_1)^{\nu_j+\alpha_1}2^{\nu_j} \nu_j^{\alpha_2}}{2^{\nu_j}\Gamma(\nu_j+\frac12)\Gamma(1/2)}
(\lambda r_2)^{-\alpha_2} (\lambda r_2)^{-\frac12} (\lambda^2r_1r_2)^{-\frac{n-2}2}, \quad \\
&\lesssim  \lambda^{-\alpha}(\lambda r_1)^{\nu_0-\frac{n-2}2} (1+\lambda r_2)^{-\frac{n-1}2},
\end{split}
\end{equation}
which implies \eqref{bean-A}.
In sum, we have proved \eqref{Spect<1}.
\end{proof}

\section{The decay estimates with NREC}
\label{sec:proof-with-NREC}

In this section, we prove Theorem \ref{thm:dispersive} about the pointwise estimates of the microlocalized half-wave propagator under the technical condition NREC stated in \eqref{eq:NREC-quantitative}. 
In particular, $Y=\mathbb{S}^{n-1}$ satisfies the technical condition NREC due to the fact the injective radius of unit sphere is $\pi$.
This technical assumption will be removed in Section~\ref{sec:without-NREC} to obtain our main results in full generality.

\subsection{The proof of \eqref{est:dispersive<}}

\begin{proof}[The proof of \eqref{est:dispersive<} in Theorem \ref{thm:dispersive}]
We only consider the first two cases $2^k r_1, 2^k r_2\lesssim 1$ and $2^k r_1\lesssim 1\ll  2^k r_2$ since the third case can be treated similarly.  Due to the compact support of $\varphi$, one has $\lambda\sim 2^k$ on the region we are concerning.
Then we can use the representation of spectral measure in \eqref{Spect<1} to write
\begin{equation*}
\begin{split}
&\int_0^\infty \varphi (2^{-k}\lambda) e^{it\lambda} dE_{\sqrt{\mathcal{L}_{{\A},a}}}(\lambda; x, y)\\
&=\int_0^\infty \varphi (2^{-k}\lambda) e^{it\lambda} \lambda^{n-1}\Big(\sum_{\pm}A_{\pm}(\lambda; x, y) e^{\pm i\lambda \max\{r_1, r_2\}}
+B(\lambda; x, y)\Big)\, d\lambda
\end{split}
\end{equation*}
where $A_\pm$ and $B$ satisfy \eqref{bean-A} and \eqref{bean-B} respectively. For the case $2^k r_1, 2^k r_2\lesssim 1$, $A_\pm$ vanishes, and we perform integration by parts for $N$-times on the $B$-term and use \eqref{bean-B} to obtain (derivatives hitting $\varphi(2^{-k}\lambda)$ create a $2^{-k}$-factor each time, which is comparable to $\lambda$):
\begin{equation} \label{eq:B-term-lambda r1,r2-small}
\begin{split}
&\Big|  \int_0^\infty  e^{it\lambda} \varphi (2^{-k}\lambda) B(\lambda; x, y)  \lambda^{n-1} d\lambda\Big|\\& \leq
C_N|t|^{-N}(2^{2k}r_1r_2)^{\nu_0-\frac{n-2}2}\int_{2^{k-1}}^{2^{k+1}}\lambda^{n-1-N}(1+\lambda
\max\{r_1, r_2\})^{-\frac{n+1}2}d\lambda\\&\leq
C_N2^{k(n-N)}(2^{2k}r_1r_2)^{\nu_0-\frac{n-2}2} |t |^{-N}(1+2^k \max\{r_1, r_2\})^{-\frac{n+1}2}.
\end{split}
\end{equation}
Thus we obtain 
\begin{equation*}
\begin{split}
&\Big|  \int_0^\infty  e^{it\lambda} \varphi (2^{-k}\lambda) B(\lambda; x, y)  \lambda^{n-1} d\lambda\Big| \leq
C_N2^{kn}(2^{2k}r_1r_2)^{\nu_0-\frac{n-2}2} (1+2^k|t|)^{-N},
\end{split}
\end{equation*}
which implies \eqref{est:dispersive<} by choosing $N$ large enough. \vspace{0.2cm}

Now we consider the case that $2^k r_1\lesssim 1\ll  2^k r_2$. For term with $A_\pm$, we perform integration by parts for $N$-times and use \eqref{bean-A} to obtain 
\begin{equation*}
\begin{split}
&\Big|  \int_0^\infty  e^{it\lambda} \varphi (2^{-k}\lambda) A_{\pm}(\lambda; x, y) e^{\pm i\lambda \max\{r_1, r_2\}} \lambda^{n-1} d\lambda\Big|\\&\leq \Big|\int_0^\infty \left(\frac1{
i(t\pm \max\{r_1, r_2\})}\frac\partial{\partial\lambda}\right)^{N}\big(e^{i(t\pm \max\{r_1, r_2\})\lambda}\big)
\varphi (2^{-k}\lambda) A_{\pm}(\lambda; x, y) \lambda^{n-1} d\lambda\Big|\\& \leq
C_N|t\pm  \max\{r_1, r_2\}|^{-N}(2^{k}r_1)^{\nu_0-\frac{n-2}2}\int_{2^{k-1}}^{2^{k+1}}\lambda^{n-1-N}(1+\lambda
\max\{r_1, r_2\})^{-\frac{n-1}2}d\lambda\\&\leq
C_N2^{k(n-N)}(2^{k}r_1)^{\nu_0-\frac{n-2}2} |t\pm \max\{r_1, r_2\} |^{-N}(1+2^k \max\{r_1, r_2\})^{-\frac{n-1}2}.
\end{split}
\end{equation*}
Therefore, it follows that
\begin{equation*}\label{dispersive2}
\begin{split}
&\Big|  \int_0^\infty  e^{it\lambda} \varphi (2^{-k}\lambda) A_{\pm}(\lambda; x, y) e^{\pm i\lambda \max\{r_1, r_2\}} \lambda^{n-1} d\lambda\Big|\\&\leq
C_N2^{kn}(2^{k}r_1)^{\nu_0-\frac{n-2}2} \big(1+2^k|t\pm \max\{r_1, r_2\}|\big)^{-N}(1+2^k \max\{r_1, r_2\})^{-\frac{n-1}2}.
\end{split}
\end{equation*}
If $|t|\sim \max\{r_1, r_2\}$, this proves \eqref{est:dispersive<}.
Otherwise, we have $|t\pm \max\{r_1, r_2\}|\geq c|t|$ for some constant $c$, choosing $N$ large enough gives us \eqref{est:dispersive<}. For the term $B$ in this case, we can do the similar argument as \eqref{eq:B-term-lambda r1,r2-small} to see \eqref{est:dispersive<}.
\end{proof}

\subsection{Preliminaries for \eqref{est:dispersive}}
 
Now we turn to \eqref{est:dispersive}. In this case, we use \eqref{spect} to see that
\begin{equation}
\begin{split}
&\int_0^\infty \varphi (2^{-k}\lambda) e^{it\lambda} dE_{\sqrt{\mathcal{L}_{{\A},a}}}(\lambda; x, y)\\
&= \frac{1}{4\pi^2} (r_1r_2)^{-\frac{n-2}2}\Big(\frac1{\pi}
\int_0^\pi \int_{\R^2} \varphi (2^{-k}|\xi|) e^{i(t|\xi|-{\bf m}_s \cdot\xi)} d\xi  \cos(s\sqrt{P})(\hat{x},\hat{y}) ds
 \\&\qquad\qquad\qquad
-\frac{\sin(\pi\sqrt{P})}{ \pi }\int_0^\infty \int_{\R^2} \varphi (2^{-k}|\xi|) e^{i(t|\xi|-{\bf n}_s \cdot\xi)} d\xi 
\, e^{-s\sqrt{P}}(\hat{x},\hat{y}) ds\Big)\\
&= \frac{2^{2k}}{4\pi^2} (r_1r_2)^{-\frac{n-2}2}\Big(\frac1{\pi}
\int_0^\pi W(2^kt, 2^k{\bf m}_s) \cos(s\sqrt{P})(\hat{x},\hat{y}) ds
 \\&\qquad\qquad\qquad
-\frac{\sin(\pi\sqrt{P})}{ \pi }\int_0^\infty W(2^kt, 2^k{\bf n}_s)
\, e^{-s\sqrt{P}}(\hat{x},\hat{y}) ds\Big).
\end{split}
\end{equation}
where 
\begin{equation}\label{def:W}
\begin{split}
W(t, v)=\int_{\R^2} \varphi (|\xi|) e^{i(t|\xi|- v \cdot\xi)} d\xi ,\quad v \in \R^2.
\end{split}
\end{equation}

To prove \eqref{est:dispersive}, by using the scaling, then it suffices to prove 
\begin{equation}\label{est:dispersive'}
\begin{split}
& (r_1r_2)^{-\frac{n-2}2}  \Big|Q_j\Big(\frac1{\pi}
\int_0^\pi W(t, {\bf m}_s) \cos(s\sqrt{P})(\hat{x},\hat{y}) ds
 \\&\qquad
-\frac{\sin(\pi\sqrt{P})}{ \pi }\int_0^\infty W(t,{\bf n}_s)
\, e^{-s\sqrt{P}}(\hat{x},\hat{y}) ds\Big)Q_j^*\Big|\leq C\big(1+|t|\big)^{-\frac{n-1}2},
\end{split}
\end{equation}
when $r_1, r_2\gg1$.
Using Lemma~\ref{lemma: parametrix 1} and Lemma~\ref{lemma: localized-poisson}, we now decompose the left hand side of \eqref{est:dispersive'} into two terms, the propagation term:
\begin{equation} \label{eq:IP,definition}
I_P(t,x,y) = Q_j(\hat{x})\Big(\frac{(r_1r_2)^{-\frac{n-2}2}}{\pi}\int_0^\pi W(t, {\bf m}_s) K_N(s;\hat{x},\hat{y}) ds \Big)Q_j(\hat{y}), 
\end{equation}
and the residual term:
\begin{align} \label{eq:IR,definition}
\begin{split}
I_R(t,x,y) = & Q_j(\hat{x})\Big(\frac{(r_1r_2)^{-\frac{n-2}2}}{\pi}\int_0^\pi W(t, {\bf m}_s)  R_N(s;\hat{x},\hat{y}) )ds
\\& -(r_1r_2)^{-\frac{n-2}2}\frac{\sin(\pi\sqrt{P})}{\pi}\int_0^\infty  W(t, {\bf n}_s) e^{-s\sqrt{P}} ds \Big)Q_j(\hat{y}).
\end{split}
\end{align}

Therefore, the inequality \eqref{est:dispersive} is proved if we could prove that both $I_P$ and $I_R$ are uniformly bounded by $(1+|t|)^{-\frac{n-1}2}$ up to a constant when $r_1, r_2\gg1$, and this is the content of the rest of this section.

From now on, without loss of generality, we only consider the cases that
\begin{equation}\label{assu:r}
r_1\gg r_2\gg1, \quad \text{or} \quad r_1\sim r_2\gg1.
\end{equation}
The proof is a stationary phase argument applied to those oscillatory integrals. As the phase is different with the case with Schr\"odinger operators in \cite{JZ,JZ2}, the corresponding region around the critical set having the major contribution would be different as well. Consequently, we need a different strategy to decompose the region of integrations.

Before estimating the two terms, we first recall \eqref{bf-m} and observe that
\begin{equation}\label{fact:ms}
|{\bf m}_s|=\sqrt{(r_1-r_2)^2+4r_1r_2\sin^2(\frac{s}2)}\sim 
\begin{cases} 
r_1, \quad &r_1\gg r_2\gg1,
\\  r_2, \quad &r_2\gg r_1\gg1,
\\ r_1\sin(\frac s2), \quad &r_1\sim r_2\gg1.
\end{cases}
\end{equation}

Recalling the definition of ${\bf n}_s$ in \eqref{bf-n}, we know
\begin{equation}\label{fact:ns}
|{\bf n}_s|=\sqrt{r_1^2+r_2^2+2r_1r_2\cosh(s)}\gtrsim r_1+r_2\sim \max\{r_1,r_2\}.
\end{equation}
Additionly, if $s \lesssim 1$, the $\gtrsim$ in \eqref{fact:ns} can be replaced by $\sim$.

 We have the following lemma estimating $W(t, {\bf m}_s),W(t, {\bf n}_s)$ and their derivatives.
\begin{lemma} Let $W(t, {\bf m}_s)$ and $W(t, {\bf n}_s)$ be given by \eqref{def:W} with ${\bf m}_s$ and ${\bf n}_s$ in \eqref{bf-m} and \eqref{bf-n}, respectively. Then, for any integers $m, N, K \geq 0$, there exists a constant $C_{N,K,m}$ such that 
\begin{equation}\label{est:W}
\begin{split}
\Big|\big(\frac{\partial}{\partial s}\big)^m &W(t, {\bf m}_s)\Big|
\leq C_{N,K,m}\big(1+|t\pm |{\bf m}_s||\big)^{-N}\\
&\times\big(1+\frac{|t|}{|{\bf m}_s|}\big)^{-K}(1+|{\bf m}_s|)^{-1/2} \Big( \big(1+\frac{r_1r_2\sin s}{|{\bf m}_s|}\big)^{m} +\big(1 + \frac{r_1r_2}{|{\bf m}_s|}\big)^{m/2} \Big),
\end{split}
\end{equation}
when $r_1, r_2\gg1$ and $s\in[0,\pi]$. Similarly, for $r_1, r_2\gg1$ and $s \in [0,\infty)$, we have
\begin{equation}\label{est:W-n}
\begin{split}
&\Big|\big(\frac{\partial}{\partial s}\big)^m W(t, {\bf n}_s)\Big|\\
&\leq C_{N,m}\big(1+|t\pm |{\bf n}_s||\big)^{-N}(1+|{\bf n}_s|)^{-1/2} \Big( \big(1+\frac{r_1r_2\sinh s}{|{\bf n}_s|}\big)^m + \big(1+\frac{r_1r_2}{|{\bf n}_s|}\big)^{m/2} \Big).
\end{split}
\end{equation}
\end{lemma}

\begin{proof} By the definition of $W(t, {\bf m}_s)$ in \eqref{def:W}, using \eqref{a-pm1}, we integrate over the angle direction (of $\xi$) to obtain:
\begin{equation}\label{def:W'}
W(t, {\bf m}_s)=\sum_{\pm} \int_0^\infty e^{i\lambda(t\pm |{\bf m}_s|)} \mathfrak{a}_{\pm}(\lambda |{\bf m}_s|) \varphi(\lambda) \,\lambda d\lambda
\end{equation}
where $\mathfrak{a}_{\pm}$ satisfies \eqref{bean}. 
%%%%%%%%%%%%%%%%%%%%%%%%%%%%%%%%%%%%%%%%%%%%%%%%%%%%
Let $\mathsf{z}=\frac{r_1r_2}{|{\bf m}_s|}$. 
When we differentiate in $s$ for $m$-times, the integrand is a sum of terms of the form (modulo uniformly bounded smooth factors)
 \begin{equation} \label{eq:ms-derivative-extra-factors}
  (\mathsf{z} \sin s)^{k} \mathsf{z}^{k'}P(\cos s),
 \end{equation} 
 where $P(\cdot)$ is a polynomial.
Here $k+k'$ is the number of the derivatives falls on $e^{i\lambda |{\bf m}_s|}$. Notice that a $\mathsf{z}$ factor without $\sin s$ paired to it can only arise by differentiating $\sin s$ (or its power), so $\mathsf{z}^{k'}$ has also costed at least $k'$ derivatives on $\sin s$ and we have (one can produce $\sin s$-factors by differentiating $\cos s$, but inequality below still holds)
 \begin{equation}
k+2k' \leq m.
 \end{equation}
So \eqref{eq:ms-derivative-extra-factors} is bounded (up to a constant) by
\begin{equation}
\big(1+\frac{r_1r_2\sin s}{|{\bf m}_s|}\big)^{m} + \big(1 + \frac{r_1r_2}{|{\bf m}_s|}\big)^{m/2},
\end{equation}
and we have
\begin{equation}
\begin{split}
&\Big|\big(\frac{\partial}{\partial s}\big)^m W(t, {\bf m}_s)\Big| \\
&\lesssim \Big( (1+\frac{r_1r_2\sin s}{|{\bf m}_s|})^{m} + (1 + \frac{r_1r_2}{|{\bf m}_s|})^{m/2} \Big) \Big| \int_0^\infty e^{i\lambda(t\pm |{\bf m}_s|)} \mathfrak{a}_{\pm}(\lambda |{\bf m}_s|) \tilde{\varphi}(\lambda) \, d\lambda\Big|,
\end{split}
\end{equation}
where $\mathfrak{a}_\pm$ still satisfies \eqref{bean} and $\tilde{\varphi}(\lambda)=\lambda^{1+m} \varphi(\lambda)$ still is supported $[1,2]$.
We use $N+K$-times integration by parts in $\lambda$, among $N$ times of which we write the integrand as (with potentially different $\mathfrak{a}_\pm$ each time)
\begin{equation*}
 [i(t\pm |{\bf m}_s|)]^{-1} (\partial_{\lambda}e^{i\lambda(t\pm |{\bf m}_s|)} )\mathfrak{a}_\pm(\lambda |{\bf m}_s|) \tilde{\varphi}(\lambda),
\end{equation*}
and for the rest $K$-times (if $t \leq |{\bf m}_s|$, ignore this step but the final inequality \eqref{est:W} with $(1+\frac{|t|}{|{\bf m}_s|}\big)^{-K}$ is still true as that is effectively a constant factor in that case) we write the integrand as
\begin{equation*}
(it)^{-1} (\partial_\lambda e^{i\lambda t}) \big(e^{ \pm i\lambda|{\bf m}_s|} \mathfrak{a}_\pm(\lambda |{\bf m}_s|)\tilde{\varphi}(\lambda)\big).
\end{equation*}
After those steps, we know (with a smooth function $g(\lambda)$):
\begin{equation*}
\begin{split}
&\Big|\big(\frac{\partial}{\partial s}\big)^m W(t, {\bf m}_s)\Big| \\
&\lesssim \Big( (1+\frac{r_1r_2\sin s}{|{\bf m}_s|})^{m} + (1 + \frac{r_1r_2}{|{\bf m}_s|})^{m/2} \Big) |t\pm |{\bf m}_s||^{-N} \min\Big(1, \frac{|{\bf m}_s|}{|t|}\Big)^{K}  \int_1^2 (1+\lambda |{\bf m}_s|)^{-1/2} g(\lambda) d\lambda\\
& \leq C_{N,K, m} \Big( (1+\frac{r_1r_2\sin s}{|{\bf m}_s|})^{m} + (1 + \frac{r_1r_2}{|{\bf m}_s|})^{m/2} \Big)  |t\pm |{\bf m}_s||^{-N}\min\Big(1,  \frac{|{\bf m}_s|}{|t|}\Big)^{K}  (1+|{\bf m}_s|)^{-1/2},
\end{split}
\end{equation*}
which implies \eqref{est:W}.

The estimate \eqref{est:W-n} follows from the same argument by replacing ${\bf m}_s$ by ${\bf n}_s$ with $K=0$.
\end{proof}

%%%%%%%%%%%%%%%%%%%%%%%%%%%%%%%%%%%%%%%%%
\subsection{The proof of \eqref{est:dispersive}}
\label{subsec:main-proof-with-NREC}

Now we estimate the contributions of $I_P$ and $I_R$ to prove \eqref{est:dispersive}.

\begin{proof}[The contribution of $I_R$] 
We first consider the integral of $W(t, {\bf m}_s)  R_N(s; \hat{x}, \hat{y} )$ in $I_R$. 
In this part, $W(t, {\bf m}_s)$ and the prefactor $(r_1r_2)^{-\frac{n-2}2}$ have already offered sufficient decay and there is no singularity of the kernel that we are integrating over.
More concretely, consider the contribution from $R_N$ first. Using \eqref{est:W} with $m=0$, we know
  \begin{equation}\label{est:RN}
\begin{split}
&(r_1r_2)^{-\frac{n-2}2}\Big| \int_0^\pi W(t, {\bf m}_s)  R_N(s; \hat{x}, \hat{y} ) ds\Big| \\
&\lesssim (r_1r_2)^{-\frac{n-2}2}\int_0^\pi \big(1+|t\pm |{\bf m}_s||\big)^{-N}(1+|{\bf m}_s|)^{-1/2} \,ds
  \end{split}
 \end{equation}
due to the fact that  $$|R_N(s,\hat{x}, \hat{y})|\lesssim 1,\quad 0\leq s\leq \pi.$$ 
If either $|t|\ll |{\bf m}_s|$ or $|t|\gg |{\bf m}_s|$,  we have $|t\pm |{\bf m}_s||\geq c|t|$ for some small constant
$c$, and then we choose $N$ large enough to obtain 
\begin{equation*}
\eqref{est:RN}\lesssim (r_1r_2)^{-\frac{n-2}2} (1+|t|)^{-N},
\end{equation*}
which implies the desired $(1+|t|)^{-(n-1)/2}$-bound since $r_1, r_2\gg1$.
Otherwise, $|t|\sim |{\bf m}_s|$, then from \eqref{fact:ms} and \eqref{assu:r}, we see $|t|\sim r_1\gg1$ or $|t|\sim r_1\sin(\frac s2)\lesssim r_1$.
In both cases, we obtain
\begin{equation*}
\eqref{est:RN}\lesssim (r_1r_2)^{-\frac{n-2}2} (1+|t|)^{-1/2}\lesssim (1+|t|)^{-(n-1)/2},
\end{equation*}
which gives the desired $(1+|t|)^{-(n-1)/2}$-bound.

Now we consider the second part of $I_R$ contributed from $W(t, {\bf n}_s)e^{(-s \pm i\pi)\sqrt{P}}$. 
When $|{\bf n}_s| \lesssim \max\{r_1,r_2\}$, and this part can be estimated in the same manner as above, with the only difference being that we instead use \eqref{est:W-n} to estimate the $W(s,{\bf n}_s)$-factor and use Lemma~\ref{lemma: localized-poisson} to show that the kernel after localization is bounded. More concretely, we have
\begin{align}\label{est:IR}
\begin{split}
& \Big| Q_j(\hat{x})\Big(
(r_1r_2)^{-\frac{n-2}2}\frac{\sin(\pi\sqrt{P})}{\pi}\int_0^\infty  W(t, {\bf n}_s) e^{-s\sqrt{P}} ds \Big)Q_j(\hat{y})\Big|\\
&\lesssim (r_1r_2)^{-\frac{n-2}2} \int_0^\infty \big(1+|t\pm |{\bf n}_s||\big)^{-N} (1+|{\bf n}_s|)^{-1/2}\, ds.
\end{split}
\end{align}
Also, on the part of $s$ such that $|t \pm |{\bf n}_s|| \geq c|t|$ for a constant $c$, then the same argument as in the previous case will give the desired bound
\begin{align*}
\begin{split}
& \eqref{est:IR}
\lesssim (r_1r_2)^{-\frac{n-2}2} \big(1+|t|\big)^{-N}\int_0^\infty (1+e^{\frac s2})^{-1/2} \, ds\lesssim  (1+|t|)^{-N}.
\end{split}
\end{align*}
On the part of $s$ such that $|{\bf n}_s| \sim |t|$ and $|{\bf n}_s| \lesssim \max\{r_1,r_2\}$, we obtain
\begin{align*}
\begin{split}
 \eqref{est:IR}
&\lesssim (r_1r_2)^{-\frac{n-2}2} \int_{\{s\in[0,+\infty): |{\bf n}_s| \lesssim \max\{r_1,r_2\}\}} (r_1r_2)^{-1/4}  e^{-\frac s4} \, ds\\
&\lesssim  (r_1r_2)^{-\frac{n-2}2} \max\{r_1,r_2\}^{-1/2}\lesssim (1+|t|)^{-(n-1)/2}.
\end{split}
\end{align*}

So the only part of the contribution we need to discuss is the part $|{\bf n}_s| \sim |t|$ and in addition $|{\bf n}_s| \sim e^{s/2}(r_1r_2)^{1/2} \gg \max\{r_1,r_2\}$, which means $|t| \gg \max\{r_1,r_2\}$ and the $(r_1r_2)^{-(n-2)/2}$-factor can't offer us $|t|$-decay anymore. In addition, we must have $s \gg 1$ on this part and we only need to estimate:
\begin{equation} 
\begin{split}
\mathsf{K}(t; x, y):= (r_1r_2)^{-\frac{n-2}2}\int_0^\infty \chi(s)  W(t, {\bf n}_s) e^{-(s\pm i\pi)\sqrt{P}} ds,
\end{split}
\end{equation}
with $t \gg \max\{r_1,r_2\}$ and $\chi(s)$ is smooth with sufficiently many ($\frac{n-2}{2}$ is sufficient) uniformly bounded derivatives, supported in $\{s\geq1: ||{\bf n}_s| -t| \leq c|t|\}$ and identically one in a slightly smaller interval. 

%%%%%%%%%%%%%%%%%%%%%%%%%%%%%%%%%%%%%%
%%%%%%%%%%%%%%%%%%%%%%%%%%%%%%%%%%%%%%%

For this part, instead of using the parametrix, we use the precise formula of the kernel from \eqref{FA} directly to write
\begin{equation} \label{eq:Poisson-Y-K1-K2-decomposition}
\begin{split}
&e^{-(s\pm i\pi)\sqrt{P}}=\sum_{k\in\mathbb{N}}e^{-(s\pm i\pi)\nu_k}\psi_{k}(\hat{x})\overline{\psi_{k}(\hat{y})} = K_{1}+K_2,
\end{split}
\end{equation}
where 
\begin{equation*}
\begin{split}
K_1(s;\hat{x},\hat{y}) = & \sum_{ \nu_k \leq \frac{n-2}{2} }   e^{-(s\pm i\pi)\nu_k}\psi_{k}(\hat{x})\overline{\psi_{k}(\hat{y})},\\
K_2(s;\hat{x},\hat{y}) = & \sum_{ \nu_k > \frac{n-2}{2} }   e^{-(s\pm i\pi)\nu_k}\psi_{k}(\hat{x})\overline{\psi_{k}(\hat{y})}.
\end{split}
\end{equation*}
And we have
\begin{equation} \label{eq:Poisson-X-K1-K2-decomposition}
\mathsf{K}(t; x, y) = \mathsf{K}_1(t;x,y)+\mathsf{K}_2(t;x,y),
\end{equation}
where
\begin{equation}
\mathsf{K}_i(t;x,y) = (r_1r_2)^{-\frac{n-2}2}\int_0^\infty \chi(s)  W(t, {\bf n}_s) K_i(s;\hat{x},\hat{y}) ds, \quad i = 1,2.
\end{equation}

As we have $s>0$ in this case, the growth of $\nu_k$ (due to the Weyl's law) leads to exponential decay.
The $\psi_k$-factors can be controlled by the estimate in \eqref{eq:eigen-L-infinity}: $\| \psi_k \|_{L^\infty}\leq C \nu_k^{(n-1)}$.
So we have
\begin{equation*}
\begin{split}
\Big|K_2(s;\hat{x},\hat{y}) \Big| \leq & \sum_{ \nu_k > \frac{n-2}{2} } e^{-\frac{s\nu_k}2}\nu_k^{2(n-1)}e^{-s(\frac{n-2}{4})} 
\lesssim e^{-s(\frac{n-2}{4})}.
\end{split}
\end{equation*}
As we have $|{\bf n}_s| \sim (r_1r_2)^{1/2}e^{s/2} \sim |t| $ on the region that we are concerning, this in turn gives (combining with \eqref{est:W-n})
\begin{equation*}
\begin{split}
|\mathsf{K}_2(t;x,y)| \lesssim  (r_1r_2)^{-\frac{n-2}2}\int_{ ||{\bf n}_s| -t| \leq c|t| } ( 1+|{\bf n}_s|)^{-1/2}  e^{-s(\frac{n-2}{4})} ds  \lesssim (1+|t|)^{- \frac{n-1}{2} }.
\end{split}
\end{equation*}
Here the length of the interval of integral in $s$ is $O(1)$.

Now we turn to $\mathsf{K}_1(t;x,y)$. For this term, we have to gain $e^{-\frac{s}{2}}$ by integration by parts. The fortunate aspect is now we are only dealing with a finite number of terms and we can consider the contribution from each eigen-space individually.
For each $k$ such that $\nu_k \leq \kappa$, we need to estimate (recall \eqref{def:W}):
\begin{equation*}
\begin{split}
\psi_{k}(\hat{x})\overline{\psi_{k}(\hat{y})} \int_0^\infty \int_{\R^2} \varphi(|\xi|) \chi(s) e^{i(t\pm|{\bf n}_s|)|\xi|}\, d\xi e^{-(s\pm i\pi)\nu_k}\, ds,
\end{split}
\end{equation*}
where $\mathrm{supp} \varphi \subset [1,2]$.
We will estimate this by integration by parts in $s$ and the integrand is controlled using the Lemma below.

\begin{lemma} \label{lemma: 1/ns derivative bound}
Let $|{\bf n}_s|$ be in \eqref{bf-n}, when $|{\bf n}_s| \sim (r_1r_2e^{s})^{1/2} \gg \max\{r_1,r_2\} \gg 1$, for any $k,j \in \mathbb{N}$, we have
\begin{equation} \label{eq: ns derivative bound}
\Big(\frac{\partial}{\partial s}\Big)^j|{\bf n}_s| \lesssim_j   (r_1r_2e^{s})^{1/2} \sim |{\bf n}_s|,
\end{equation}
and
\begin{equation}
\Big|\frac{\partial^k}{\partial s^k}\big( \frac{1}{|{\bf n}_s|}\big)\Big| \lesssim_k \frac{1}{|{\bf n}_s|}. 
\end{equation}

\end{lemma}

\begin{proof}

A direct computation gives
\begin{equation}
\frac{\partial }{\partial s}|{\bf n}_s| =  \frac{r_1r_2\sinh s}{ (r_1^2+r_2^2+2r_1r_2 \cosh s)^{1/2}}\lesssim |{\bf n}_s|.
\end{equation}
And  
\begin{equation}
\frac{\partial ^2}{\partial s^2}|{\bf n}_s| =  \frac{(r_1 r_2 (2 (r_1^2 + r_2^2) \cosh s + r_1 r_2 (3 + \cosh(2s)))} {(2 (r_1^2 + 
   r_2^2 + 2 r_1 r_2 \cosh s)^{3/2})}\lesssim |{\bf n}_s|.
\end{equation}
We can inductively prove that $(\frac{\partial ^j}{\partial s})^j|{\bf n}_s|$ takes the form:
\begin{equation} \label{eq:ns,derivative}
(\frac{\partial}{\partial s})^j|{\bf n}_s| = \frac{ \sum_{l=-j}^j A_{jl}(r_1r_2,r_1^2+r_2^2) e^{ls} }{ |{\bf n}_s|^{2j-1} },
\end{equation}
where $A_{jl}$ is a polynomial of degree at most $j$ which is $(j-|l|)$-degree in $(r_1^2+r_2^2)$ and $|l|$-degree in $r_1r_2$. Here $(r_1^2+r_2^2)$-factors arise due to we want to combine terms having denominators with different powers of $|{ \bf n}_s|$.
Since both $(r_1^2+r_2^2)$ and $r_1r_2e^{\pm s}$ can be bounded by $|{ {\bf n} }_s|^2$, this gives us \eqref{eq: ns derivative bound}.

Notice that $(\frac{\partial}{\partial s})^k( \frac{1}{|{\bf n}_s|}) =  \frac{P_k}{|{\bf n}_s|^{k+1} } $, where $P_k$ is a polynomial of derivatives of $|{\bf n}_s|$ of degree at most $k$. Thus we can conclude the proof using \eqref{eq: ns derivative bound} and $|{\bf n}_s| \geq (r_1r_2e^{s})^{1/2}$.
\end{proof}

Now we turn to consider the integral
  \begin{equation*}
 \int_0^\infty e^{\pm i|{\bf n}_s||\xi|}e^{-s\nu_k}\chi(s)\, ds 
 = \int_0^\infty \frac{\partial}{\partial s}\big(e^{\pm i|{\bf n}_s||\xi|}\big)
 ((\pm i|\xi|\frac{\partial}{\partial s}|{\bf n}_s|)^{-1} e^{-s\nu_k}\chi(s)) \, ds.
\end{equation*}
Integrating by parts gives
\begin{equation*}
 \mp \int_0^\infty e^{\pm i|{\bf n}_s||\xi|}\frac{\partial}{\partial s}\Big(\big( i|\xi|\frac{d|{\bf n}_s|}{ds}\big)^{-1}e^{-s\nu_k}\chi(s)\Big)\, ds.
\end{equation*}

Since $|\xi|\sim 1$ by the support condition of $\varphi$ and we only concern those $\nu_k$ that are uniformly bounded by an absolute constant, we will ignore those factors when we discuss the upper bound of the integrand below.
Integrating by parts in $s$ for $K$-times, the integrand becomes a sum of terms of the form (modulo uniformly bounded coefficients)
\begin{equation*}
   (\frac{\partial}{\partial s}|{\bf n}_s|)^{-K}F(s),
\end{equation*}
where $F(s)$ is a product of $\frac{ \partial^k_s |{\bf n}_s| }{|{\bf n}_s|}  ,e^{-s\nu_k},\chi_A(s)$ and their derivatives. Factors like $\frac{ \partial^k_s |{\bf n}_s| }{|{\bf n}_s|}$ are bounded using Lemma~\ref{lemma: 1/ns derivative bound} and the rest is straightforward. In sum, the integrand is bounded by
\begin{equation*}
(r_1r_2e^{s})^{-K/2}e^{-s\nu_k}.
\end{equation*}

Combining this with \eqref{est:W-n} ($m=0$) and \eqref{fact:ns}, we obtain
\begin{equation*}
\begin{split}
&|\mathsf{K}_1(t; x, y)| \lesssim  (r_1r_2)^{-\frac{n-2}2}\sum_{\{k \in\mathbb{N}: \; \nu_k\leq \frac{n-2}{2} \}}\nu_k^K \int_0^\infty  (1+|{\bf n}_s|)^{-1/2}\chi(s) (r_1r_2e^s)^{-K/2} e^{-s\nu_k} ds\\
&\lesssim  (1+|t|)^{-\frac{n-1}2}\int_{||{\bf n}_s|-t| \leq c|t|} (1+|{\bf n}_s|)^{\frac{n-2}2} (r_1r_2e^s)^{-K/2} e^{-s\nu_0} ds\lesssim (1+|t|)^{-\frac{n-1}2},
\end{split}
\end{equation*}
if $K$ is larger than $n-2$.

In summary, we have shown that all parts of $I_R$ are uniformly bounded by $O((1+|t|)^{-\frac{n-1}2})$ when $r_1, r_2 \gg 1$, concluding the proof.

%%%%%%%%%%%%%%%%%%%%%%%%%%%%%%%%%%%%%
\end{proof}

\begin{proof}[The contribution of $I_P$] 
Next we recall \eqref{eq:IP,definition} and consider the term associated to $K_N(s; \hat{x}, \hat{y})$.
For the rest of this section, we only consider the term in \eqref{KN1} that is associated to $\mk{d} = d_h(\hat{x},\hat{y})$ for notational convenience. Terms associated to all other $\mk{d}$ can be bounded by the same proof. In fact the proof is simpler for $\mk{d} \neq d_h(\hat{x},\hat{y})$ since it will have an absolute lower bound, and we don't need to consider cases like $\mk{d}(\hat{x},\hat{y}) \lesssim z^{-1/2}$ where $z$ is given below.

Recall \eqref{KN1}, we only need to estimate 
\begin{equation}\label{osi-1}
\begin{split}
&(r_1r_2)^{-\frac{n-2}2}  \Big|\int_0^\pi W(t, {\bf m}_s) \int_{0}^\infty b_\pm(\rho d_h) e^{\pm i \rho d_h} a(s, \hat{x}, \hat{y}; \rho) \cos(s \rho) \rho^{n-2}d\rho \, ds\Big| \\ & \lesssim (1+|t|)^{-(n-1)/2},
 \end{split}
 \end{equation}
 where $a(s, \hat{x}, \hat{y}; \rho)$ and $b_\pm(\rho d_h)$ satisfy \eqref{a} and \eqref{b+-} respectively, and $a$ is supported in $\{ s \leq \pi - \frac{\delta_0}{2} \}$.

The high level idea of the proof is using various dyadic decompositions in different regimes around the propagating Lagrangian (the analogue of $\mathscr{L}_\pm$ in \eqref{eq: propagating lagrangians} over $\mathbb{R}_t \times X \times X$).  
To this end , let us fix a bump function
$\beta\in C^\infty_0((1/2,2))$ satisfying
\begin{equation}\label{beta-d}
\sum_{\ell=-\infty}^\infty \beta(2^{-\ell} s)=1, \quad s>0,
\end{equation}
and we set
$$\beta_{J}(s)=\sum_{\ell\le J} \beta(2^{-\ell}s)
\in C^\infty_0((0,2^{J+1})),$$
for $J \in \mathbb{N}$ to be determined. \vspace{0.1cm}

To prove \eqref{osi-1}, we divide into two parts considering $|t|\geq 1$ and $|t|\leq 1$, respectively. \vspace{0.05cm}

\textbf{Part 1.} We first consider $|t|\geq 1$ and set  $z:=\frac{r_1r_2}{|t|}$, and then we further consider two cases depending on how large $z^{1/2}d_h$ is.

 \textbf{Case 1}. $d_h(\hat{x}, \hat{y})\leq C_1 z^{-\frac12}$.   In this case, we take $J$ large enough so that $2^{J-1}\geq 2C_1$ and we want to show that 
 \begin{equation}\label{osi-1<b}
 \begin{split}
z^{-\frac{n-2}2}
&\Big|\int_0^\pi W(t, {\bf m}_s) \Big(\beta_{J}(z^{1/2} s)+\sum_{j\ge J+1}\beta(2^{-j}z^{1/2}s)\Big) \\
&\times \int_{0}^\infty b_\pm(\rho d_h) e^{\pm i \rho d_h} a(s, \hat{x}, \hat{y}; \rho) \cos(s \rho) \rho^{n-2}d\rho\, ds \Big|\lesssim (1+|t|)^{-1/2}.
\end{split}
 \end{equation}
 For the term associated with $\beta_J$, we have $|s|\lesssim z^{-\frac12}\ll 1$ due to the compact support of $\beta_J$. If we also have $\rho\le 4z^{1/2}$, since \eqref{est:W} with $K=m=0$, thus the integral in \eqref{osi-1<b} with $\beta_J$ is always bounded by 
  \begin{equation}\label{rho-l}
 \begin{split}
& z^{-\frac{n-2}2}
\int_{|s|\lesssim z^{-\frac12}} \big(1+|t\pm |{\bf m}_s||\big)^{-N}(1+|{\bf m}_s|)^{-1/2} ds
\int_{\rho\leq 4z^{\frac12}} \rho^{n-2}d\rho.
\end{split}
 \end{equation}
 If either $|t|\ll |{\bf m}_s|$ or $|t|\gg |{\bf m}_s|$,  we have $|t\pm |{\bf m}_s||\geq c|t|$ for some constant $c$, and then we choose $N$ large enough to prove 
\begin{equation*}
\eqref{rho-l}\lesssim z^{-\frac{n-2}2}  z^{-1/2} z^{\frac{n-1}{2}} (1+|t|)^{-N}\lesssim  (1+|t|)^{-N},
\end{equation*}
which implies the $(1+|t|)^{-(n-1)/2}$-bound.
Otherwise, we have $|t| \sim |{\bf m}_s|$ and we obtain
\begin{equation*}
\eqref{rho-l} \lesssim (z)^{-\frac{n-2}2} z^{-1/2} z^{\frac{n-1}{2}} (1+|t|)^{-1/2}\lesssim (1+|t|)^{-1/2},
\end{equation*}
which proves \eqref{osi-1<b} and in turn gives \eqref{est:dispersive'} again. 
 
On the other hand, for the part $\rho\ge 4z^{1/2}$, we do integration by parts in $s$. Notice that the terms at the boundary $(s=0, \pi)$ vanish, then each time we gain a factor of $\rho^{-1}$ from the function $\cos(s \rho)$. Next we consider factors introduced by differentiating other factors of the integrand. 
Suppose we integrate by parts for $m$ times. 
When the derivative hits $a(s,\hat{x},\hat{y},\rho)$, it has the same symbol order before $s$-differentiation.
When the derivative hits $\beta_J(z^{1/2}s)$, it produces a $z^{1/2}$-factor. When the derivative hits $W(t,{\bf m}_s)$, then it can be estimated by \eqref{est:W}. A further observation we need in the current circumstance is that 
\begin{equation} \label{est:r1r2s/ms}
\frac{r_1r_2\sin s}{|{\bf m}_s|}=\Big(\frac{|t|}{|{\bf m}_s|}\Big)z\sin s \lesssim \frac{|t|}{|{\bf m}_s|}z^{1/2},
\end{equation}
since $s \lesssim z^{-1/2}$ on the support of $\beta_J(z^{1/2}s)$. 

%%%%%%%%%%%%%%%%%%%%%%%%%%%%%%%%%%%%%%%%%%%%%%%%

Therefore, taking $K=m$ in \eqref{est:W},  we have
 \begin{equation} \label{eq:derivative-Wms-small-s}
 \begin{split} \Big|\Big(\frac{\partial}{\partial s}\Big)^m\Big(W(t, {\bf m}_s)  \beta_{J}(z^{1/2} s)\Big)\Big|
\leq C_{m,N} z^\frac{m}2\big(1+|t\pm |{\bf m}_s||\big)^{-N}(1+|{\bf m}_s|)^{-1/2}.
 \end{split}
 \end{equation}
The same type estimates for $\Big(\frac{\partial}{\partial s}\Big)^k\Big(W(t, {\bf m}_s)  \beta_{J}(z^{1/2} s)\Big)$ with $k \leq m$ also holds, with $z^{k/2}$ on the right hand side and can be absorbed into this leading order term. For the rest of the paper, we only discuss this leading part in the proof.

So, after integration by parts $m$ times for $m \ge n$, the integral in \eqref{osi-1<b} is bounded by 
 \begin{equation}\label{osi-1<b'}
 \begin{split}
 z^{-\frac{n-2}2} z^{\frac m2}&\int_{|s|\lesssim z^{-\frac12}} \big(1+|t\pm |{\bf m}_s||\big)^{-N}(1+|{\bf m}_s|)^{-1/2} ds  \int_{4z^{1/2}}^\infty (1+\rho d_h)^{-\frac{n-2}2}\rho^{n-2-m}d\rho.
 \end{split}
 \end{equation}
  If either $|t|\ll |{\bf m}_s|$ or $|t|\gg |{\bf m}_s|$,  we have $|t\pm |{\bf m}_s||\geq c|t|$ for some small constant
$c$, and then we choose $m=K$ and $N$ large enough to prove 
\begin{equation*}
\eqref{osi-1<b'}\lesssim z^{-\frac{n-2}2}  z^{-1/2} z^{\frac m2} z^{\frac{n-1-m}{2}} (1+|t|)^{-N}\lesssim  (1+|t|)^{-N},
\end{equation*}
which implies \eqref{est:dispersive'}.
Otherwise, $|t|\sim |{\bf m}_s|$,  we obtain
\begin{equation*}
\eqref{osi-1<b'}\lesssim z^{-\frac{n-2}2} z^{-1/2} z^{\frac m2} z^{\frac{n-1-m}{2}} (1+|t|)^{-1/2}\lesssim (1+|t|)^{-1/2}.
\end{equation*}
 In sum, we have proved
\begin{equation}
\begin{split}
&z^{-\frac{n-2}2}
\Big|\int_0^\pi W(t, {\bf m}_s)  \beta_{J}(z^{1/2} s)\\
&\times \int_{0}^\infty b_\pm(\rho d_h) e^{\pm i \rho d_h} a(s, \hat{x}, \hat{y}; \rho) \cos(s \rho) \rho^{n-2}d\rho ds\Big| \lesssim (1+|t|)^{-1/2}.
\end{split}
\end{equation}

 For the terms with $\beta(2^{-j}z^{1/2}s), j \geq J+1$, we have $ 2^{j-1}z^{-1/2} \leq s \leq 2^{j+1}z^{-1/2}$ and $2^j\lesssim z^{1/2}$ on the support of this $\beta-$factor.  In this case, we will show that 
   \begin{equation}\label{beta-j-case1}
 \begin{split}
 &z^{-\frac{n-2}2}
\Big|\int_0^\pi W(t, {\bf m}_s) \beta(2^{-j}z^{1/2}s)\\
&\times \int_{0}^\infty b_\pm(\rho d_h) e^{\pm i \rho d_h} a(s, \hat{x}, \hat{y}; \rho) \cos(s \rho) \rho^{n-2}d\rho ds\Big| \lesssim 2^{-j(n-2)}(1+|t|)^{-1/2},
\end{split}
 \end{equation}
which would give us desired bounds after summing over $j$ when $n\geq3$.
 For the contribution from the part $\rho\le 2^{-j}z^{1/2}$, we do not do any integration by parts and the integral in \eqref{beta-j-case1} is always bounded by 
 \begin{equation*} 
 \begin{split}
& z^{-\frac{n-2}2} \int_{|s|\sim z^{-\frac12}2^j} \big(1+|t\pm |{\bf m}_s||\big)^{-N}(1+|{\bf m}_s|)^{-1/2} ds\, (2^{-j}z^{\frac{1}{2}})^{n-1}\\
 &\lesssim z^{-\frac{n-2}2} (z^{-\frac12}2^j) (2^{-j}z^{\frac{1}{2}})^{n-1} \big((1+|t|)^{-1/2} +(1+|t|)^{-N}\big)
\\ & \lesssim 2^{-j(n-2)}(1+|t|)^{-1/2}.
 \end{split}
 \end{equation*}

On the other hand, if we have $\rho\ge 2^{-j}z^{1/2}$, we write $\cos(s \rho)=\frac12\big(e^{is\rho}+e^{-is\rho}\big)$, then we do integration by parts in $d\rho$ instead\footnote{To rigorously justify the argument near the boundary at $\rho=+\infty$, one may further introduce a dyadic decomposition in $\rho$ to localize the analysis. The boundary term at $\rho=+\infty$ can be dropped since this equality is interpreted as for oscillatory integrals and one only need to pair with functions with sufficient decay in $\rho$.}, 
then each time we gain a factor of $\rho^{-1}$, and we at most lose a factor of $(s \pm d_h)^{-1}$. Recalling that $J$ is large enough so that $2^{J-2}$ is larger than $C_1$, then we have
\begin{equation*}
|s\pm d_h|^{-1}\lesssim \Big((2^{j-1}-C_1)z^{-\frac12}\Big)^{-1}\lesssim \Big((2^{j-2}+2^{J-2}-C_1)z^{-\frac12}\Big)^{-1}\sim 2^{-j}z^{\frac12}.
\end{equation*} 
So after integration by parts $m$ times for $m \geq n$, the integral in \eqref{beta-j-case1} is bounded by 
\begin{equation}\nonumber
 \begin{split}
& z^{-\frac{n-2}2} (z^{-\frac12}2^j)  \big(2^{-j}z^{\frac12}\big)^{m} \big((1+|t|)^{-1/2} +(1+|t|)^{-N}\big) \int_{2^{-j}z^{1/2}}^\infty \rho^{n-2-m}d\rho\\
&\lesssim 2^{-j(n-2)}(1+|t|)^{-1/2}.
\end{split}
  \end{equation}
 \bigskip
 
\textbf{Case 2}. $d_h(\hat{x}, \hat{y})\geq C_1 z^{-\frac12}$. 
In this case, taking $J=0$, we will show that 
 \begin{equation}\label{osi-1>b}
 \begin{split}
&z^{-\frac{n-2}2}
\Big|\int_0^\pi W(t, {\bf m}_s) \Big(\beta_{0}(zd_h |s-d_h|)+\sum_{j\ge 1}\beta(2^{-j}zd_h |s-d_h|)\Big) \\
&\times \int_{0}^\infty b_\pm(\rho d_h) e^{\pm i \rho d_h} a(s, \hat{x}, \hat{y}; \rho) \cos(s \rho) \rho^{n-2}d\rho\, ds \Big|\lesssim (1+|t|)^{-1/2},
\end{split}
 \end{equation}
where $\beta_0$ and $\beta$ are same to the above ones \eqref{beta-d}.

For the term associated with $\beta_0$, we have $|s-d_h|\leq (zd_h)^{-1}\lesssim z^{-\frac12}$ due to the compact support of $\beta_0$. 
For the $\rho$-integral, we first consider the part with $\rho\le zd_h$. By \eqref{est:W} with $m=0$, the integral in \eqref{osi-1>b} with $\beta_0$ is always bounded by 
  \begin{equation}
 \begin{split}
& z^{-\frac{n-2}2}
\int_{|s-d_h|\lesssim (zd_h)^{-1}}\big(1+|t\pm |{\bf m}_s||\big)^{-N}(1+|{\bf m}_s|)^{-1/2}   ds
\int_{\rho\leq z d_h} (1+\rho d_h)^{-\frac{n-2}2} \rho^{n-2}d\rho 
\\&\lesssim z^{-\frac{n-2}2} (zd_h)^{-1} (zd_h)^{\frac{n-2}{2}+1} d_h^{-\frac{n-2}2}\big((1+|t|)^{-1/2} +(1+|t|)^{-N}\big) \lesssim (1+|t|)^{-\frac12}.
\end{split}
 \end{equation}
 On the other hand, if we have $\rho\ge z d_h$, we do integration by parts in $ds$. Due to the support of $a(s,\hat{x},\hat{y},\rho)$, the term at the boundary $s=\pi$ still vanishes. 
 While at $s=0$, the boundary term also vanishes because on the support of $\beta_0(zd_h|s-d_h|)$, one has $|s-d_h|\leq 2(z d_h)^{-1}\leq  2C_1^{-1} z^{-1/2}$ which implies $s\geq C_1\big(1- 2C_1^{-2}\big)z^{-1/2}>0$ if $C_1>2$.
So each time we gain a factor of $\rho^{-1}$ from the function $\cos(s \rho)$. 
To estimate the factor introduced by differentiating in $s$, we use \eqref{est:W} with $K=m$ and the same proof of \eqref{eq:derivative-Wms-small-s}, with the only difference being that now we instead control $z \sin s$ using
\begin{equation*}
z\sin s\lesssim z (d_h+z^{-\frac12})\lesssim z d_h,
\end{equation*}
and the factor introduced when differentiating $\beta_0$-term is also $zd_h$ now. In sum, we have
 $$\Big|\Big(\frac{\partial}{\partial s}\Big)^m\Big(W(t, {\bf m}_s)\beta_{0}(z d_h |s-d_h|)\Big)\Big|\leq C_{m,N} (z d_h)^m\big(1+|t\pm |{\bf m}_s||\big)^{-N}(1+|{\bf m}_s|)^{-1/2}.$$
So after integration by parts $m$ times for $m>n/2$, 
  the integral in \eqref{osi-1>b} is bounded by 
  \begin{equation*}
 \begin{split}
& z^{-\frac{n-2}2} \int_{|s-d_h|\lesssim (zd_h)^{-1}}\big(1+|t\pm |{\bf m}_s||\big)^{-N}(1+|{\bf m}_s|)^{-1/2}   ds  (z d_h)^{m}  d_h^{-\frac{n-2}2}\int_{z d_h}^\infty \rho^{\frac{n-2}2-m}d\rho
 \\&\lesssim z^{-\frac{n-2}2} (zd_h)^{-1}(z d_h)^{m}  d_h^{-\frac{n-2}2} (z d_h)^{\frac{n}2-m}\big((1+|t|)^{-1/2} +(1+|t|)^{-N}\big)\\
 &  \lesssim (1+|t|)^{-\frac{1}2}.
 \end{split}
 \end{equation*}
In sum, we have proved
 \begin{equation}
 \begin{split}
&z^{-\frac{n-2}2}
\Big|\int_0^\pi W(t, {\bf m}_s) \beta_{0}(z d_h |s-d_h|)\\
&\times \int_{0}^\infty b_\pm(\rho d_h) e^{\pm i \rho d_h} a(s, \hat{x}, \hat{y}; \rho) \cos(s \rho) \rho^{n-2}d\rho ds\Big| \lesssim (1+|t|)^{-\frac{1}2}.
\end{split}
 \end{equation}
 
For terms associated with $\beta(2^{-j}z d_h |s-d_h|), j \geq 1$, we have $|s-d_h| \approx 2^{j}(z d_h)^{-1}$, due to the support condition of $\beta$, and $2^j\lesssim z d_h$ since $s,d_h$ are bounded. In this case, we will show that
\begin{equation}\label{beta-j'}
\begin{split}
 &z^{-\frac{n-2}2}
\Big|\int_0^\pi W(t, {\bf m}_s)  \beta(2^{-j}z d_h |s-d_h|)\\
&\times \int_{0}^\infty b_\pm(\rho d_h) e^{\pm i \rho d_h} a(s, \hat{x}, \hat{y}; \rho) \cos(s \rho) \rho^{n-2}d\rho ds\Big| \lesssim 2^{-j\frac{n-2}2}(1+|t|)^{-\frac{1}2},
\end{split}
 \end{equation}
which would give us desired bounds \eqref{osi-1>b} after summing over $j\geq1$.
 For the part of the $\rho$-integral with $\rho\le 2^{-j}z d_h$, we do not do any integration by parts, the integral in \eqref{beta-j'} is always bounded by 
 \begin{equation}\nonumber 
 \begin{split}
& z^{-\frac{n-2}2} \int_{|s-d_h|\sim 2^j(zd_h)^{-1}}\big(1+|t\pm |{\bf m}_s||\big)^{-N}(1+|{\bf m}_s|)^{-1/2} \,ds \int_{\rho\leq 2^{-j}z d_h}(1+\rho d_h)^{-\frac{n-2}2} \rho^{n-2}\, d\rho\\
 &\lesssim z^{-\frac{n-2}2} ((zd_h)^{-1}2^j) (2^{-j}z d_h)^{\frac{n-2}2+1} d_h^{-\frac{n-2}2} \big((1+|t|)^{-1/2} +(1+|t|)^{-N}\big)\\
 & \lesssim 2^{-j\frac{n-2}2}(1+|t|)^{-\frac{1}2}.
 \end{split}
 \end{equation}

On the other hand, if we have $\rho\ge 2^{-j}z d_h$, we write $\cos(s \rho)=\frac12\big(e^{is\rho}+e^{-is\rho}\big)$, then we do integration by parts in $d\rho$ again, then each time we gain a factor of $\rho^{-1}$, and we at most lose a factor of  
$$|s \pm d_h|^{-1}\lesssim  2^{-j}z d_h.$$ 
So after integration by parts $m$ times for $m > \frac{n}{2}$, the integral in \eqref{beta-j'} is bounded by 
 \begin{equation}\nonumber
 \begin{split}
&z^{-\frac{n-2}2}   \big(2^{-j}z d_h\big)^{m} \int_{|s-d_h|\sim 2^j(zd_h)^{-1}}\big(1+|t\pm |{\bf m}_s||\big)^{-N}(1+|{\bf m}_s|)^{-1/2} \,ds \int_{2^{-j}z d_h}^\infty \rho^{\frac{n-2}2-m} d_h^{-\frac{n-2}2}d\rho\\
&\lesssim r_1^{-\frac{n-2}2} (z d_h)^{-\frac{n-2}2-1} 2^j  \big(2^{-j}z d_h\big)^{m} \big(2^{-j}z d_h\big)^{\frac{n-2}2+1-m}  \big((1+|t|)^{-1/2} +(1+|t|)^{-N}\big)\\
 & \lesssim 2^{-j\frac{n-2}2}(1+|t|)^{-\frac{1}2}.
\end{split}
  \end{equation}
  Therefore we have proved \eqref{osi-1}. \vspace{0.2cm}
  
  \textbf{Part 2:} We next consider $|t|\leq 1$ and we  instead set  $z:=r_1r_2$. We only sketch the proof of this part since it is almost the same as above.

 \textbf{Case 1}. $d_h(\hat{x}, \hat{y})\leq C_1 z^{-\frac12}$.   In this case, we take $J$ large enough so that $2^{J-1}\geq 2C_1$ and we want to show that 
 \begin{equation}\label{osi-1<b<}
 \begin{split}
z^{-\frac{n-2}2}
&\Big|\int_0^\pi W(t, {\bf m}_s)  \Big(\beta_{J}(z^{1/2} s)+\sum_{j\ge J+1}\beta(2^{-j}z^{1/2}s)\Big) \\
&\times \int_{0}^\infty b_\pm(\rho d_h) e^{\pm i \rho d_h} a(s, \hat{x}, \hat{y}; \rho) \cos(s \rho) \rho^{n-2}d\rho\, ds \Big|\lesssim 1.
\end{split}
 \end{equation}
The similar argument as above can be applied to prove this. The only difference occurs when we treat the case $\rho\ge 4z^{1/2}$ and we do integration by parts in $ds$ to gain a $\rho^{-1}$ factor from $\cos(s\rho)$. The factors introduced by differentiating other functions in $s$ is estimated in the same (in fact simpler, as we don't concern $|t|$ now) manner as in the proof of \eqref{eq:derivative-Wms-small-s} using \eqref{est:W} and we point out minor difference below.
Concretely, when $|{\bf m}_s| \gtrsim 1$, noticing that $s \lesssim z^{-1/2}$ on the support of $\beta_J(z^{1/2}s)$, thus 
\begin{equation} \label{eq:small-s-bound3}
\frac{r_1r_2\sin s}{|{\bf m}_s|}=\Big(\frac{1}{|{\bf m}_s|}\Big)z\sin s \lesssim z^{1/2}. 
\end{equation}  
When $ |{\bf m}_s|\ll 1$,  by \eqref{fact:ms}, we must have that $r_1\sim r_2$ and $s\ll 1/r_1$ which implies $\frac{r_1r_2\sin s}{|{\bf m}_s|}\sim r_1\sim z^{1/2}$. So the entire factor is when integrate by parts for $m$-times is controlled by $z^{m/2}$ and the rest of the proof is the same.
 \bigskip
 
   \textbf{Case 2}. $d_h(\hat{x}, \hat{y})\geq C_1 z^{-\frac12}$. 
   In this case, taking $J=0$,
  we will show that 
 \begin{equation}\label{osi-1>b<}
 \begin{split}
&z^{-\frac{n-2}2}
\Big|\int_0^\pi W(t, {\bf m}_s)  \Big(\beta_{0}(zd_h |s-d_h|)+\sum_{j\ge 1}\beta(2^{-j}z d_h |s-d_h|)\Big) \\
&\times \int_{0}^\infty b_\pm(\rho d_h) e^{\pm i \rho d_h} a(s, \hat{x}, \hat{y}; \rho) \cos(s \rho) \rho^{n-2}d\rho\, ds \Big|\lesssim 1,
\end{split}
 \end{equation}
where $\beta_0$ and $\beta$ are same to the above ones \eqref{beta-d}. Similarly, the above argument proves this and the only difference happens when we have $\rho\ge z d_h$ and do integration by parts in $ds$.
Indeed, due to the support of $\beta_0$,  one has $|s-d_h|\leq 2(z d_h)^{-1}\leq  2C_1^{-1} z^{-1/2}$ which implies $s\geq C_1\big(1- 2C_1^{-2}\big)z^{-1/2}>0$ if  $C_1>2$.
So each time we gain a factor of $\rho^{-1}$ from the function $\cos(s \rho)$, and again the factor introduced by differentiating $W(t, {\bf m}_s  )\beta_{0}(zd_h |s-d_h|)$ for $m$-times is bounded by 
\begin{equation} \label{eq:small-s-DW-bound4}
|(\frac{\partial}{\partial s})^m \Big(W(t, {\bf m}_s)\beta_{0}(zd_h |s-d_h|) \Big)| \lesssim (zd_h)^m.
\end{equation} 
This follows from the same proof as \eqref{eq:derivative-Wms-small-s}, with the only minor difference being the same as in the discussion after \eqref{eq:small-s-bound3}, which we sketch again here.
When $|{\bf m}_s| \gtrsim 1$, we have
\begin{equation*}
\frac{r_1r_2\sin s}{|{\bf m}_s|}=\Big(\frac{1}{|{\bf m}_s|}\Big)z\sin s \lesssim z d_h,
\end{equation*}
where the last inequality follows from (again, by the support condition of $\beta_0$ and $d_h \gtrsim z^{-1/2}$ in the current case):
$$ z\sin s\lesssim z (d_h+z^{-\frac12})\lesssim z d_h. $$
When $1\gg |{\bf m}_s|$,  we must have that $r_1\sim r_2$ and $s\ll 1/r_1$ which implies $\frac{r_1r_2\sin s}{|{\bf m}_s|}\sim r_1\sim z^{1/2}\lesssim zd_h$. Then we can apply \eqref{est:W} as before to obtain \eqref{eq:small-s-DW-bound4}. And the rest of the proof of \eqref{osi-1>b<} is the same as previous cases. And this finishes the proof of \eqref{est:dispersive}.

 \vspace{0.2cm}
  \end{proof}

\section{The Strichartz estimates}\label{sec:striW}

In this section, we prove the Strichartz estimates in Theorem
\ref{thm:stri-wave}. To obtain the Strichartz estimates, we need a
variant of Keel-Tao's \cite{KT} abstract Strichartz estimate for wave
equation.

\subsection{Abstract Strichartz estimates}
We first recall a variant of the abstract Keel-Tao's Strichartz estimates, which is proven in \cite[Proposition~6.4]{CDYZ}:
\begin{proposition}\label{prop:KT}
Let $(X,\mathcal{M},\mu)$ be a $\sigma$-finite measured space and
$U: \mathbb{R}\rightarrow B(L^2(X,\mathcal{M},\mu))$ be a weakly
measurable map satisfying, for some constants $C$, $\kappa\geq0$,
$\sigma, h>0$,
\begin{equation}\label{md-1}
\begin{split}
\|U(t)\|_{L^2\rightarrow L^2}&\leq C,\quad t\in \mathbb{R},\\
\|U(t)U(s)^*f\|_{L^{p_0}}&\leq
Ch^{-\kappa(1-\frac2{p_0})}(h+|t-s|)^{-\sigma(1-\frac2{p_0})}\|f\|_{L^{p_0'}},\quad 2\leq p_0\leq +\infty.
\end{split}
\end{equation}
Then for every pair $q,p\in[2,\infty]$ such that $(q, p,\sigma)\neq
(2,\infty,1)$ and
\begin{equation*}
\frac{1}{q}+\frac{\sigma}{p}\leq\frac\sigma 2,\quad 2\leq p\leq p_0,
\end{equation*}
there exists a constant $\tilde{C}$ depending only on $C$, $\sigma$,
$q$ and $p$ such that
\begin{equation*}
\Big(\int_{\mathbb{R}}\|U(t) u_0\|_{L^p}^q dt\Big)^{\frac1q}\leq \tilde{C}
\Lambda(h)\|u_0\|_{L^2}
\end{equation*}
where $\Lambda(h)=h^{-(\kappa+\sigma)(\frac12-\frac1p)+\frac1q}$.
\end{proposition}

\begin{remark} The only minor difference between \cite{KT}  and here is that one needs to restrict to $p\leq p_0$. So the bilinear argument in \cite{KT} still works.
\end{remark}

For the $Q_j$ ($j=1,\cdots \mathsf{F}$) given in \eqref{Id-p-Q} and $\varphi$ in Theorem \ref{thm:dispersive}, we define the microlocalized half wave operator
\begin{equation}\label{Ukj}
U_{k,j}(t)= Q_j \varphi(2^{-k}\sqrt{\mathcal{L}_{{\A},a}}) e^{it\sqrt{\mathcal{L}_{{\A},a}}}, \quad k\in\Z, 
\end{equation}
then by dual, it gives
\begin{equation}
U^*_{k,j}(t)= e^{-it\sqrt{\mathcal{L}_{{\A},a}}} \varphi(2^{-k}\sqrt{\mathcal{L}_{{\A},a}}) Q^*_j .
\end{equation}
Hence it shows
\begin{equation}\label{equ:UU*}
U_{k,j}(t)U^*_{k,j}(s)=Q_j  \varphi(2^{-k}\sqrt{\mathcal{L}_{{\A},a}}) e^{i(t-s)\sqrt{\mathcal{L}_{{\A},a}}}  \varphi(2^{-k}\sqrt{\mathcal{L}_{{\A},a}}) Q^*_j.
\end{equation}
As a consequence of Theorem \ref{thm:dispersive}, we shall have
\begin{proposition}\label{prop:Dispersive} Let $U_{k,j}(t)$ be defined in \eqref{Ukj} and let $\alpha$ be in \eqref{def:alpha}.
Then there exists a constant $C$ independent of $t, s$ for all
$ k\in\Z$ such that:

$\bullet$ if $\alpha\geq0$, then for $2\le p\leq +\infty$, it holds
\begin{equation}\label{est:Dis}
\|U_{k,j}(t)U^*_{k,j}(s)\|_{L^{p'}(\cone)\rightarrow L^p(\cone)}\leq C
2^{k\frac{n+1}2(1-\frac2{p})}(2^{-k}+|t-s|)^{-\frac{n-1}2(1-\frac2{p})};
\end{equation}

$\bullet$ if $-(n-2)/2<\alpha<0$, then \eqref{est:Dis} still holds for $2\leq p<p(\alpha)$ with $p(\alpha)$ in \eqref{def:q-alpha}.
\end{proposition}

\begin{proof}[The proof of Proposition \ref{prop:Dispersive}] Without loss of generality, we assume $s=0$ for simplicity. By using Theorem \ref{thm:dispersive} and \eqref{equ:UU*}, we have that
\begin{equation}\label{est:dispersive<'}
 \begin{split}
\big| U_{k,j}(t)U^*_{k,j}(0) (x,y)\big|\leq &C2^{kn}\big(1+2^k|t|\big)^{-\frac{n-1}2} \\
&\times\begin{cases}(2^{2k}r_1r_2)^{\nu_0-\frac{n-2}2},\quad &2^k r_1, 2^k r_2\lesssim 1;\\
(2^k r_1)^{\nu_0-\frac{n-2}2},\quad &2^k r_1\lesssim 1\ll  2^k r_2;\\
(2^k r_2)^{\nu_0-\frac{n-2}2},\quad &2^k r_2\lesssim 1\ll  2^k r_1;\\
1, \quad &2^k r_2,  2^k r_1\gg1.
 \end{cases}
\end{split}
\end{equation}
Recall $\alpha=\nu_0-(n-2)/2$ given in \eqref{def:alpha}, if $\alpha\geq0$, then
\begin{equation}\label{est:Dis'}
\|U_{k,j}(t)U^*_{k,j}(0)\|_{L^{1}(\cone)\rightarrow L^\infty(\cone)}\leq C2^{kn}\big(1+2^k|t|\big)^{-\frac{n-1}2}.
\end{equation}
Hence by interpolation,  \eqref{est:Dis} follows if we could prove 
\begin{equation}\label{est:L2}
\|U_{k,j}(t)U^*_{k,j}(0)\|_{L^{2}(\cone)\to L^2(\cone)}\leq C,
\end{equation}
which can be done by the spectral theorem. Since $\|Q_j\|_{L^2\to L^2}, \|\varphi(2^{-k}\sqrt{\mathcal{L}_{{\A},a}})\|_{L^2\to L^2}\leq C$, it suffices to consider $\|e^{it\sqrt{\mathcal{L}_{{\A},a}}} \|_{L^2\to L^2}$. 
From \eqref{funct}, we can write the kernel of $e^{it\sqrt{\mathcal{L}_{{\A},a}}}$ as
\begin{equation}
e^{it\sqrt{\mathcal{L}_{{\A},a}}}  =\sum_{k\in\N}\psi_{k}(\hat{x})\overline{\psi_{k}(\hat{y})}K_{\nu_k}(t; r_1,r_2),
\end{equation}
and
\begin{equation*}
  K_{\nu_k}(t; r_1,r_2)=(r_1r_2)^{-\frac{n-2}2}\int_0^\infty  e^{it\rho} J_{\nu_k}(r_1\rho)J_{\nu_k}(r_2\rho) \,\rho d\rho.
\end{equation*}
For $f\in L^2(\cone)$ and the Hankel transform of order $\mu$ defined in \eqref{hankel}, as in \cite[p.523]{BPSS}, we have the unitary property 
\begin{equation*}
\|\mathcal{H}_{\mu} f\|_{L^2_{\rho^{n-1}d\rho}(\R^+)}=\|f(r)\|_{L^2_{r^{n-1}dr}(\R^+)}.
\end{equation*}

%\begin{equation*}
%(\mathcal{H}_{\mu}f)(\rho, \hat{x})=\int_0^\infty (r\rho)^{-\frac{n-2}2}J_{\mu}(r\rho)f(r, \hat{x}) \,r^{n-1}dr.
%\end{equation*}

For $f\in L^2(\cone)$, we expand 
\begin{equation}\label{f:exp}
\begin{split}
f=\sum_{k\in\mathbb{N}}c_{k}(r) \psi_k(\hat{x}),
\end{split}
\end{equation}
then, by orthogonality and the unitarity of the Hankel transform, we obtain
\begin{equation*}
\begin{split}
\|e^{it\sqrt{\mathcal{L}_{{\A},a}}}  f\|_{L^2(\cone)}&=\Big(\sum_{k\in\mathbb{N}} \big\| \mathcal{H}_{\nu_k}\big(e^{-it\rho} ( \mathcal{H}_{\nu_k} c_{k})\big)(r)\big\|_{L^2_{r^{n-1}dr}}^2\Big)^{1/2}\\&
=\Big(\sum_{k\in\Z, \atop m\in\mathbb{N}} \big\| c_{k}(r)\big\|^2_{L^2_{r^{n-1}dr}}\Big)^{1/2}
=\|f\|_{L^2(\cone)}.
\end{split}
\end{equation*}
In sum, we have proved \eqref{est:Dis} when $\alpha\geq0$.
 \vspace{0.2cm}

Now we prove \eqref{est:Dis} when $-(n-2)/2<\alpha<0$. To this end, we introduce the orthogonal projections on $L^{2}$
\begin{equation}\label{def-pro1}
  P_\ell:
  L^{2}(\cone)\to     L^{2}(r^{n-1}dr)\otimes  h_{\ell}(\CS),
\end{equation}
and
\begin{equation}\label{def-pro}
  P_<:
  L^{2}(\cone)\to   
  \bigoplus_{\{k\in \mathbb{N}: \nu_\ell< (n-2)/2\}}   L^{2}(r^{n-1}dr)\otimes  h_{\ell}(\CS),
  \quad
  P_{\geq }=I-P_{<}.
\end{equation}
Here the space $h_{\ell}(\CS)$ in \eqref{hk} is the linear
span of 
$\{ \psi_\ell(\hat{x})\}$
defined in \eqref{equ:eig-Aa}. Then we can decompose the operator as
\begin{equation}\label{dec:UU*}
\begin{split}
&U_{k,j}(t)U^*_{k,j}(0)=Q_j  \varphi(2^{-k}\sqrt{\mathcal{L}_{{\A},a}}) e^{it\sqrt{\mathcal{L}_{{\A},a}}}  \varphi(2^{-k}\sqrt{\mathcal{L}_{{\A},a}}) Q^*_j\\
&=Q_j  \varphi(2^{-k}\sqrt{\mathcal{L}_{{\A},a}}) e^{it\sqrt{\mathcal{L}_{{\A},a}}} (P_<+P_{\geq}) \varphi(2^{-k}\sqrt{\mathcal{L}_{{\A},a}}) Q^*_j.
\end{split}
\end{equation}
By \eqref{funct} and \eqref{equ:knukdef}, we see that the kernels 
\begin{equation}\label{ker:W-l}
\begin{split}
 e^{it\sqrt{\mathcal{L}_{{\A},a}}}P_{<}
&=\big(r_1 r_2\big)^{-\frac{n-2}2}\sum_{\{\ell\in\mathbb{N}: \nu_\ell<(n-2)/2\}}\psi_{\ell}(\hat{x})\overline{\psi_{\ell}(\hat{y})}K_{\nu_\ell}(t,r_1,r_2),
\end{split}
\end{equation}
and 
\begin{equation}\label{ker:W-h}
\begin{split}
e^{it\sqrt{\mathcal{L}_{{\A},a}}}P_{\geq }
&=\big(r_1 r_2\big)^{-\frac{n-2}2}\sum_{\{\ell\in\mathbb{N}: \nu_\ell\geq \frac12(n-2)\}}\psi_{\ell}(\hat{x})\overline{\psi_{\ell}(\hat{y})}K_{\nu_\ell}(t,r_1,r_2).
\end{split}
\end{equation}
Since the kernel $e^{it\sqrt{\mathcal{L}_{{\A},a}}}P_{\geq }$ has a projection to large angular modes, 
thus we can repeat the argument of \eqref{est:dispersive<} and \eqref{est:dispersive} with $\nu_0 \geq (n-2)/2$ to obtain
\begin{equation*}
\begin{split}
\|Q_j  \varphi(2^{-k}\sqrt{\mathcal{L}_{{\A},a}}) e^{it\sqrt{\mathcal{L}_{{\A},a}}} P_{\geq} \varphi(2^{-k}\sqrt{\mathcal{L}_{{\A},a}}) Q^*_j\|_{L^1(\cone)\to L^\infty(\cone)}\leq C2^{kn}\big(1+2^k|t|\big)^{-\frac{n-1}2}.
\end{split}
\end{equation*}

Therefore, the same as the case $\alpha\geq 0$, we can prove \eqref{est:Dis}  for the part of \eqref{dec:UU*} associated with $P_{\geq }$ with $p\geq2$. Thus it remains to consider the part of \eqref{dec:UU*} associated with $P_{<}$, in which we are restricted at small angular modes. Due to the Weyl’s asymptotic formula (e.g. see \cite{Garding53}) 
$$
\nu^2_\ell\sim (1+\ell)^{\frac 2{n-1}},\quad \ell\geq 1,
$$ the summation the part of \eqref{dec:UU*} associated with $P_{<}$ is finite. Consequently, using \eqref{dec:UU*}, to prove \eqref{est:Dis} for  the part with $P_{<}$,
it suffices to prove, for each $\ell$ satisfying $\nu_\ell<(n-2)/2$,
\begin{equation}\label{est:UU*<}
\begin{split}
\|Q_j  \tilde{\varphi}(2^{-k}\sqrt{\mathcal{L}_{{\A},a}}) &e^{it\sqrt{\mathcal{L}_{{\A},a}}} P_\ell  Q^*_j\|_{L^{p'}(\cone)\rightarrow L^p(\cone)}\\&\leq C
2^{k\frac{n+1}2(1-\frac2{p})}(2^{-k}+|t|)^{-\frac{n-1}2(1-\frac2{p})},
\end{split}
\end{equation}
where $ \tilde{\varphi}=\varphi^2$. In the following argument, since $\tilde{\varphi}$ has the same property of $\varphi$, without confusion, we drop off the tilde above $\varphi$ for brief.   

To prove \eqref{est:UU*<}, we need the following proposition, which can be interpreted as the dispersive estimate for spherically symmetric functions.
\begin{proposition}\label{prop:est-pp'}
  Let $0<\nu\leq \frac{n-2}2$ and $\sigma(\nu)=-(n-2)/2+\nu$. Let $T_\nu$ be the operator defined as
  \begin{equation}\label{Tnu-operator}
\begin{split}
(T_{\nu}g)(t,r_1)=\int_0^\infty  K^l_{\nu}(t;r_{1},r_{2}) g(r_2)\, r^{n-1}_2 dr_2,
 \end{split}
\end{equation}
 and 
\begin{equation*}
\begin{split}
  K^l_{\nu}(t,r_1,r_2)&=(r_1r_2)^{-\frac{n-2}2}\int_0^\infty e^{it\rho}J_{\nu}(r_1\rho)J_{\nu}(r_2\rho) \varphi(\rho)\,\rho d\rho,
  \end{split}
\end{equation*}
where $\varphi$ is given in \eqref{LP-dp}.
Then, for $2\leq q<q(\sigma)$, the following estimate holds
  \begin{equation}\label{est:q-q'}
  \|T_{\nu}g\|_{L^p({r^{n-1}_1 dr_1})}\le
    C_{\nu}(1+|t|)^{-\frac{n-1}2(1-\frac2{p})}\|g\|_{L^{p'}_{r^{n-1}_2 dr_2}}.
  \end{equation}
\end{proposition}
We postpone the proof of Proposition \ref{prop:est-pp'} for a moment. For $\tilde{f}\in L^{p'}(X)$, taking $f=Q_j^* \tilde{f}$ and recalling \eqref{f:exp}, \eqref{funct} and \eqref{Tnu-operator}, we write
  \begin{equation*}
    \begin{split}
\varphi(2^{-k}\sqrt{\mathcal{L}_{{\A},a}})e^{it\sqrt{\mathcal{L}_{{\A},a}}} P_\ell f&= \psi_\ell(\hat{x}) 2^{kn} \int_0^\infty  K^l_{\nu_\ell}(2^{k}t;2^kr_{1}, 2^kr_{2}) c_\ell(r_2)\, r^{n-1}_2 dr_2\\
&= \psi_\ell(\hat{x}) \big(T_{\nu_\ell}c_\ell(2^{-k}r_2)\big)(2^{k}t, 2^kr_1).
\end{split}
\end{equation*}

Notice that $p<p(\alpha)\leq p(\sigma)$, we use the eigenfunction's estimates
\begin{equation*}
\| \psi_\ell(\hat{x})\|_{L^p(\CS)}\leq C_\ell \| \psi_\ell(\hat{x})\|_{L^{p'}(\CS)},\quad p\geq 2,
\end{equation*}
and \eqref{est:q-q'} to obtain that (recalling the expansion \eqref{f:exp}):
\begin{equation*}
\begin{split}
&\Big\|\varphi(2^{-k}\sqrt{\mathcal{L}_{{\A},a}})e^{it\sqrt{\mathcal{L}_{{\A},a}}} P_\ell f\Big\|_{L^p((\cone))}\\
&\le
C \|\big(T_{\nu_\ell}c_\ell(2^{-k}\cdot)\big)(2^{k}t, 2^kr_1)\|_{L^p_{r^{n-1}_1 dr_1}} \| \psi_\ell(\hat{x})\|_{L^p(\CS)}
    \\&\le C_{\ell}2^{kn(1-\frac 2p)}(1+2^k|t|)^{-\frac{n-1}2(1-\frac 2p)}\|c_\ell(r)\|_{L^{p'}_{r^{n-1} dr}}\| \psi_\ell(\hat{x})\|_{L^{p'}(\CS)}\\
    &\le C_\ell
2^{k\frac{n+1}2(1-\frac2{p})}(2^{-k}+|t|)^{-\frac{n-1}2(1-\frac2{p})}  \big\| P_{\ell} f\big\|_{L^{p'}(\cone)}.
    \end{split}
  \end{equation*}
  Since we are only concerning finitely many $ \psi_\ell$ such that corresponding $\nu_\ell \in (0,\frac{n-2}{2}]$, those $C_{\ell}$ are uniformly bounded.
This completes the proof of the desirable estimate \eqref{est:UU*<}. 
\end{proof}
  
\begin{proof}[The proof of  Proposition \ref{prop:est-pp'}]
Let $\chi\in \mathcal{C}_c^\infty ([0,+\infty))$ be defined as 
\begin{equation}\label{def:chi}
\chi(r)=
\begin{cases}1,\quad r\in [0, \frac12],\\
0, \quad r\in [1,+\infty)
\end{cases}
\end{equation}
and let us set $\chi^c=1-\chi$. We decompose the kernel $K^l_{\nu}(t;r_{1},r_{2})$ into four terms as follows:
  \begin{equation}
\begin{split}
K^l_{\nu}(t;r_{1},r_{2})=&\chi(r_1)K^l_{\nu}(t;r_{1},r_{2})\chi(r_2)+\chi^c(r_1)K^l_{\nu}(t;r_{1},r_{2})\chi(r_2)\\
&+\chi(r_1)K^l_{\nu}(t;r_{1},r_{2})\chi^c(r_2)+\chi^c(r_1)K^l_{\nu}(t;r_{1},r_{2})\chi^c(r_2).
 \end{split}
\end{equation}
This yields a corresponding decomposition for the operator $T_{\nu}=T^1_{\nu}+T^2_{\nu}+T^3_{\nu}+T^4_{\nu}$. We thus estimate separately the norms $\|T^j_{\nu}g\|_{L^p_{r_1^{n-1} dr_1}}$ for $j=1,2,3,4$.

Consider the term $T^1_\nu$ first. From \eqref{eq:bess1} below, one has 
  \begin{equation}
\begin{split}
|\chi(r_1)K^l_{\nu}(t;r_{1},r_{2})\chi(r_2)|\lesssim  (r_1r_2)^{\sigma} \chi(r_1)\chi(r_2).
 \end{split}
\end{equation}
Therefore, as long as $2\leq p< p(\sigma)$, if $|t|\leq 1$, we can show
  \begin{equation}\label{est:q-q'1-1}
  \begin{split}
  \|T^1_{\nu}g\|_{L^p_{r^{n-1}_1 dr_1}}&\le
    C_{\nu} \Big(\int_0^1 r^{\sigma p} r^{n-1} dr\Big)^{2/p}\|g\|_{L^{p'}_{r^{n-1}_2 dr_2}}\\
    &\le
   C_{\nu}  \|g\|_{L^{p'}_{r^{n-1}_2 dr_2}} \lesssim (1+|t|)^{-\frac{n-1}2(1-\frac2p)}  \|g\|_{L^{p'}_{r^{n-1}_2 dr_2}}.
   \end{split}
  \end{equation}
For the case that $|t|\geq 1$, we perform integration by parts in $d\rho$ to obtain 
\begin{equation}
\begin{split}
&|\chi(r_1)K^l_{\nu}(t;r_{1},r_{2})\chi(r_2)|\\
&\lesssim   \big(r_1r_2\big)^{-\frac{n-2}2} \chi(r_1)\chi(r_2)|t|^{-N}\int_0^\infty \Big| \Big(\frac{\partial}{\partial\rho}\Big)^{N}\Big(J_{\nu}(r_1\rho)J_{\nu}(r_2\rho) \varphi(\rho)\Big)\Big| d\rho \\
&\lesssim   \big(r_1r_2\big)^{\nu-\frac{n-2}2} \chi(r_1)\chi(r_2)|t|^{-N},
 \end{split}
\end{equation}
where in the last inequality we used the fact that 
$$\Big| \Big(\frac{\partial}{\partial\rho}\Big)^{N}\Big(J_{\nu}(r_1\rho)J_{\nu}(r_2\rho) \varphi(\rho)\Big)\Big| \lesssim (r_1r_2)^{\nu},$$
provided $r_1, r_2\le 1$. Finally, if $|t|\geq1$ and taking $N$ large enough, as before, we obtain
  \begin{equation}\label{est:q-q'1-2}
  \begin{split}
  \|T^1_{\nu}g\|_{L^p_{r^{n-1}_1 dr_1}}&\le
    C_{\nu} |t|^{-N} \Big(\int_0^1 r^{\sigma p} r^{n-1} dr\Big)^{2/p}\|g\|_{L^{p'}_{r^{n-1}_2 dr_2}}\\
    & \lesssim(1+|t|)^{-\frac{n-1}2(1-\frac2p)}  \|g\|_{L^{p'}_{r^{n-1}_2 dr_2}}.
   \end{split}
\end{equation}
Next we consider $T^3_{\nu}$ and $T^2_{\nu}$ can be bounded in the same manner.
Using \eqref{eq:bess3}, we are reduced to estimate two integrals
  \begin{equation}\label{I+-}
    I_{\pm}(t; r_1,r_2)=(r_1r_2)^{-\frac{n-2}2}\int_{0}^{\infty}
    \rho \varphi(\rho)J_{\nu}(r_{1}\rho)(r_{2}\rho)^{-1/2}
   e^{it\rho} e^{\pm i r_{2}\rho}j_\pm(r_{2}\rho)d \rho.
  \end{equation}
 If $|t|\leq1$, by using integration by parts and recalling $\sigma=\nu-(n-2)/2$, we obtain 
   \begin{equation*}
   \begin{split}
    I_{\pm}&\lesssim (r_1r_2)^{-\frac{n-2}2} r_2^{-\frac12-N}
    \int_{0}^{\infty}
   \Big| \Big(\frac{\partial}{\partial\rho}\Big)^{N}\Big(J_{\nu}(r_1\rho)j_\pm(r_{2}\rho)\varphi(\rho)\rho^{1/2} e^{it\rho}\Big)\Big|d\rho\\
  & \lesssim r_1^{\sigma} r_2^{-\frac{n-1}2-N}.
  \end{split}
  \end{equation*}
   Hence if $|t|\leq 1$ and  $2\leq p< p(\sigma)$, by choosing $N$ large enough, we have
      \begin{equation}
         \begin{split}
  \|T^3_{\nu}g\|_{L^p_{r^{n-1}_1 dr_1}}&\lesssim \Big(\int_0^1 r_1^{\sigma q} r^{n-1}_1 dr_1\Big)^{1/p}\Big(\int_{\frac12}^{+\infty} r_2^{-(\frac{n-1}2+N)p} r^{n-1}_2 dr_2\Big)^{1/p}\|g\|_{L^{p'}_{r^{n-1}_2 dr_2}}\\&\lesssim \|g\|_{L^{p'}_{r^{n-1}_2 dr_2}}   \lesssim (1+|t|)^{-\frac{n-1}2(1-\frac2p)}  \|g\|_{L^{p'}_{r^{n-1}_2 dr_2}}.
      \end{split}
  \end{equation}
   It remains to consider the region $|t|\geq 1$.  
  In this case,  from \eqref{I+-}, we aim to estimate the kernel 
   \begin{equation}\label{I-pm}
   \begin{split}
    I_{\pm}(t; r_1, r_2)&=\int_{0}^{\infty}
   e^{i\rho(t\pm r_2)} \tilde{j}_{\pm}(\rho, r_{1}, r_{2}) \rho^{n-1}d \rho,
   \end{split}
  \end{equation}
  where 
\begin{equation} \label{def:ta}
\tilde{j}_{\pm}(\rho, r_1, r_2)= \varphi(\rho)(r_1\rho)^{-\frac{n-2}2} J_{\nu}(r_1\rho)(r_{2}\rho)^{-\frac{n-1}2} j_{\pm}(r_{2}\rho)
\end{equation}
with $j_{\pm}$ satisfying \eqref{eq:bess4}.
Since $r_1\rho\lesssim 1$ and $\sigma=\nu-\frac{n-2}2$, therefore we obtain
\begin{equation} \label{est:symb}
\Big| \Big(\frac{\partial}{\partial\rho}\Big)^N \Big(  \tilde{j}_\pm(\rho, r_1, r_2)\Big)\Big| \lesssim  (r_{1}\rho)^{\sigma} (r_{2}\rho)^{-\frac{n-1}2} \rho^{-N}
\lesssim r_{1}^{\sigma} r_{2}^{-\frac{n-1}2} \rho^{-N},
\end{equation}
since $\rho\sim 1$ on the support of $\varphi(\rho)$. Let us define $$\Phi(\rho, \bar{r}_2)=\rho(t\pm r_2), \quad L=L(r_2)=(t\pm r_2)^{-1}\partial_\rho.$$
By \eqref{est:symb} and using the integration by parts, for any $N$, we obtain 
     \begin{equation}
    \begin{split}
     \tilde{ I}_{\pm}&\leq \Big|\int_{0}^{\infty}
   L^N\Big(e^{i\rho(t\pm r_2)}\Big) \tilde{j}_\pm(\rho, r_1, r_2) \rho^{n-1} d \rho\Big|\\
&\leq r_{1}^{\sigma} r_{2}^{-\frac{n-1}2} |t\pm r_2|^{-N} \int_{\rho\sim 1}
 \rho^{-N} \rho^{n-1} d \rho\lesssim r_{1}^{\sigma} r_{2}^{-\frac{n-1}2}  |t\pm r_2|^{-N}.
    \end{split}
  \end{equation} 
Therefore, for any $N\geq 0$, we obtain that
   \begin{equation}
    \begin{split}
    \big| I_{\pm}(t; r_1, r_2)\big|\lesssim  r_{1}^{\sigma} r_{2}^{-\frac{n-1}2} \big(1+ |t\pm r_2|\big)^{-N}.
    \end{split}
  \end{equation} 
Then,  for $|t|\geq1$, we see
      \begin{equation*}
         \begin{split}
  &\|T^3_{\nu}g\|_{L^p_{r^{n-1}_1 dr_1}}\\
  &\lesssim \Big(\int_0^1 r_1^{\sigma p} r^{n-1}_1 dr_1\Big)^{1/p}\Big(\int_{\frac12}^{+\infty} r_2^{-\frac{(n-1)p}2} r^{n-1}_2 \big(1+ |t\pm r_2|\big)^{-pN} dr_2\Big)^{1/p}\|g\|_{L^{p'}_{r^{n-1}_2 dr_2}}.
        \end{split}
  \end{equation*}

On one hand, it is easy to check that 
$$\int_0^1 r_1^{\sigma p} r^{n-1}_1 dr_1\lesssim 1$$
  provided $2\leq p<p(\sigma)=\frac{2n}{n-2-2\nu}$. On the other hand, we have
  \begin{align*}
&\int_{\frac{1}{2}}^\infty r_2^{-\frac{(n-1)p}2} r^{n-1}_2 \big(1+ |t\pm r_2|\big)^{-pN} dr_2\\
&= \int_{\{r_2\geq 1/2: |r_2-t| \geq \frac{1}{2}|t| \}} r_2^{-\frac{(n-1)p}2} r^{n-1}_2 \big(1+ |t\pm r_2|\big)^{-pN} dr_2\\
 &\qquad+ \int_{\frac{1}{2}|t|}^{ \frac{3}{2}|t| } r_2^{-\frac{(n-1)p}2} r^{n-1}_2 \big(1+ |t\pm r_2|\big)^{-pN} dr_2
\\ & \lesssim (1+|t|)^{-N}
+ |t|^{-\frac{(n-1)p}2} t^{n-1} \int_{ \frac{1}{2}|t| }^{\frac{3}{2} |t|} \big(1+ |t\pm r_2|\big)^{-pN} dr_2  \lesssim (1+|t|)^{ -\frac{n-1}{2}(p-2) }.
\end{align*}
So in sum, for $2\leq p<p(\sigma)=\frac{2n}{n-2-2\nu}, (\nu>0)$, we have proved 
   \begin{equation}
         \begin{split}
  \|T^3_{\nu}g\|_{L^p_{r^{n-1}_1 dr_1}}\lesssim (1+|t|)^{ -\frac{n-1}{2}(1-\frac2p) }\|g\|_{L^{p'}_{r^{n-1}_2 dr_2}}.
        \end{split}
  \end{equation}

\vspace{0.2cm}

  We finally deal with $T^4_{\nu}$ by modifying the argument of $T^3_{\nu}$. Using \eqref{eq:bess3} again, we are reduced to estimate the
  two integrals
  \begin{equation}\label{I+-'}
    I_{\pm}=(r_1r_2)^{-\frac{n-2}2}\int_{0}^{\infty}
    \rho \varphi(\rho)(r_1r_{2}\rho^2)^{-1/2} e^{i[t\pm  (r_{1}\pm r_2)]\rho}j_\pm(r_{1}\rho) j_\pm(r_{2}\rho)d \rho.
  \end{equation}
  If $|t|\leq1$, by using integration by parts, we obtain 
   \begin{equation*}
   \begin{split}
    I_{\pm}&\lesssim (r_1r_2)^{-\frac{n-1}2} (1+|r_1\pm r_2|)^{-N}
    \int_{0}^{\infty}
   \Big| \Big(\frac{\partial}{\partial\rho}\Big)^{N}\Big(j_\pm(r_{1}\rho) j_\pm(r_{2}\rho)\varphi(\rho) e^{it\rho^2}\Big)\Big|d\rho\\
  & \lesssim (r_1r_2)^{-\frac{n-1}2} (1+|r_1\pm r_2|)^{-N}.
  \end{split}
  \end{equation*}
 Since $r_1, r_2\geq 1/2$,  hence if $|t|\leq 1$, we have
      \begin{equation}
         \begin{split}
  \|T^4_{\nu}g\|_{L^\infty_{r^{n-1}_1 dr_1}}&\lesssim  \|g\|_{L^{1}_{r^{n-1}_2 dr_2}}   \lesssim (1+|t|)^{-\frac {n-1}2}  \|g\|_{L^{1}_{r^{n-1}_2 dr_2}}.
      \end{split}
  \end{equation}
Now we consider the region $|t|\geq 1$. As above, by using integration by parts, we obtain 
   \begin{equation*}
   \begin{split}
    I_{\pm}&\lesssim (r_1r_2)^{-\frac{n-1}2} (1+|t\pm(r_1\pm r_2)|)^{-N}
    \int_{0}^{\infty}
   \Big| \Big(\frac{\partial}{\partial\rho}\Big)^{N}\Big(j_\pm(r_{1}\rho) j_\pm(r_{2}\rho)\varphi(\rho)\Big)\Big|d\rho\\
  & \lesssim (r_1r_2)^{-\frac{n-1}2} (1+|t\pm(r_1\pm r_2)|)^{-N}.
  \end{split}
  \end{equation*}
  If either $|t|\ll |r_1\pm r_2|$ or $|t|\gg |r_1\pm r_2|$, one always has $|t\pm(r_1\pm r_2)|\geq c|t|$ with a small positive constant $c$. So we have 
     \begin{equation*}
   \begin{split}
   | I_{\pm}|\lesssim (r_1r_2)^{-\frac{n-1}2} (1+|t|)^{-N}\lesssim (1+|t|)^{-\frac{n-1}2}.
  \end{split}
  \end{equation*}
  Otherwise, $1\leq |t|\sim |r_1\pm r_2|\leq 2\max\{r_1,r_2\}$,  therefore 
       \begin{equation*}
   \begin{split}
   | I_{\pm}|\lesssim (r_1r_2)^{-\frac{n-1}2} \lesssim (1+|t|)^{-\frac{n-1}2}.
  \end{split}
  \end{equation*}
 So we have
 \begin{equation}
  \|T^4_{\nu}g\|_{L^\infty_{r^{n-1}_1 dr_1}}\lesssim (1+|t|)^{-\frac{n-1}2}\|g\|_{L^{1}_{r^{n-1}_2 dr_2}}.
  \end{equation}
  By interpolating this with the $L^2$-estimate for $T^4_{\nu}$, we obtain 
      \begin{equation}\label{est:q-q'4}
    \begin{split}
  &\|T^4_{\nu}g\|_{L^p_{r^{n-1}_1 dr_1}}\leq C (1+|t|)^{-\frac{n-1}2(1-\frac2p)}\|g\|_{L^{p'}_{r^{n-1}_2 dr_2}},\quad p\geq2.
    \end{split}
  \end{equation}
  Collecting the estimates on the terms $T^j_\nu$, yields \eqref{est:q-q'} and the proof is concluded.
\end{proof}

\subsection{The proof of homogeneous Strichartz estimates \eqref{est:Stri}} The proof is based on Proposition \ref{prop:Dispersive}. Applying the
Littlewood-Paley frequency projector $ \varphi_k(\sqrt{\LL_{{\A},a}})$ in \eqref{LP-dp} to the wave equation \eqref{eq:wave}, we obtain 
\begin{equation}\label{eq:wave'}
\partial_{t}^2u_k+\LL_{{\A},a} u_k=0, \quad u_k(0)=f_k(x),
~\partial_tu_k(0)=g_k(x),
\end{equation}
where $u_k=\varphi_k(\sqrt{\LL_{\A},a})u(t,x)$, $f_k= \varphi_k(\sqrt{\LL_{{\A},a}})u_0$ and
$g_k= \varphi_k(\sqrt{\LL_{{\A},a}})u_1$.
By the square function estimate \eqref{square} and Minkowski's inequality, we obtain 
\begin{equation}\label{LP}
\|u\|_{L^q(\R;L^p(\cone))}\lesssim
\Big(\sum_{k\in\Z}\|u_k\|^2_{L^q(\R;L^p(\cone))}\Big)^{\frac12}
\end{equation}
provided $q\geq 2$ and $2\leq p<p(\alpha)$.
Let $U(t)=e^{it\sqrt{\LL_{{\A},a}}}$ be the half wave operator, then
we write
\begin{equation}\label{sleq}
\begin{split}
u_k(t,x)
=\frac{U(t)+U(-t)}2f_k+\frac{U(t)-U(-t)}{2i\sqrt{\LL_{{\A},a}}}g_k.
\end{split}
\end{equation}
To prove the homogeneous estimates \eqref{est:Stri} in Theorem \ref{thm:stri-wave},  by \eqref{LP} and \eqref{sleq}, it suffices to show the frequency localized Strichartz estimates:
\begin{proposition}\label{prop:stri} Let
$f=\varphi_k(\sqrt{\LL_{{\A},a}})f$ for $k\in\Z$, we have
\begin{equation}\label{lstri}
\|U(t)f\|_{L^q_tL^p_x(\mathbb{R}\times \cone)}\lesssim
2^{ks}\|f\|_{L^2(\cone)},
\end{equation}
where the admissible pair $(q,p)\in \Lambda_{s,\alpha(\nu_0)}$ and $s$ satisfies \eqref{scaling}.
\end{proposition}
Now we prove this proposition. Notice that
\begin{equation*}
U(t)=\sum_{j=1}^{\mathsf{F}}\sum_{k\in\Z}U_{k,j}(t),
\end{equation*}
where $U_{k,j}(t)$ is given by \eqref{Ukj}. To prove \eqref{lstri}, it is sufficient to prove
\begin{equation*}
\|U_{k,j}(t)f\|_{L^q_t(\R:L^p(\cone))}\lesssim
2^{k[n(\frac12-\frac1p)-\frac1q]} \|f\|_{L^2(\cone)}.
\end{equation*}
This is a consequence of Proposition \ref{prop:KT}  and
\begin{equation}\label{md-1}
\begin{split}
\|U_{k,j}(t)\|_{L^2\rightarrow L^2}&\leq C,\quad t\in \mathbb{R},\\
\|U_{k,j}(t)U(s)_{k,j}^*f\|_{L^{p_0}}&\leq
C2^{k\frac{n+1}2(1-\frac2{p_0})}(2^{-k}+|t-s|)^{-\frac{n-1}2(1-\frac2{p_0})}\|f\|_{L^{p_0'}},
\end{split}
\end{equation}
where either $p_0=+\infty$ when $\alpha\geq0$ or $p_0=p(\alpha)-\epsilon$ with any $0<\epsilon\ll1$ when $-(n-2)/2<\alpha<0$.
The first one $L^2$-estimate follows from the spectrum theory, as proved \eqref{est:L2} and the second one has been proved in Proposition \ref{prop:Dispersive}.

\subsection{Inhomogeneous Strichartz estimates}
In this subsection, we prove the inhomogeneous Strichartz estimates \eqref{est:in-Stri}.

Recall that $U(t)=e^{it\sqrt{\LL_{{\A},a}}}$ is unitary from $L^2$ to $L^2$, and we have already proved that
\begin{equation}\label{est:lstri}
\|U(t)u_0\|_{L^q_t(\R;L^p(\cone))}\lesssim\|u_0\|_{\dot{H}^s_{{\A},a}}
\end{equation} holds for all $(q,p)\in  \Lambda_{s,\alpha(\nu_0)}$ satisfying \eqref{scaling}.
For $s\in\R$ and $(q,p)\in  \Lambda_{s,\alpha(\nu_0)}$,
we define the operator $T_s$ by
\begin{equation}\label{Ts}
\begin{split}
T_s: L^2(\cone)&\rightarrow L^q_t(\R;L^p(\cone)),\quad f\mapsto \LL_{{\A},a}^{-\frac
s2}e^{it\sqrt{\LL_{{\A},a}}}f.
\end{split}
\end{equation}
Then \eqref{est:lstri} shows that $T_s$ is bounded.
By duality, the operator 
\begin{equation}\label{Ts*}
\begin{split}
T^*_{1-s}: L^{\tilde{q}'}_t(\R;L^{\tilde{p}'}(\cone))\rightarrow L^2(\cone),\quad
F(\tau,x)&\mapsto \int_{\R}\LL_{{\A},a}^{\frac
{s-1}2}e^{-i\tau\sqrt{\LL_{{\A},a}}}F(\tau)d\tau,
\end{split}
\end{equation}
is also bounded, where $(\tilde{q},\tilde{p})\in  \Lambda_{1-s,\alpha(\nu_0)}$ satisfies $1-s=n(\frac12-\frac1{\tilde{p}})-\frac1{\tilde{q}}$.
Therefore, we obtain
\begin{equation*}
\Big\|\int_{\R}U(t)U^*(\tau)\LL_{{\A},a}^{-\frac12}F(\tau)d\tau\Big\|_{L^q_t(\R;L^p(\cone))}
=\big\|T_sT^*_{1-s}F\big\|_{L^q_tL^r_z}\lesssim\|F\|_{L^{\tilde{q}'}_t(\R;L^{\tilde{p}'}(\cone))}.
\end{equation*}
Since $s=n(\frac12-\frac1p)-\frac1q$ and
$1-s=n(\frac12-\frac1{\tilde{p}})-\frac1{\tilde{q}}$, thus $(q,p)\in  \Lambda_{s,\alpha(\nu_0)}$ and
$(\tilde{q},\tilde{r})\in  \Lambda_{1-s,\alpha(\nu_0)}$ satisfy \eqref{scaling}. By the
Christ-Kiselev lemma \cite{CK}, we thus obtain for $q>\tilde{q}'$,
\begin{equation}\label{non-inhomgeneous}
\begin{split}
\Big\|\int_{\tau<t}\frac{\sin{(t-\tau)\sqrt{\LL_{{\A},a}}}}
{\sqrt{\LL_{{\A},a}}}F(\tau)d\tau\Big\|_{L^q_t(\R;L^p(\cone))}\lesssim\|F\|_{L^{\tilde{q}'}_t(\R;L^{\tilde{p}'}(\cone))}.
\end{split}
\end{equation}
Notice that for all $(q,p)\in  \Lambda_{s,\alpha(\nu_0)}$ and
$(\tilde{q},\tilde{r})\in  \Lambda_{1-s,\alpha(\nu_0)}$, we must have $q>\tilde{q}'$.
Therefore we have proved inhomogeneous Strichartz estimates \eqref{est:in-Stri}
including the endpoint $q=2$.

\subsection{Optimality on the condition $p<p(\alpha)$ in \eqref{est:Stri} }
In this subsection, we construct a counterexample to show that the restriction $p<p(\alpha)$ is in fact necessary. 

\begin{proposition}[Counterexample] 
  Let $(q,p)$ be as in $\Lambda_s$, that is,  $(q,p)$ satisfies \eqref{adm} and \eqref{scaling} and let $p\geq p(\alpha)$.
  Then the Strichartz estimates \eqref{est:Stri} may fail, in the sense that there exists an initial condition $f\in \dot{H}^s_{{\A},a} $ such that
  \begin{equation}\label{counter}
  \|e^{it\sqrt{\LL_{{\A},a}}}f\|_{L^q(\R;L^p(\cone))}=\infty
  \quad\text{for any}\quad p\geq p(\alpha).
  \end{equation}
\end{proposition}

\begin{proof} 
Recall that $\nu_0$ is the positive square roof of the smallest eigenvalue of the positive operator
$P_{{\A}, a}$. By the definition of $p(\alpha)$, we only consider $0<\nu_0<(n-2)/2$, i.e $\alpha<0$. Choose initial data of the form $f=[\mathcal{H}_{\nu_0}\chi](r) \psi_0(\hat{x})$, as a Hankel transform of $\chi$ (see \eqref{hankel}), 
 where $\chi(\rho)\in\CC_c^\infty([1,2])$ takes value in $[0,1]$. Obviously,  $f\in \dot{H}^s_{{\A},a} $ due to the compact support of  $\chi$ and to the boundedness of the Hankel transform on $L^2(r^{n-1}dr)$.
  We prove that \eqref{counter} holds for this choice of $f$. Recalling \eqref{funct} and \eqref{equ:knukdef},  we obtain
 \begin{equation*}
 \begin{split}
e^{it\sqrt{\LL_{{\A},a}}}f&=\psi_{0}(\hat{x})\int_0^\infty  K_{\nu_0}(r_1,r_2) [\mathcal{H}_{\nu_0}\chi](r_2)\; r^{n-1}_2\;dr_2\\
&=\psi_{0}(\hat{x})r_1^{-\frac{n-2}2} \int_0^\infty J_{\nu_0}(r_1\rho)e^{it\rho}\chi(\rho)\rho d\rho,
\end{split}
\end{equation*}
where
\begin{equation*}
  K_{\nu_0}(r_1,r_2)=(r_1r_2)^{-\frac{n-2}2}\int_0^\infty e^{it\rho} J_{\nu_0}(r_1\rho)J_{\nu_0}(r_2\rho) \,\rho d\rho.
\end{equation*}
So we must prove that the quantity
  \begin{equation*}
    Z=\psi_{0}(\hat{x}) r^{-\frac{n-2}2} \int_0^\infty 
    J_{\nu_0}(r\rho)e^{it\rho}\chi(\rho)\rho d\rho
  \end{equation*}
  satisfies
  \begin{equation}\label{est:aim}
    \|Z\|_{L^q(\R;L^p(X))} =\infty,
    \qquad p\geq p(\alpha).
  \end{equation}
  From the series expansion of $J_{\nu_0}(r)$ at 0,
  we know that
  \begin{equation}\label{Bessel1}
    J_{\nu_0}(r)=C_{\nu_0}r^{\nu_0}+S_{\nu_0}(r)
  \end{equation}
  where
  \begin{equation}\label{Bessel3}
    |S_{\nu_0}(r)|\leq C_{\nu_0} r^{1+\nu_0},
    \qquad
    r\in(0,2].
  \end{equation}
Defining 
    \begin{equation*}
    \tilde{Z}= r^{-\frac{n-2}2} \int_0^\infty 
    J_{\nu_0}(r\rho)e^{it\rho}\chi(\rho)\rho d\rho,
  \end{equation*}
then for any $0<\epsilon<1$, we can estimate
  \begin{equation*}
  \begin{split}
    \|Z\|_{L^q(\R;L^p(X))}&\ge
  C \|\psi_0\|_{L^2(Y)}\|\tilde{Z}\|_{L^q_{t}([0,1/2];L^p_{r^{n-1}dr}[\epsilon,1])}\\
    &\ge
    C_{\nu_0}\Big(\|P\|_{L^q_{t}([0,1/4];L^p_{r^{n-1}dr}[\epsilon,1])}
    -
    \|Q\|_{L^q_{t}([0,1/2];L^p_{r^{n-1}dr}[\epsilon,1])}\Big),
    \end{split}
  \end{equation*}
  where
  \begin{equation*}
    P= r^{-\frac{n-2}2} \int_{0}^{\infty}(r \rho)^{\nu_0}e^{it \rho}\chi(\rho)
      \rho d \rho,
    \qquad
    Q= r^{-\frac{n-2}2} \int_{0}^{\infty}S_{\nu_0}(r\rho)e^{it \rho}\chi(\rho)
      \rho d \rho.
  \end{equation*}
  Now, on one hand by \eqref{Bessel3} and \eqref{def:alpha} we have
  \begin{equation*}
  \begin{split}
    \|Q\|_{L^q_{t}([0,1/2];L^p_{r^{n-1}dr}[\epsilon,1])}
   & \lesssim
    \left\|
     r^{-\frac{n-2}2} \int_{0}^{\infty}(r \rho)^{1+\nu_0}\chi(\rho)\rho d \rho
    \right\|_{L^q_{t}([0,1/2];L^p_{r^{n-1}dr}[\epsilon,1])}\\
   & \lesssim \max\{ 1, \epsilon^{1+\alpha+\frac np}\}.
    \end{split}
  \end{equation*}
  On the other hand, we have
  \begin{equation*}
    \|P\|_{L^p([0,1/4];L^q_{rdr}[\epsilon,1])}=
    \left(\int_0^{\frac14}\left( \int_{\epsilon}^1 \left|\int_0^\infty (r\rho)^{\nu_0}e^{
    it\rho}\chi(\rho)\rho d\rho\right|^p r^{-\frac{(n-2)p}2}r^{n-1} dr\right)^{q/p}dt\right)^{1/q}
  \end{equation*}
  and by the assumption $p\geq p(\alpha)=\frac{n}{|\alpha|}$
  \begin{equation*}
    \gtrsim
    \left(\int_0^{\frac14}
    \left|\int_0^\infty \rho^{\nu_0}
    e^{it\rho}\chi(\rho)\rho d\rho\right|^{q}dt\right)^{1/q}
    \times 
    \begin{cases}
      \epsilon^{\alpha+\frac np} &
      \text{if}\quad p |\alpha|>n\\
      \ln\epsilon &
      \text{if}\quad p|\alpha|=n
    \end{cases}
    \gtrsim C
    \begin{cases}
      \epsilon^{\alpha+\frac np} &
      \text{if}\quad p|\alpha|>n\\
      \ln\epsilon &
      \text{if}\quad p|\alpha|=n.
    \end{cases}
  \end{equation*}
  In the last inequality, we have used the fact that 
  $\cos(\rho t)\geq 1/2$ for $t\in [0, 1/4]$ and 
  $\rho\in [1,2]$, so that
  \begin{equation}
  \begin{split}
  \left|\int_0^\infty \rho^{\nu_0}e^{
  it\rho}\chi(\rho)\rho d\rho\right|\geq \frac12\int_0^\infty \rho^{\nu_0}\chi(\rho)\rho d\rho\geq c.
  \end{split}
  \end{equation}
  Summing up, if $p>p(\alpha)$, we obtain 
  \begin{equation*}
    \Big\|r^{-\frac{n-2}2}\int_0^\infty 
      J_{\nu_0}(r\rho)e^{it\rho}\chi(\rho)\rho d\rho\Big\|
      _{L^q(\R;L^p(X))} \geq
      c\epsilon^{\alpha+\frac np} -C \max\{ 1, \epsilon^{1+\alpha+\frac np}\} 
      \to +\infty \quad \text{as}\, \epsilon\to 0,
  \end{equation*}
  while if $p|\alpha|=n$, we have
  \begin{equation*}
    \Big\| r^{-\frac{n-2}2}\int_0^\infty 
      J_{\nu_0}(r\rho)e^{it\rho}\chi(\rho)\rho d\rho
    \Big\|_{L^q(\R;L^p(X))}
    \ge
    c\ln\epsilon-C\max\{ 1, \epsilon\}\to +\infty \quad \text{as}\, \epsilon\to 0,
  \end{equation*}
  and this implies \eqref{est:aim}.
  
 \end{proof}

\section{The proof without NREC}
\label{sec:without-NREC}
In the previous section, we have proved Theorem \ref{thm:dispersive} under the NREC condition, and then proved Theorem \ref{thm:stri-wave}. 
In this section we remove the NREC condition by estimating the entire left hand side of \eqref{est:dispersive'} directly, instead of using the NREC to say the contribution of it near $s=\pi$ and $I_R$ is almost trivial.
 
To this end, we introduce a smooth cutoff function $\chi_\delta\in C^\infty([0,\pi])$ with small $0<\delta\ll 1$ such that
 \begin{equation}
 \chi_\delta(s)=
 \begin{cases} 1, \quad s\in[0, \delta];\\
 0, \quad s\in[2\delta,\pi],
 \end{cases}
 \end{equation}
 and set $\chi^c_\delta(s)=1- \chi_\delta(s)$.
We now decompose the left hand side of \eqref{est:dispersive'} into three terms:
\begin{equation}\label{S1}
\begin{split}
I_G(t; x, y):=\frac{(r_1r_2)^{-\frac{n-2}2}}{\pi}Q_j\int_0^\pi  W(t, {\bf m}_s)  \chi^c_\delta(\pi-s)  \cos(s\sqrt{P}) ds \, Q_j^*,
\end{split}
\end{equation}
\begin{equation}\label{S2}
\begin{split}
I_{GD}(t; x, y):&=\frac{(r_1r_2)^{-\frac{n-2}2}}{\pi}Q_j \Big(\int_0^\pi  W(t, {\bf m}_s) \chi_\delta(\pi-s) \cos(s\sqrt{P}) ds\\
&-\sin(\pi\sqrt{P})\int_0^\infty  W(t, {\bf n}_s) \chi_\delta(s) e^{-s\sqrt{P}} ds\Big) Q_j^*,
\end{split}
\end{equation}
and
\begin{equation}\label{S3}
\begin{split}
I_{D}(t; x,y):=-(r_1r_2)^{-\frac{n-2}2}Q_j\frac{\sin(\pi\sqrt{P})}{\pi}\int_0^\infty  W(t, {\bf n}_s)  \chi^c_\delta(s) e^{-s\sqrt{P}} ds \,Q_j^*.
\end{split}
\end{equation}
Here $G$ stands for geometric propagation and $D$ stands for diffraction. We are not using $I_P$ here to avoid confusion with $I_P$ in the previous approach. But this $I_G$-term can be estimated using the same argument for $I_P$ in Section~\ref{subsec:main-proof-with-NREC}. Also, this $I_D$-term can be estimated in the same manner as the $\mathsf{K}_2$-term of \eqref{eq:Poisson-X-K1-K2-decomposition} in Section~\ref{subsec:main-proof-with-NREC} as well. So we are only left to estimate the $I_{GD}$-term 
\begin{equation}\label{est:IGD}
|I_{GD}(t; x, y)|\lesssim (1+|t|)^{-(n-1)/2},
\end{equation}
and this is the goal of the rest of this section. This part is residual under the NREC, but now it is not residual anymore as we are removing that condition.

We recall its expression here:
\begin{equation*}
\begin{split}
I_{GD}(t; x, y)&=\frac{(r_1r_2)^{-\frac{n-2}2}}{\pi}Q_j\Big(\int_0^\pi  W(t, {\bf m}_s)  \chi_\delta(\pi-s) \cos(s\sqrt{P}) ds\\
&-\sin(\pi\sqrt{P})\int_0^\infty  W(t, {\bf n}_s) \chi_\delta(s) e^{-s\sqrt{P}} ds\Big)\,Q_j^*.
\end{split}
\end{equation*}

For this term,  as mentioned in \cite{JZ2}, there are two issues to overcome. 
The first one is that, as $s \to \pi$, if there is no localizer $Q_j$, the kernel of $\cos(s\sqrt{P})(\hat{x},\hat{y})$ will not have the oscillatory integral form we used in \cite[Lemma~3.1, Lemma~3.3]{JZ} when $(\hat{x},\hat{y})$ is a pair of conjugate points due to the degeneracy of the projection from the propagating Lagrangian $\mathscr{L}_\pm$ in \eqref{eq: propagating lagrangians} to the base manifold.
This issue can be overcome by investigating $Q_j\cos(s\sqrt{P})Q^*_j(\hat{x},\hat{y})$ instead, in which $(\hat{x},\hat{y})$ cannot be a pair conjugate point due to the support of $Q_j$.
The second issue is that this term will produce boundary terms at $s=\pi$ which does not vanish when integrating by parts. 
In contrast to $I_{G}(t; x, y)$, we need to take the boundary term from the first part at $s=\pi$ and the second part at $s=0$ into consideration.
The fortunate fact is that the Taylor expansion of the integrand of the first term at $s=\pi$ differ with that of the second term at $s=0$ by a sign, which makes their contributions in the process of integration by parts cancel each other.
Since we are treating the wave operator rather than the Schr\"odinger operator as we did in \cite{JZ2} and we have no known results that completely covers the current setting, we have to estimate it by hand again.

Before estimating this part, we use the integration by parts to obtain the following lemma that we will apply to the amplitude of $I_{GD}(t; x, y)$. See \cite[Lemma~4.4]{JZ} for details of the computation.

\begin{lemma}\label{lem:in-parts} For any $m\in \mathbb{N}$, we have the following identity:
\begin{equation}\label{in-parts}
\begin{split}
&\frac1{\pi} \int_0^\pi W(t, {\bf m}_s) \chi_\delta(\pi-s)\cos(\nu s) ds
-\frac{\sin(\nu \pi )}{\pi} \int_0^\infty W(t, {\bf n}_s) \chi_\delta(s) e^{-s\nu} ds\\
&=\frac{(-1)^{m}}{\pi}\int_0^\pi \Big(\frac{\partial}{\partial s}\Big)^{2m}\big( W(t, {\bf m}_s) \chi_\delta(\pi-s)\big)\big) \frac{ \cos(\nu s)}{\nu^{2m}} ds
\\&\qquad-\frac{\sin(\nu \pi)}{\pi}\int_0^\infty \Big(\frac{\partial}{\partial s}\Big)^{2m}\big( W(t, {\bf n}_s)\chi_\delta(s)\big) \frac{e^{-s\nu}}{\nu^{2m}} ds.
\end{split}
\end{equation}
\end{lemma}
 
%%%%%%%%%%%%%%%%%%%%%%%%%%%%%%%%%%%%%%%%%%%% 

Exploiting \eqref{eq:wave-poisson-jet-match}, this can be used to derive the same type identity for microlocalized version of parametrices. 
We first define the frequency localized version of $K_{\pm,N}$ and $\tilde{K}_{\pm,N}$. To avoid notations with too many sub-indices, we will fix and abbreviate the index $N$, which is the number of terms in the asymptotic sum in the parametrix construction when we have frequency localization. 
For $-1 \leq A,B \leq \infty$ and $\mk{d} \in \mk{D}(\hat{x},\hat{y})$, inspired by \eqref{K-pm-N} and \eqref{eq:Poisson-kernel-2} respectively, we set
\begin{equation}
\begin{split}
&K_{\pm,\mk{d}, [A,B]}(s,\hat{x},\hat{y}) \\
&=  \sum_{\varsigma = \pm} \int_0^\infty \chi_{[A,B]}(\rho) b_{\varsigma}(\rho \mk{d}) e^{\varsigma i \rho \mk{d}} a_{\pm,\mk{d}}(s, \hat{x},\hat{y}; \rho) e^{\pm  is\rho} \rho^{n-2} d\rho,
\end{split}
\end{equation}
where $\chi_{[A,B]}$ is a smooth cut-off function supported in  $[A,B]$ (we will be more specific when we apply this construction) and 
\begin{equation}
\begin{split}
&\tilde{K}_{\pm, \mk{d}, [A,B]}(s,\hat{x},\hat{y}) \\
&=  \sum_{\varsigma = \pm}\int_0^\infty \chi_{[A,B]}(\rho) b_{\varsigma}(\rho \mk{d}) e^{\varsigma i \rho \mk{d}} \tilde{a}_{\pm,\mk{d}}(s, \hat{x},\hat{y}; \rho) e^{(-s \pm  i\pi)\rho } \rho^{n-2} d\rho.
\end{split}
\end{equation}

Noticing that 
\begin{equation*}
\sin(\pi\sqrt{P})e^{-s\sqrt{P}}=\Im \big(e^{(-s+i\pi))\sqrt{P}}\big),
\end{equation*}
we define the corresponding frequency localized version of $I_{GD}$ to be
\begin{equation}\label{I-GD-L}
\begin{split}
&I_{GD,[A,B]}(t; x, y)\\
&:= \sum_{\mk{d} \in \mk{D}(\hat{x},\hat{y})} (r_1r_2)^{-\frac{n-2}2} \Big( \frac{1}{\pi}\int_0^\pi  
W(t, {\bf m}_s) \chi_\delta(\pi-s)  \frac{1}{2}\Big(K_{+, \mk{d}, [A,B]}(s,\hat{x},\hat{y})+K_{-, \mk{d}, [A,B]}(s,\hat{x},\hat{y})\Big)  ds\\
&-\frac{1}{\pi}\int_0^\infty  W(t, {\bf n}_s) \chi_\delta(s)
\frac{1}{2i}\Big(\tilde{K}_{+, \mk{d}, [A,B]}(s,\hat{x},\hat{y})-\tilde{K}_{-, \mk{d}, [A,B]}(s,\hat{x},\hat{y})\Big) ds \Big).
\end{split}
\end{equation}

\begin{proposition} \label{prop:localized-GD-IBP} For any $m \in \mathbb{N}$, $I_{GD,[A,B]}$ defined in \eqref{I-GD-L} can be rewritten as
\begin{equation} \label{eq:localized-GD-IBP}
\begin{split}
&I_{GD,[A,B]}(t;x, y)\\
&=\sum_{\mk{d} \in \mk{D}(\hat{x},\hat{y})} (r_1r_2)^{-\frac{n-2}2} \Big[\frac{1}{\pi}\int_0^\pi  
P_m(t,r_1,r_2,s) \frac{1}{2}\Big(K_{+,\mk{d}, [A,B],m}(s,\hat{x},\hat{y})+K_{-,\mk{d}, [A,B],m}(s,\hat{x},\hat{y})\Big)  ds\\
&-\frac{1}{\pi}\int_0^\infty  Q_m(t,r_1,r_2,s) \frac{1}{2i}\Big((\tilde{K}_{+, \mk{d}, [A,B],m}(s,\hat{x},\hat{y})-\tilde{K}_{-,\mk{d}, [A,B],m}(s,\hat{x},\hat{y})\Big) ds\Big],
\end{split}
\end{equation}
where $P_m(t,r_1,r_2,s)$ is a (linear) combination of derivatives of $W(t, {\bf m}_s) \chi_\delta(\pi-s)$ with respect to $s$ up to $2m$-th order and 
$Q_m(t,r_1,r_2,s)$ is a (linear) combination of derivatives of $W(t, {\bf n}_s) \chi_\delta(s)$ with respect to $s$ up to $2m$-th order, and most importantly, for $\mk{d} \in \mk{D}(\hat{x},\hat{y})$, $K_{\pm,\mk{d}, [A,B],m}(s,\hat{x},\hat{y})$, $\tilde{K}_{\pm, \mk{d}, [A,B],m}(s,\hat{x},\hat{y})$ are defined to be
\begin{align*}
\begin{split}
K_{\pm,\mk{d}, [A,B],m}(s,\hat{x},\hat{y}) = &  \sum_{\varsigma = \pm} \int_0^\infty \chi_{[A,B]}(\rho) b_{\varsigma}(\rho \mk{d}) e^{\varsigma i \rho \mk{d}} a_{\pm,m,\mk{d}}(s, \hat{x},\hat{y}; \rho) e^{\pm  is\rho} \rho^{n-2} d\rho,\\
\tilde{K}_{\pm, \mk{d}, [A,B],m}(s,\hat{x},\hat{y}) = &\sum_{\varsigma = \pm}\int_0^\infty \chi_{[A,B]}(\rho) b_{\varsigma}(\rho \mk{d}) e^{\varsigma i \rho \mk{d}} \tilde{a}_{\pm,m,\mk{d}}(s, \hat{x},\hat{y}; \rho) e^{(-s \pm  i\pi)\rho } \rho^{n-2} d\rho,
\end{split}
\end{align*}
where $a_{\pm,m, \mk{d}}, \tilde{a}_{\pm,m,\mk{d}} \in S^{-2m}$ and the symbol order is in terms of $\rho$.
\end{proposition}

\begin{proof}
This follows from integrating by parts in $s$ by writing 
\begin{equation*}
e^{is\rho} = i^{-1}\rho^{-1}\partial_s(e^{is\rho}), 
\end{equation*}
and our amplitudes are supported in $\rho \geq 1$.
The boundary terms from two parts cancels each other as in the proof of Lemma~\ref{lem:in-parts}, in combination with \eqref{eq:wave-poisson-jet-match}, which deals with terms having derivatives falling on 
$a_{\pm,\mk{d}}$ and $\tilde{a}_{\pm,\mk{d}}$.
\end{proof}

Using these identities for frequency localization, the modified Hadamard parametrices in Lemma \ref{lemma: localized-halfwave} for half-wave operator and  in Lemma \ref{lemma: localized-poisson}  for the Poisson-wave operator, we split the kernel $I_{GD}(t; x, y)$ into three parts
\begin{equation}
\begin{split}
I_{GD}(t; x, y)=I^{<\kappa}_{GD}(t;x, y)+I^{>\kappa}_{GD}(t;x, y)+R(t; x, y),
\end{split}
\end{equation}
where
\begin{equation}
I^{<\kappa}_{GD}(t;x, y) = I_{GD,[-1,2\kappa]}(t;x, y),
\quad I^{>\kappa}_{GD}(t;x, y) = I_{GD,[\kappa,\infty]}(t;x, y),
\end{equation}
are given in \eqref{I-GD-L} and we choose cut-off functions such that $\chi_{[-1,2\kappa]}+\chi_{[\kappa,\infty]}$ is identically $1$ on $[0,\infty)$. \footnote{The choice of $\chi_{[-1,2\kappa]}$ on $[-1,0)$ is unimportant, as long as it is smooth.}
Here $R(t; x, y)$ contains contributions from smooth terms $R_N(s; \hat{x}, \hat{y})$ and $\tilde{R}_N(s; \hat{x}, \hat{y})$ in Lemma~\ref{lemma: parametrix 1} and Lemma~\ref{lemma: localized-poisson}:
\begin{equation}
\begin{split}
R(t; x, y)=(r_1r_2)^{-\frac{n-2}2}&\Big(\int_0^\pi  W(t, {\bf m}_s) \chi_\delta(\pi-s) R_N(s; \hat{x}, \hat{y}) \,ds\\
&\quad+\int_0^\infty W(t, {\bf n}_s) \chi_\delta(s) \tilde{R}_N(s; \hat{x}, \hat{y}) ds\Big).
\end{split}
\end{equation}

Explicitly, using \eqref{eq:localized-GD-IBP}, for any $m\geq0$, we have 
\begin{equation} \label{I>}
\begin{split}
&I^{>\kappa}_{GD}(t;x, y)\\
=&\sum_{\mk{d} \in \mk{D}(\hat{x},\hat{y})} (r_1r_2)^{-\frac{n-2}2} \Big[\frac{1}{\pi}\int_0^\pi  
P_m(t,r_1,r_2,s) \frac{1}{2}\Big(K_{+,\mk{d}, [\kappa,+\infty],m}(s,\hat{x},\hat{y})+K_{-,\mk{d}, [\kappa,+\infty],m}(s,\hat{x},\hat{y})\Big)  ds\\
&-\frac{1}{\pi}\int_0^\infty  Q_m(t,r_1,r_2,s) \frac{1}{2i}\Big((\tilde{K}_{+, \mk{d}, [\kappa,+\infty],m}(s,\hat{x},\hat{y})-\tilde{K}_{-,\mk{d}, [\kappa,+\infty],m}(s,\hat{x},\hat{y})\Big) ds\Big],
\end{split}
\end{equation}
and
\begin{equation} \label{I<}
\begin{split}
&I^{<\kappa}_{GD}(t;x, y) \\
=&\sum_{\mk{d} \in \mk{D}(\hat{x},\hat{y})} (r_1r_2)^{-\frac{n-2}2}\Big[\frac{1}{\pi}\int_0^\pi  
P_m(t,r_1,r_2,s) \frac{1}{2}\Big(K_{+,\mk{d}, [-1,2\kappa],m}(s,\hat{x},\hat{y})+K_{-,\mk{d}, [-1,2\kappa],m}(s,\hat{x},\hat{y})\Big)  ds\\
&-\frac{1}{\pi}\int_0^\infty  Q_m(t,r_1,r_2,s) \frac{1}{2i}\Big((\tilde{K}_{+, \mk{d}, [-1,2\kappa],m}(s,\hat{x},\hat{y})-\tilde{K}_{-,\mk{d}, [-1,2\kappa],m}(s,\hat{x},\hat{y})\Big) ds\Big].
\end{split}
\end{equation}
In the following argument, we will take  $m$ large enough to treat with the high frequency term $I^{>\kappa}_{GD}(t;x, y)$, while take $m=0$ to estimate $I^{<\kappa}_{GD}(t;x, y)$, which equals to 
\begin{equation} \label{I<'}
\begin{split}
&I^{<\kappa}_{GD}(t;x, y) 
=\sum_{\mk{d} \in \mk{D}(\hat{x},\hat{y})} (r_1r_2)^{-\frac{n-2}2}\\&\Big[\frac{1}{\pi}\int_0^\pi  
W(t, {\bf m}_s) \chi_\delta(\pi-s) \sum_{\varsigma = \pm} \int_0^\infty \chi_{[-1,2\kappa]}(\rho) b_{\varsigma}(\rho \mk{d}) e^{\varsigma i \rho \mk{d}} a_{\mk{d}}(s, \hat{x},\hat{y}; \rho) \cos(s\rho) \rho^{n-2} d\rho  ds\\
&-\frac{1}{\pi}\int_0^\infty  W(t, {\bf n}_s) \chi_\delta(s) \sum_{\varsigma = \pm}\int_0^\infty \chi_{[-1,2\kappa]}(\rho) b_{\varsigma}(\rho \mk{d}) e^{\varsigma i \rho \mk{d}} \tilde{a}_{\mk{d}}(s, \hat{x},\hat{y}; \rho) e^{(-s \pm  i\pi)\rho } \rho^{n-2} d\rho ds\Big].
\end{split}
\end{equation}

\vspace{0.2cm}

%%%%%%%%%%%%%%%%%%%%%%%%%%%%%%%%%%%%%% 

Now we return to estimate $I_{GD}(t;x, y)$. The first term in $R(t; x, y)$ can be estimated  using \eqref{est:RN} associated with $R_{N}$ and ${\bf m}_s$.
Similar to the case of \eqref{est:RN}, we estimate the second term associated with $\tilde{R}_{N}$ and ${\bf n}_s$  by using \eqref{est:W-n} with $m=0$:
 \begin{equation}\label{est:RN'}
\begin{split}
&(r_1r_2)^{-\frac{n-2}2}\Big|\int_0^\infty W(t, {\bf n}_s) \chi_\delta(s) \tilde{R}_N(s; \hat{x}, \hat{y}) ds\Big| \\
&\lesssim (r_1r_2)^{-\frac{n-2}2}\int_0^\infty \big(1+|t\pm |{\bf n}_s||\big)^{-N}(1+|{\bf n}_s|)^{-1/2}\chi_\delta(s)  \,ds.
  \end{split}
 \end{equation}

If either $|t|\ll |{\bf n}_s|$ or $|t|\gg |{\bf n}_s|$, we have $|t\pm |{\bf n}_s||\geq c|t|$ for some constant $c$, and then we choose $N$ above large enough to obtain
\begin{equation*}
\eqref{est:RN'}\lesssim (r_1r_2)^{-\frac{n-2}2} (1+|t|)^{-N}\lesssim  (1+|t|)^{-\frac{n-1}2}
\end{equation*}
where we used $r_1, r_2\gg1$ in the last inequality. 
Otherwise, $|t|\sim |{\bf n}_s|$, then from \eqref{fact:ns} and \eqref{assu:r}, we see $|t|\sim r_1\gg1$.
So we obtain
\begin{equation*}
\eqref{est:RN'}\lesssim (r_1r_2)^{-\frac{n-2}2} (1+|t|)^{-1/2}\lesssim (1+|t|)^{-(n-1)/2}.
\end{equation*}

From now on, we focus on estimating $I^{<\kappa}_{GD}(t;x, y)$ and $I^{>\kappa}_{GD}(t;x, y)$.
In these terms, since $\pi-\delta\leq s\leq\pi$, combining with \eqref{assu:r} and \eqref{fact:ms} (we assumed either $r_1 \gg r_2$ or $r_1 \sim r_2$), we always have 
\begin{equation}\label{fact:ms'}
|{\bf m}_s|\sim r_1.
\end{equation}
Similarly, as in the proof of the uniform boundedness of \eqref{S1}, we only consider the case $\mk{d} = d_h(\hat{x}, \hat{y})$ and omit the $\mk{d}$-index. Terms with all other $\mk{d}$ can be bounded in the same (in fact simpler) way.
For this $d_h$-term, we consider two cases that $d_h(\hat{x}, \hat{y})\leq C_1 r_2^{-\frac12}$ and $d_h(\hat{x}, \hat{y})\geq C_1 r_2^{-\frac12}$ where $C_1\gg1$. In each case, we choose different $\kappa$ in the argument. In contrast to the part in which we try to bound \eqref{S1}, the fact \eqref{fact:ms'} enables us to consider this term without dividing into two cases $t\geq1$ and $t\leq 1$. \vspace{0.2cm}

\textbf{Case 1}. $d_h(\hat{x}, \hat{y})\leq C_1 r_2^{-\frac12}$. In this case, we take $\kappa=2r_2^{\frac12}$. We first consider $I^{<\kappa}_{GD}(t; x, y)$. For this low frequency term, since we do not do integration by parts in $ds$, the boundary issue mentioned above will not be involved, so we use \eqref{I<} with $m=0$. For the first term of \eqref{I<'},  we modify the argument for \eqref{osi-1<b} due to the fact that \eqref{fact:ms'}.
 In this case, we take $J$ large enough so that $2^{J-1}\geq 2C_1$ and we want to show that\footnote{In the following argument, we omit the $\pm, \mk{d}$-index of the amplitudes $a, \tilde{a}$, and left the $m$-index to denote that  $a_m, \tilde{a}_m\in S^{-2m}$ where the symbol order is in terms of $\rho$. } 
 \begin{equation}\label{osi-1<b-s>}
 \begin{split}
&(r_1r_2)^{-\frac{n-2}2}
\Big|\int_0^\pi W(t, {\bf m}_s)  \chi_\delta(\pi-s)\Big(\beta_{J}(r_2^{1/2} s)+\sum_{j\ge J+1}\beta(2^{-j}r_2^{1/2}s)\Big) \\
&\times \int_{0}^\infty \chi_{[-1,2\kappa]}(\rho) b_\pm(\rho d_h) e^{\pm i \rho d_h} a_0(s, \hat{x}, \hat{y}; \rho) \cos(s \rho) \rho^{n-2}d\rho\, ds \Big|\lesssim (1+|t|)^{-(n-1)/2}.
\end{split}
 \end{equation}
 For the term associated with $\beta_J$, we have $|s|\lesssim r_2^{-\frac12}\ll 1$ due to the compact support of $\beta_J$. Thus this term vanishes since $s$ also is close to $\pi$ due to the compact support of $\chi_\delta$. 
 
 For the terms with $\beta(2^{-j}r_2^{1/2}s), j \geq J$, we have $ 2^{j-1}r_2^{-1/2} \leq s \leq 2^{j+1}r_2^{-1/2}$ and $2^j\lesssim  r_2^{1/2}$ on the support of this $\beta-$factor.  In this case, we will show that 
   \begin{equation}\label{beta-j-2}
 \begin{split}
 &(r_1r_2)^{-\frac{n-2}2}
\Big|\int_0^\pi W(t, {\bf m}_s)  \beta(2^{-j}r_2^{1/2}s)\\
&\qquad\times \int_{0}^\infty \chi_{[-1,2\kappa]}(\rho) b_\pm(\rho d_h) e^{\pm i \rho d_h} a_0(s, \hat{x}, \hat{y}; \rho) \cos(s \rho) \rho^{n-2}d\rho ds\Big|\\
&\qquad \lesssim 2^{-j(n-2)}(1+|t|)^{-(n-1)/2},
\end{split}
 \end{equation}
which would give us desired bounds after summing over $j$ when $n\geq3$.
Now we repeat the previous argument, for the contribution from the part $\rho\le 2^{-j}r_2^{1/2}$, we do not do any integration by parts, the integral in \eqref{beta-j-2} is always bounded by (again, by \eqref{est:W} with $m=0$):
 \begin{equation*} 
 \begin{split}
& (r_1r_2)^{-\frac{n-2}2} \int_{|s|\sim r_2^{-\frac12}2^j} \big(1+|t\pm |{\bf m}_s||\big)^{-N}(1+|{\bf m}_s|)^{-1/2} ds\, (2^{-j}r_2^{\frac{1}{2}})^{n-1}\\
 &\lesssim (r_1r_2)^{-\frac{n-2}2} (r_2^{-\frac12}2^j) (2^{-j}r_2^{\frac{1}{2}})^{n-1} \big(\chi_{\{|t|\lesssim r_1\}}(1+|t|)^{-1/2} +(1+|t|)^{-N}\big)
\\ & \lesssim 2^{-j(n-2)}(1+|t|)^{-(n-1)/2}.
 \end{split}
 \end{equation*}

On the other hand, for the part $\rho\ge 2^{-j}r_2^{1/2}$, we write $\cos(s \rho)=\frac12\big(e^{is\rho}+e^{-is\rho}\big)$, then we do integration by parts in $d\rho$ instead, then each time we gain a factor of $\rho^{-1}$, and we at most lose a factor of $|s \pm d_h|^{-1}$. Recalling that $J$ is large enough so that $2^{J-2}$ is larger than $C_1$, then we have
\begin{equation*}
| s\pm d_h|^{-1}\lesssim \Big((2^{j-1}-C_1)r_2^{-\frac12}\Big)^{-1}\lesssim \Big((2^{j-2}+2^{J-2}-C_1)r_2^{-\frac12}\Big)^{-1}\sim 2^{-j}r_2^{\frac12}.
\end{equation*}
So after integration by parts $N$ times for $N \ge n$, the integral in \eqref{beta-j-2} is bounded by 
 \begin{equation}\nonumber
 \begin{split}
& (r_1r_2)^{-\frac{n-2}2} (r_2^{-\frac12}2^j)  \big(2^{-j}r_2^{\frac12}\big)^{N} \big(\chi_{\{|t|\lesssim r_1\}}(1+|t|)^{-1/2} +(1+|t|)^{-N}\big) \int_{2^{-j}r_2^{1/2}}^\infty \rho^{n-2-N}d\rho\\
&\lesssim 2^{-j(n-2)}(1+|t|)^{-(n-1)/2}.
\end{split}
  \end{equation}
 
\bigskip
 
For the second term of \eqref{I<'}, we want to show that 
 \begin{equation}\label{osi-1<d}
 \begin{split}
 (r_1r_2)^{-\frac{n-2}2}&
\Big|\int_0^\infty W(t, {\bf n}_s)  \chi_\delta(s) \Big(\beta_{J}(r_2^{1/2} s)+\sum_{j\ge J+1}\beta(2^{-j}r_2^{1/2}s)\Big) \\
&\times \int_{0}^{\infty} \chi_{[-1,2\kappa]}(\rho) b_\pm(\rho d_h) e^{\pm i \rho d_h} \tilde{a}_0(s, \hat{x}, \hat{y}; \rho) e^{-(s\pm i\pi)\rho} \rho^{n-2}d\rho\, ds \Big|\\
&\lesssim (1+|t|)^{-(n-1)/2},
\end{split}
\end{equation}
where $\beta$ and $\beta_J$ are in \eqref{beta-d} with $2^{J-2}\geq C_1$.
For the term associated with $\beta_J$, we have $|s|\lesssim r_2^{-\frac12}\ll 1$ due to the compact support of $\beta_J$. 
We have $\rho\le \kappa=4r_2^{1/2}$ for this part because of the $\chi_{[-1,2\kappa]}$-factor in \eqref{I<}, thus the integral in \eqref{osi-1<d} with $\beta_J$ is always bounded by 
\begin{equation}\label{d-rho-l}
\begin{split}
 (r_1r_2)^{-\frac{n-2}2}
&\int_{|s|\lesssim r_2^{-\frac12}} \big(1+|t\pm |{\bf n}_s||\big)^{-N}(1+|{\bf n}_s|)^{-1/2}\chi_\delta(s)  \,ds
\int_{\rho\leq 4r_2^{\frac12}} \rho^{n-2}d\rho. \end{split}
 \end{equation}
 If either $|t|\ll |{\bf n}_s|$ or $|t|\gg |{\bf n}_s|$,  we have $|t\pm |{\bf n}_s||\geq c|t|$ for some constant $c$, and then we choose $N$ large enough to obtain
\begin{equation*}
\eqref{d-rho-l}\lesssim (r_1r_2)^{-\frac{n-2}2} r_2^{-1/2} r_2^{\frac{n-1}{2}} (1+|t|)^{-N}\lesssim  (1+|t|)^{-\frac{n-1}2},
\end{equation*}
where the last inequality is due to the fact that $r_1, r_2\gg1$.
Otherwise, $|t|\sim |{\bf n}_s|$, then from \eqref{fact:ns} and \eqref{assu:r}, we see $|t|\sim r_1\gg1$.
So we obtain
\begin{equation*}
\eqref{d-rho-l}\lesssim (r_1r_2)^{-\frac{n-2}2}r_2^{-1/2} r_2^{\frac{n-1}{2}} (1+|t|)^{-1/2}\lesssim (1+|t|)^{-(n-1)/2}.
\end{equation*}
 
 For the terms with $\beta(2^{-j}r_2^{1/2}s), j \geq J$, we have $s \sim 2^{j}r_2^{-1/2}$, and $2^j \lesssim  r_2^{1/2}$, due to the compact support of $\beta$.  In this case, we will show that 
   \begin{equation}\label{beta-d-j}
 \begin{split}
 (r_1r_2)^{-\frac{n-2}2}&
\Big|\int_0^\infty W(t, {\bf n}_s)  \chi_\delta(s)\beta(2^{-j}r_2^{1/2}s)\\
&\times \int_{0}^{\infty} \chi_{[-1,2\kappa]}(\rho) b_\pm(\rho d_h) e^{\pm i \rho d_h} \tilde{a}_0(s, \hat{x}, \hat{y}; \rho) e^{-(s\pm i\pi)\rho} \rho^{n-2}d\rho\, ds \Big| \\
&\lesssim 2^{-j(n-2)}(1+|t|)^{-(n-1)/2},
\end{split}
 \end{equation}
which would give us desired bounds after summing over $j$.
Similar as before, for the contribution from the part $\rho\le 2^{-j}r_2^{1/2}$, we do not do any integration by parts, the integral in \eqref{beta-d-j} is always bounded by 
 \begin{equation}\nonumber (r_1r_2)^{-\frac{n-2}2} (r_2^{-\frac12}2^j) (2^{-j}r_2^{\frac{1}{2}})^{n-1}\big(\chi_{r_1\sim |t|} (1+|t|)^{-1/2}+ (1+|t|)^{-N}\big) \lesssim 2^{-j(n-2)}(1+|t|)^{-(n-1)/2}.
 \end{equation}

On the other hand, for the part $\rho\ge 2^{-j}r_2^{1/2}$, we use the factor $e^{-(s\pm i\pi)\rho}$ to do integration by parts in $d\rho$ instead, then each time we gain a factor of $\rho^{-1}$, and we at most lose factors of 
\begin{equation*}
|s\pm i\pi|^{-1}, d_h\lesssim 1.
\end{equation*}
So after integration by parts $N$ times for $N\ge n$, the integral in \eqref{beta-d-j} is bounded by 
 \begin{equation}\nonumber
 \begin{split}
&(r_1r_2)^{-\frac{n-2}2} (r_2^{-\frac12}2^j) \big(\chi_{r_1\sim |t|} (1+|t|)^{-1/2}+ (1+|t|)^{-N}\big)   \int_{2^{-j}r_2^{1/2}}^{2\kappa=4r_2^{1/2}} \rho^{n-2-N}d\rho\\
&\lesssim (r_1r_2)^{-\frac{n-2}2}   (2^{-j}r_2^{1/2})^{n-1-N}\big(\chi_{r_1\sim |t|} (1+|t|)^{-1/2}+ (1+|t|)^{-N}\big)\\
& \lesssim 2^{-j(n-2)}(1+|t|)^{-(n-1)/2}
\end{split}
  \end{equation}
due to that $2^j\lesssim  r_2^{1/2}$.

Next we consider $I^{>\kappa}_{GD}(t; x, y)$. For this high frequency term, we use \eqref{I>} with $m$ large enough.  
We aim to show that 
 \begin{equation}\label{osi-1<gd}
 \begin{split}
& (r_1r_2)^{-\frac{n-2}2}
\Big|\int_0^\pi P_m(t,r_1,r_2,s)\Big(\beta_{J}(r_2^{1/2} s)+\sum_{j\ge J+1}\beta(2^{-j}r_2^{1/2}s)\Big) \\
&\times \int_{0}^\infty \chi_{[\kappa,\infty]}(\rho) b_\pm(\rho d_h) e^{\pm i \rho d_h} a_m(s, \hat{x}, \hat{y}; \rho) \cos(s \rho) \rho^{n-2}d\rho\, ds \Big|\lesssim (1+|t|)^{-(n-1)/2}.
\end{split}
 \end{equation}
 For the term associated with $\beta_J$, we have $|s|\lesssim r_2^{-\frac12}\ll 1$ due to the compact support of $\beta_J$. Recall that $P_m(t,r_1,r_2,s)$ is $2m$-derivatives of 
 $\big(W(t, {\bf m}_s)  \chi_\delta(\pi-s)\big)$, so it is supported $|s-\pi|\leq \delta$, thus this term vanishes.

For the terms with $\beta(2^{-j}r_2^{1/2}s)$, we have $2^{j-1}r_2^{-1/2} \leq s \leq 2^{j+1}r_2^{-1/2}$ and $2^j \lesssim  r_2^{1/2}$ on its support. And by our construction we have $j \geq J+1$, hence $2^{j-2} > C_1$. This will be used to bound $|s-d_h|^{-1}$ in our proof below.

By using \eqref{est:W} and \eqref{fact:ms'} again, we have for all $m' \in \mathbb{N}$:
\begin{equation*}
\begin{split}
&\Big|\Big(\frac{\partial}{\partial s}\Big)^{2m'}\Big(W(t, {\bf m}_s) \chi_\delta(\pi-s)\Big)\Big|
\leq C_{m'} 2^{2m'j} r_2^{m'} \big(1+|t\pm |{\bf m}_s||\big)^{-N}(1+|{\bf m}_s|)^{-1/2}.
\end{split}
 \end{equation*}

In the rest of the proof, we will show that 
\begin{equation}\label{beta-j-gd}
\begin{split}
 (r_1r_2)^{-\frac{n-2}2}&
\Big|\int_0^\pi P_m(t,r_1,r_2,s)\beta(2^{-j}z^{1/2}s) \\
&\times \int_{0}^\infty \chi_{[\kappa,\infty]}(\rho) b_\pm(\rho d_h) e^{\pm i \rho d_h} a_m(s, \hat{x}, \hat{y}; \rho) \cos(s \rho) \rho^{n-2}d\rho\, ds\Big| \\
&\lesssim 2^{-j(n-2)}(1+|t|)^{-\frac{n-1}2},
\end{split}
 \end{equation}
which would give us desired bounds after summing over $j$.

Recall that in the current part, we have $\rho \ge \kappa =2r_2^{1/2}$ due to the $\chi_{[\kappa,\infty]}$-factor. We write $\cos(s \rho)=\frac12\big(e^{is\rho}+e^{-is\rho}\big)$ and we integrate by parts in $d\rho$ instead. Then each time we gain a factor of $\rho^{-1}$, and we at most lose a factor of (by our choice of $J$, $s$ will dominate $d_h$)
\begin{equation*}
 |s \pm d_h|^{-1}\lesssim 2^{-j}r_2^{\frac12}.
\end{equation*}
 So after integration by parts $N$ times for $N\ge n+2m$, the integral in \eqref{beta-j-gd} is bounded by 
 \begin{equation}\nonumber
 \begin{split}
&(r_1r_2)^{-\frac{n-2}2} (r_2^{-\frac12}2^j) \big(2^{2mj} r_2^m\big)  \big(2^{-j}r_2^{\frac12}\big)^{N} \big(\chi_{r_1\sim |t|} (1+|t|)^{-1/2}+ (1+|t|)^{-N}\big) \int_{4r_2^{1/2}}^\infty \rho^{n-2-2m-N}d\rho\\
&\lesssim 2^{-j(n-2)}(1+|t|)^{-\frac{n-1}2},
\end{split}
  \end{equation}
which proves \eqref{osi-1<gd} for $n\geq 3$.

To treat the second term of \eqref{I>}, we closely follow the argument above but with minor modifications. The thing remains to prove is that
 \begin{equation}\label{osi-1<gd-d}
 \begin{split}
& (r_1r_2)^{-\frac{n-2}2}
\Big|\int_0^\infty  Q_m(t,r_1,r_2,s) \Big(\beta_{J}(r_2^{1/2} s)+\sum_{j\ge J+1}\beta(2^{-j}r_2^{1/2}s)\Big) \\
&\times \int_{\kappa=2r_2^{1/2}}^\infty b_\pm(\rho d_h) e^{\pm i \rho d_h} \tilde{a}_m(s, \hat{x}, \hat{y}; \rho) e^{-(s\pm i\pi)} \rho^{n-2}d\rho\, ds \Big|\lesssim (1+|t|)^{-\frac{n-1}2}.
\end{split}
 \end{equation}
 For the term associated with $\beta_J$, we have $|s|\lesssim r_2^{-\frac12}\ll 1$ due to the compact support of $\beta_J$. 
 Since \eqref{fact:ns} gives $|{\bf n}_s| \sim r_1$ in the current setting, by using \eqref{est:W-n}, for all $m' \in \mathbb{N}$, we have  
 \begin{equation*}
\begin{split}
\Big|\Big(\frac{\partial}{\partial s}\Big)^{2m'}\Big(W(t, {\bf n}_s) \chi_\delta(s)\Big)\Big|
\leq C_{m'} r_2^{m'} \big(1+|t\pm |{\bf n}_s||\big)^{-N}(1+|{\bf n}_s|)^{-1/2}.
\end{split}
\end{equation*}
So for $2m\ge n$, the integral in \eqref{osi-1<gd-d} is bounded by 
 $$(r_1r_2)^{-\frac{n-2}2} r_2^{-1/2} r_2^{m} \big(\chi_{r_1\sim |t|} (1+|t|)^{-1/2}+ (1+|t|)^{-N}\big)  \int_{r_2^{1/2}}^\infty \rho^{n-2-2m}d\rho\lesssim (1+|t|)^{-\frac{n-1}2}.
 $$
 For the terms with $\beta(2^{-j}r_2^{1/2}s)$, we have $s \approx 2^{j}r_2^{-1/2}$, and $2^j\lesssim  r_2^{1/2}$, due to the compact support of $\beta$. Therefore, as above we have
\begin{equation*}
\begin{split}
&\Big|\Big(\frac{\partial}{\partial s}\Big)^{2m'}\Big(W(t, {\bf n}_s) \chi_\delta(s)\Big)\Big|
\leq C_{m'} 2^{2m'j}  r_2^{m'} \big(1+|t\pm |{\bf n}_s||\big)^{-N}(1+|{\bf n}_s|)^{-1/2}.
\end{split}
 \end{equation*}
Since $\rho\ge\kappa =2r_2^{1/2}$ (again, by the $\chi_{[\kappa,\infty]}$-factor), we use the factor $e^{-(s\pm i\pi)}$ to do integration by parts in $d\rho$ instead, then each time we gain a factor of $\rho^{-1}$, and we at most lose factors of 
\begin{equation*}
|s \pm i\pi|^{-1},\, d_h \lesssim 1\lesssim 2^{-j}r_2^{\frac12}.
\end{equation*}
So after integration by parts $N$ times for $N\ge n+2m$, the integral in \eqref{osi-1<gd-d} with $\beta(2^{-j}r_2^{1/2}s)$ is bounded by 
\begin{equation*}
\begin{split}
&  (r_1r_2)^{-\frac{n-2}2}
\Big|\int_0^\pi Q_m(t,r_1,r_2,s) \beta(2^{-j}r_2^{1/2}s) \\
&\qquad \qquad\times \int_{0}^\infty \chi_{[\kappa,\infty]}(\rho) b_\pm(\rho d_h) e^{\pm i \rho d_h} \tilde{a}_m(s, \hat{x}, \hat{y}; \rho) e^{-(s\pm i\pi)} \rho^{n-2}d\rho\, ds\Big| \\&\lesssim 
(r_1r_2)^{-\frac{n-2}2} (r_2^{-\frac12}2^j) \big(2^{2mj} r_2^m\big)  \big(2^{-j}r_2^{\frac12}\big)^{N}  \big(\chi_{r_1\sim |t|} (1+|t|)^{-1/2}+ (1+|t|)^{-N}\big)  \int_{2r_2^{1/2}}^\infty \rho^{n-2-2m-N}d\rho\\
&\lesssim 2^{-j(n-2)}  (1+|t|)^{-\frac{n-1}2},
\end{split}
 \end{equation*}
which would give us desired bounds \eqref{osi-1<gd-d} after summing over $j$ provided $n\geq3$.\vspace{0.2cm}

\textbf{Case 2}. $d_h(\hat{x}, \hat{y})\geq C_1 r_2^{-\frac12}$.  In this case, we take $\kappa=r_2d_h$ and $J=0$ in $\beta_J$ \eqref{beta-d}. We first consider $I^{<\kappa}_{GD}(t; x, y)$. As in the proof of \eqref{osi-1>b}, one can control the first term of \eqref{I<}, since we do not use the integration by parts in $ds$. 

In this case, taking $J=0$, we will show that 
 \begin{equation}\label{osi-1>b-old}
 \begin{split}
&(r_1r_2)^{-\frac{n-2}2}
\Big|\int_0^\pi W(t, {\bf m}_s)  \chi_\delta(\pi-s)\Big(\beta_{0}(r_2 d_h |s-d_h|)+\sum_{j\ge 1}\beta(2^{-j}r_2 d_h |s-d_h|)\Big) \\
&\times \int_{0}^\infty b_\pm(\rho d_h) e^{\pm i \rho d_h} a_0(s, \hat{x}, \hat{y}; \rho) \cos(s \rho) \rho^{n-2}d\rho\, ds \Big|\lesssim (1+|t|)^{-(n-1)/2},
\end{split}
 \end{equation}
where $\beta_0$ and $\beta$ are same to the above ones \eqref{beta-d}. 

 For the term associated with $\beta_0$, we have $|s-d_h|\leq (r_2 d_h)^{-1}\lesssim r_2^{-\frac12}$ due to the compact support of $\beta_0$. If we also have $\rho\le r_2 d_h$,  thus the integral in \eqref{osi-1>b-old} with $\beta_0$ is always bounded by 
  \begin{equation}
 \begin{split}
& (r_1r_2)^{-\frac{n-2}2}
\int_{|s-d_h|\lesssim (r_2d_h)^{-1}}\big(1+|t\pm |{\bf m}_s||\big)^{-N}(1+|{\bf m}_s|)^{-1/2}   ds
\int_{\rho\leq r_2 d_h} (1+\rho d_h)^{-\frac{n-2}2} \rho^{n-2}d\rho 
\\&\lesssim (r_1r_2)^{-\frac{n-2}2} (r_2d_h)^{-1} (r_2d_h)^{\frac{n-2}{2}+1} d_h^{-\frac{n-2}2}\big(\chi_{\{|t|\lesssim r_1\}}(1+|t|)^{-1/2} +(1+|t|)^{-N}\big) \\
& \lesssim (1+|t|)^{-\frac{n-1}2}.
\end{split}
 \end{equation}
 
For the terms associated with $\beta(2^{-j}r_2 d_h |s-d_h|), j \geq 1$, we have $|s-d_h| \approx 2^{j}(r_2 d_h)^{-1}$, due to the support condition of $\beta$, and $2^j\lesssim r_2 d_h$ since $s,d_h$ are bounded. In this case, we will show that
\begin{equation}\label{beta-j'}
\begin{split}
 &(r_1r_2)^{-\frac{n-2}2}
\Big|\int_0^\pi W(t, {\bf m}_s) \chi^c_\delta(\pi-s) \beta(2^{-j}r_2 d_h |s-d_h|)\\
&\times \int_{0}^\infty b_\pm(\rho d_h) e^{\pm i \rho d_h} a_0(s, \hat{x}, \hat{y}; \rho) \cos(s \rho) \rho^{n-2}d\rho ds\Big| \lesssim 2^{-j\frac{n-2}2}(1+|t|)^{-\frac{n-1}2},
\end{split}
 \end{equation}
which would give us desired bounds \eqref{osi-1>b-old} after summing over $j\geq1$.
 Now we repeat the argument above. If in this case we have $\rho\le 2^{-j}r_2 d_h$, then we do not do any integration by parts, the integral in \eqref{beta-j'} is always bounded by 
 \begin{equation}\nonumber 
 \begin{split}
& (r_1r_2)^{-\frac{n-2}2} \int_{|s-d_h|\sim 2^j(r_2d_h)^{-1}}\big(1+|t\pm |{\bf m}_s||\big)^{-N}(1+|{\bf m}_s|)^{-1/2} \,ds \int_{\rho\leq 2^{-j}r_2 d_h}(1+\rho d_h)^{-\frac{n-2}2} \rho^{n-2}\, d\rho\\
 &\lesssim (r_1r_2)^{-\frac{n-2}2} ((r_2d_h)^{-1}2^j) (2^{-j}r_2 d_h)^{\frac{n-2}2+1} d_h^{-\frac{n-2}2} \big(\chi_{\{|t|\lesssim r_1\}}(1+|t|)^{-1/2} +(1+|t|)^{-N}\big)\\
 & \lesssim 2^{-j\frac{n-2}2}(1+|t|)^{-\frac{n-1}2}.
 \end{split}
 \end{equation}

On the other hand, if we have $\rho\ge 2^{-j}r_2 d_h$, we write $\cos(s \rho)=\frac12\big(e^{is\rho}+e^{-is\rho}\big)$, then we do integration by parts in $d\rho$ again, then each time we gain a factor of $\rho^{-1}$, and we at most lose a factor of 
$$|s \pm d_h|^{-1}\lesssim  2^{-j}r_2 d_h.$$ 
So after integration by parts $N$ times for $N\ge n$, 
 the integral in \eqref{beta-j'} is bounded by 
 \begin{equation}\nonumber
 \begin{split}
&(r_1r_2)^{-\frac{n-2}2}   \big(2^{-j}r_2 d_h\big)^{N} \int_{|s-d_h|\sim 2^j(r_2d_h)^{-1}}\big(1+|t\pm |{\bf m}_s||\big)^{-N}(1+|{\bf m}_s|)^{-1/2} \,ds  \\& \qquad\qquad\times \int_{2^{-j}r_2 d_h}^\infty \rho^{\frac{n-2}2-N} d_h^{-\frac{n-2}2}d\rho\\
&\lesssim r_1^{-\frac{n-2}2} (r_2 d_h)^{-\frac{n-2}2-1} 2^j  \big(2^{-j}r_2 d_h\big)^{N} \big(2^{-j}r_2 d_h\big)^{\frac{n-2}2+1-N}  \big(\chi_{\{|t|\lesssim r_1\}}(1+|t|)^{-1/2} +(1+|t|)^{-N}\big)\\
 & \lesssim 2^{-j\frac{n-2}2}(1+|t|)^{-\frac{n-1}2}.
\end{split}
  \end{equation}
  Therefore we have proved \eqref{osi-1>b-old}. \vspace{0.2cm}

For the second term of \eqref{I<} with $m=0$ (i.e. \eqref{I<'}), we want to show that 
 \begin{equation}\label{osi-1<d2}
 \begin{split}
 (r_1r_2)^{-\frac{n-2}2}&
\Big|\int_0^\infty W(t, {\bf n}_s) \chi_\delta(s) \Big(\beta_{0}(r_2d_h |s-d_h|)+\sum_{j\ge 1}\beta(2^{-j}r_2d_h |s-d_h|)\Big) \\
&\times \int_{0}^{\infty} \chi_{[-1,2\kappa]}(\rho) b_\pm(\rho d_h) e^{\pm i \rho d_h} a_0(s, \hat{x}, \hat{y}; \rho) e^{-(s\pm i\pi)\rho} \rho^{n-2}d\rho\, ds \Big|\\
&\lesssim (1+|t|)^{-\frac{n-1}2}.
\end{split}
 \end{equation}
For the term corresponding to $\beta_0$, we have $|s-d_h|\leq (r_2d_h)^{-1}\lesssim r_2^{-\frac12}\ll 1$ due to the compact support of $\beta_0$. Due to that $\rho\le \kappa=r_2d_h$,  thus the integral in \eqref{osi-1<d2} with $\beta_0$ is always bounded by 
  \begin{equation}\label{d-rho-l2}
 \begin{split}
 &(r_1r_2)^{-\frac{n-2}2}
\int_{|s-d_h|\lesssim (r_2d_h)^{-1}} \big(1+|t\pm |{\bf n}_s||\big)^{-N}(1+|{\bf n}_s|)^{-1/2}  \chi_\delta(s)  ds\\
&\qquad\qquad\quad \times\int_{\rho\leq r_2d_h} (1+\rho d_h)^{-\frac{n-2}2}\rho^{n-2}d\rho 
\\&\lesssim (r_1r_2)^{-\frac{n-2}2} (r_2d_h)^{-1} (r_2d_h)^{\frac{n-2}{2}+1} d_h^{-\frac{n-2}2} \big(\chi_{\{|t|\lesssim r_1\}}(1+|t|)^{-1/2} +(1+|t|)^{-N}\big)\\
&\lesssim (1+|t|)^{-\frac{n-1}2}.
\end{split}
 \end{equation}
 For the terms associated with $\beta(2^{-j}r_2 d_h |s-d_h|)$, we have $|s-d_h| \approx 2^{j}(r_2 d_h)^{-1}$, $j\ge 1$ and $2^j\lesssim r_2 d_h$, due to the compact support of $\beta$.  In this case, we want to show that 
   \begin{equation}\label{beta-j'2}
 \begin{split}
 (r_1r_2)^{-\frac{n-2}2}&
\Big|\int_0^\infty W(t, {\bf n}_s) \chi_\delta(s) \beta(2^{-j}r_2 d_h |s-d_h|)\\
&\times \int_{0}^{\infty} \chi_{[-1,2\kappa]}(\rho) b_\pm(\rho d_h) e^{\pm i \rho d_h} a_0(s, \hat{x}, \hat{y}; \rho) e^{-(s\pm i\pi)\rho} \rho^{n-2}d\rho\, ds\Big| \\
&\lesssim 2^{-j\frac{n-2}2}(1+|t|)^{-\frac{n-1}2},
\end{split}
 \end{equation}
which would give us desired bounds \eqref{osi-1<d2} after summing over $j$.
If in this case we have $\rho\le 2^{-j}r_2 d_h$, then we do not do any integration by parts, the integral in \eqref{beta-j'} is always bounded by 
 \begin{equation}\nonumber 
 \begin{split}
& (r_1r_2)^{-\frac{n-2}2} \int_{|s-d_h|\sim 2^j(r_2d_h)^{-1}} \big(1+|t\pm |{\bf n}_s||\big)^{-N}(1+|{\bf n}_s|)^{-1/2}  \,ds\int_{\rho\leq 2^{-j}r_2 d_h}(1+\rho d_h)^{-\frac{n-2}2} \rho^{n-2}\, d\rho\\
 &\lesssim (r_1r_2)^{-\frac{n-2}2} ((r_2d_h)^{-1}2^j) (2^{-j}r_2 d_h)^{\frac{n-2}2+1} d_h^{-\frac{n-2}2}\big(\chi_{|t|\sim r_1}(1+|t|)^{-1/2}+(1+|t|)^{-N}\big)\\
 & \lesssim 2^{-j\frac{n-2}2}(1+|t|)^{-\frac{n-1}2}.
 \end{split}
 \end{equation}

On the other hand, if we have $\rho\ge 2^{-j}r_2d_h$, we use the factor $e^{-(s\pm i\pi)}$ to do integration by parts in $d\rho$ instead, then each time we gain a factor of $\rho^{-1}$, and we at most lose factors of 
\begin{equation*}
|s\pm i\pi|^{-1},\, d_h \lesssim 1\lesssim 2^{-j}r_2d_h.
\end{equation*}
So after integration by parts $m$ times for $m\ge n$, the integral in \eqref{beta-j'2} is bounded by 
 \begin{equation}\nonumber
 \begin{split}
&(r_1r_2)^{-\frac{n-2}2} 2^j (r_2d_h)^{-1}   \big(2^{-j}r_2d_h\big)^{m} \big(\chi_{|t|\sim r_1}(1+|t|)^{-1/2}+(1+|t|)^{-N}\big)  \int_{2^{-j}r_2d_h}^\infty \rho^{\frac{n-2}2-m} d_h^{-\frac{n-2}2}d\rho\\
&\lesssim (r_2d_h)^{-\frac{n-2}2-1} 2^j  \big(2^{-j}r_2d_h\big)^{m} \big(2^{-j}r_2d_h\big)^{\frac{n-2}2+1-m}\big(\chi_{|t|\sim r_1}(1+|t|)^{-1/2}+(1+|t|)^{-N}\big) \\
&\lesssim 2^{-j\frac{n-2}2}(1+|t|)^{-\frac{n-1}2}.
\end{split}
  \end{equation}
   \vspace{0.2cm}

Next we consider $I^{>\kappa}_{GD}(t; x, y)$. For this high frequency term, we use \eqref{I>} with $m$ large enough. For the first term of \eqref{I>}, since the term with $R_N(s; \hat{x}, \hat{y})$ is easy as before, we only need to consider the term associated with $K_N(s; \hat{x}, \hat{y})$. We need to show that 
 \begin{equation}\label{osi-1>gd'}
 \begin{split}
 (r_1r_2)^{-\frac{n-2}2}
&\Big|\int_0^\pi P_m(t,r_1,r_2,s) \Big(\beta_{0}(r_2d_h |s-d_h|)+\sum_{j\ge 1}\beta(2^{-j}r_2d_h |s-d_h|)\Big) \\
&\times \int_{0}^\infty \chi_{[\kappa,\infty]}(\rho) b_\pm(\rho d_h) e^{\pm i \rho d_h} a_m(s, \hat{x}, \hat{y}; \rho) \cos(s \rho) \rho^{n-2}d\rho\, ds \Big|\\
&\lesssim (1+|t|)^{-(n-1)/2}.
\end{split}
 \end{equation}
 For the term associated with $\beta_0$, we have $|s-d_h|\leq (r_2d_h)^{-1}\lesssim r_2^{-\frac12}\ll 1$ due to the compact support of $\beta_0$. Therefore, $s\leq d_h+(r_2d_h)^{-1}\lesssim d_h$, we have
\begin{equation}\nonumber
\begin{split}
 \Big|\Big(\frac{\partial}{\partial s}\Big)^{2m}\Big( W(t, {\bf m}_s) \chi_\delta(\pi-s)\Big)\Big|&\leq C_m \big(1+|t\pm |{\bf m}_s||\big)^{-N}(1+|{\bf m}_s|)^{-1/2} (1+r_2\sin s)^{2m}\\
 &\lesssim (r_2d_h)^{2m}\big(1+|t\pm |{\bf m}_s||\big)^{-N}(1+|{\bf m}_s|)^{-1/2},
\end{split}
\end{equation}
and also the same form of upper bounds for lower order derivatives. So for $2m\ge n$, the integral in \eqref{osi-1>gd'} is bounded by 
 \begin{equation}\nonumber
 \begin{split}
& (r_1r_2)^{-\frac{n-2}2} (r_2d_h)^{-1} (r_2d_h)^{2m} \big(\chi_{|t|\sim r_1}(1+|t|)^{-1/2}+(1+|t|)^{-N}\big) \int_{r_2d_h}^\infty (1+\rho d_h)^{-\frac{n-2}2} \rho^{n-2-2m}d\rho\\
 &\lesssim r_1^{-\frac{n-2}2} \big(\chi_{|t|\sim r_1}(1+|t|)^{-1/2}+(1+|t|)^{-N}\big)\lesssim (1+|t|)^{-(n-1)/2}.
  \end{split}
  \end{equation}
Hence, we have proved \eqref{osi-1>gd'} with $\beta_0$.
 For the terms with $\beta(2^{-j}r_2d_h |s-d_h|)$, we have $|s-d_h| \approx 2^{j}(r_2d_h)^{-1}$, and $2^j\lesssim  r_2d_h$, due to the compact support of $\beta$. 
 Therefore, $s\leq d_h+2^j(r_2d_h)^{-1}$, we have
  \begin{equation}\nonumber
 \begin{split}
 &\Big|\Big(\frac{\partial}{\partial s}\Big)^{2m}\Big(W(t, {\bf m}_s)  \chi_\delta(\pi-s)\Big)\Big|\\
 &\leq C_m (1+r_2\sin s)^{2m}\big(1+|t\pm |{\bf m}_s||\big)^{-N}(1+|{\bf m}_s|)^{-1/2}\big)\\
& \leq C_m \big(1+|t\pm |{\bf m}_s||\big)^{-N}(1+|{\bf m}_s|)^{-1/2}(r_2d_h+2^j r_2(r_2d_h)^{-1})^{2m}\\
& \lesssim \big(1+|t\pm |{\bf m}_s||\big)^{-N}(1+|{\bf m}_s|)^{-1/2}\big((r_2d_h)^{2m}+(2^jr_2^{\frac12})^{2m}\big).
 \end{split}
  \end{equation}
  In this case, we will show that 
   \begin{equation} \label{beta-j-gd'}
 \begin{split}
& (r_1r_2)^{-\frac{n-2}2}
\Big|\int_0^\pi P_m(t,r_1,r_2,s) \beta(2^{-j}r_2d_h |s-d_h|) \\
&\times \int_{0}^\infty \chi_{[\kappa,\infty]}(\rho) b_\pm(\rho d_h) e^{\pm i \rho d_h} a_m(s, \hat{x}, \hat{y}; \rho) \cos(s \rho) \rho^{n-2}d\rho\, ds\Big| \lesssim 2^{-j(n-2)}(1+|t|)^{-(n-1)/2},
\end{split}
 \end{equation}
which would give us desired bounds after summing over $j$.
To prove this, since $\rho\ge\kappa =r_2d_h$, we write $\cos(s \rho)=\frac12\big(e^{is\rho}+e^{-is\rho}\big)$, then we do integration by parts in $d\rho$ instead, then each time we gain a factor of $\rho^{-1}$, and we at most lose a factor of 
\begin{equation*} 
|s\pm d_h|^{-1}\lesssim 2^{-j}r_2d_h.
\end{equation*} 
 So after integration by parts $N$ times for $N\ge n+2m$, 
 the integral in \eqref{beta-j-gd'} is bounded by 
 \begin{equation}\nonumber
 \begin{split}
(r_1r_2)^{-\frac{n-2}2} &(2^j(r_2d_h)^{-1}) \big[(r_2d_h)^{2m}+(2^jr_2^{\frac12})^{2m}\big]  \big(2^{-j}r_2d_h\big)^{N}  \int_{r_2d_h}^\infty (1+\rho d_h)^{-\frac{n-2}2}\rho^{n-2-2m-N}d\rho\\
&\times \big(\chi_{|t|\sim r_1}(1+|t|)^{-1/2}+(1+|t|)^{-N}\big)\lesssim 2^{-j(N-2m-1)} (1+|t|)^{-(n-1)/2}.
\end{split}
  \end{equation}
Therefore, we have proved \eqref{osi-1>gd'} for $n\geq 3$.

To treat the second term of \eqref{I>}, we closely follow the above argument but with minor modifications. 
Since the term with $\tilde{R}_N(s; \hat{x}, \hat{y})$ can be bounded in the same manner as before and we only consider the term associated with $\tilde{K}_N(s; \hat{x}, \hat{y})$. The bound we need to show is
 \begin{equation} \label{osi-1<gd-d'}
 \begin{split}
 (r_1r_2)^{-\frac{n-2}2}&
\Big|\int_0^\pi Q_m(t,r_1,r_2,s) \Big(\beta_{0}(r_2d_h |s-d_h|)+\sum_{j\ge 1}\beta(2^{-j}r_2d_h |s-d_h|)\Big) \\
&\times \int_{0}^\infty \chi_{[\kappa,\infty]}(\rho) b_\pm(\rho d_h) e^{\pm i \rho d_h} \tilde{a}_m(s, \hat{x}, \hat{y}; \rho) e^{-(s\pm i\pi)} \rho^{n-2}d\rho\, ds \Big|\\
&\lesssim (1+|t|)^{-(n-1)/2}.
\end{split}
 \end{equation}
 For the term associated with $\beta_0$, we have $|s-d_h|\leq (r_2d_h)^{-1}\lesssim r_2^{-\frac12}\ll 1$ due to the compact support of $\beta_0$. Therefore, $s=d_h+(r_2d_h)^{-1}\lesssim d_h$, we have
      \begin{equation*}
 \begin{split}
&\Big|\Big(\frac{\partial}{\partial s}\Big)^{2m}\Big(W(t, {\bf n}_s) \chi_\delta(s)\Big)\Big|\\
&\leq C_m (1+r_2\sinh s)^{2m}\big(1+|t\pm |{\bf n}_s||\big)^{-N}(1+|{\bf n}_s|)^{-1/2}\\
&\leq C_m  (r_2d_h)^{2m} \big(1+|t\pm |{\bf n}_s||\big)^{-N}(1+|{\bf n}_s|)^{-1/2}.
\end{split}
 \end{equation*}
So for $2m\ge n$, 
  the integral in \eqref{osi-1<gd-d'} is bounded by 
      \begin{equation*}
 \begin{split}
 (r_1r_2)^{-\frac{n-2}2} (r_2d_h)^{-1} (r_2d_h)^{2m} &\big(\chi_{|t|\sim r_1}(1+|t|)^{-1/2}+(1+|t|)^{-N}\big) \int_{r_2d_h}^\infty \rho^{n-2-2m}d\rho\\
 &\lesssim (1+|t|)^{-(n-1)/2}.
\end{split}
 \end{equation*}
 For the terms with $\beta(2^{-j}r_2d_h |s-d_h|)$, we have $|s-d_h| \approx 2^{j}(r_2d_h)^{-1}$, and $2^j\lesssim  r_2d_h$, due to the compact support of $\beta$. 
 Therefore, $s=d_h+2^j(r_2d_h)^{-1}\in [0,\delta]$, we have
       \begin{equation*}
 \begin{split}
&\Big|\Big(\frac{\partial}{\partial s}\Big)^{2m}\Big(W(t, {\bf n}_s) \chi_\delta(s)\Big)\Big|\\
&\leq C_m  \Big((r_2d_h)^{2m}+(2^jr_2^{\frac12})^{2m} \Big)\big(1+|t\pm |{\bf n}_s||\big)^{-N}(1+|{\bf n}_s|)^{-1/2}.
\end{split}
 \end{equation*}
  In this case, we will show that 
   \begin{equation}\label{beta-j-gd-d2}
 \begin{split}
 (r_1r_2)^{-\frac{n-2}2}&
\Big|\int_0^\pi Q_m(t,r_1,r_2,s) \beta(2^{-j}r_2d_h |s-d_h|) \\
&\times \int_{0}^\infty \chi_{[\kappa,\infty]}(\rho) b_\pm(\rho d_h) e^{\pm i \rho d_h} \tilde{a}_m(s, \hat{x}, \hat{y}; \rho) e^{-(s\pm i\pi)\rho} \rho^{n-2}d\rho\, ds\Big| \\
&\lesssim 2^{-j(n-2)} (1+|t|)^{-(n-1)/2}.
\end{split}
 \end{equation}
which would give us desired bounds after summing over $j$.
 
Now we modify the previous argument, since $\rho\ge\kappa =r_2d_h$, we use the factor $e^{-(s\pm i\pi)}$ to do integration by parts in $d\rho$ instead, then each time we gain a factor of $\rho^{-1}$, and we at most lose factors of 
\begin{equation*}
|s\pm i\pi|^{-1},\, d_h \lesssim 1\lesssim 2^{-j}r_2d_h.
\end{equation*}
So after integration by parts $N$ times for $N\ge n+2m$, the integral in \eqref{beta-j-gd-d2} is bounded by 
  \begin{equation}\nonumber
 \begin{split}
&(r_1r_2)^{-\frac{n-2}2} (2^j(r_2d_h)^{-1}) \big[(r_2d_h)^{2m}+(2^jr_2^{\frac12})^{2m}\big]  \big(2^{-j}r_2d_h\big)^{N}  \int_{r_2d_h}^\infty (1+\rho d_h)^{-\frac{n-2}2}\rho^{n-2-2m-N}d\rho\\
&\times \big(\chi_{|t|\sim r_1}(1+|t|)^{-1/2}+(1+|t|)^{-N}\big) \lesssim 2^{-j(N-2m-1)}(1+|t|)^{-(n-1)/2},
\end{split}
  \end{equation}
which would give us desired bounds \eqref{osi-1<gd-d'} after summing over $j$ provided $n\geq3$.

In summary, we have shown that $I_{GD}(t; x, y)$ is also uniformly bounded by $O((1+|t|)^{-\frac{n-1}2})$ when $r_1, r_2\gg1$, concluding the proof of \eqref{est:IGD}.

\appendix
\section{About non-focusing condition and Bessel function }
In this appendix, we first prove the Proposition \ref{prop:non-focusing-curvature-condition}, which provides some examples of curvature conditions that implies the non-focusing condition (NFC).
Next, we provide a lemma about the asymptotic property of Bessel function.

\begin{proof}[The proof of Proposition \ref{prop:non-focusing-curvature-condition}]
We consider the case of $\mathscr{L}_+$ and the case with the other sign can be proved in the same manner.

By the compactness of $Y$, we only need to show that for any $y \in Y$, there is a neighborhood $U_y$ of $(y,y)$ in $Y \times Y$ such that 
\begin{equation}
\mathscr{P}_+^{-1}(U_y) 
= \bigcup_i U_i,
\end{equation}
where $U_i$ are pairwise disjoint and $\mathscr{P}_+|_{U_i \cap (S^*\R \times S^*Y \times S^*Y)}$ is a diffeomorphism.
Notice that the property in Definition~\ref{defn:non-focusing-Lagrangian} is satisfied if the exponential map is non-degenerate over $\mathcal{U}$. 

Consider the case $K<1$ first. Now we prove by contradiction, suppose $(Y,h)$ has sectional curvature less than $1$ and does not have non-focusing wave propagation relation within time $\pi$ in the sense of Definition~\ref{defn:non-focusing-Lagrangian}.
Then we know that there is a sequence of points $\mk{p}_i = (s_i,\tau_i,y_{1,i},\hat{\mu}_{1,i},y_{2,i},\hat{\mu}_{2,i}) \in \mathscr{L}_+$ with $s_i \in (0,\pi]$ such that the differential of $\mathscr{P}_+|_{ S^*\R \times S^*Y \times S^*Y}$ at $\mk{p}_i$ is degenerate and $\mathscr{P}_+(\mk{p}_i) \to (y,y)$.

Notice that this happens if and only if the exponential map starting at $(y_{2,i},\hat{\mu}_{2,i})$ is degenerate at corresponding points, which can't happen if $s_i \in (0,\mathrm{inj}(Y))$. So we know $s_i \geq \mathrm{inj}(Y)$ and $\mk{p}_i$ is in a compact set. 
In addition, we can assume that 
\begin{align}
\limsup_{i \to \infty} s_i \leq \pi
\end{align}
since we assumed by contraction that no $\epsilon$ as in Definition~\ref{defn:non-focusing-Lagrangian} exists.
Consequently, we can assume that it converges to a limit
\begin{equation} \label{eq: s0,loop}
(s_0,-1,y,\hat{\mu}_1,y,\hat{\mu}_2), \quad s_0 \in [\mathrm{inj}(Y),\pi],
\end{equation}
after passing to a subsequence.
In particular, we know $\exp(s_0\mathsf{H}_p)(y,\hat{\mu}_2) = (y,\hat{\mu}_1)$.

However, by the curvature condition, Rauch's comparison theorem tells us
that $\exp$ is non-degenerate until time $\pi$, contradiction.

Now we turn to the case $\frac{1}{2} \leq K < 2$ and $Y$ is compact and simply connected. 
In this case, we know there is a $K_{\max}<2$ such that $\frac{1}{2} \leq K \leq K_{\max}$. So we can select a constant $\lambda<2$ such that $\lambda h$ has sectional curvature $K' = \frac{K}{\lambda}$ such that $\frac{1}{4} \leq K' \leq 1$. Consequently by the $\frac{1}{4}$-pinching theorem of Klingenberg \cite{klingenberg1962}, Klingenberg-Sakai \cite{klingenberg1980} and Cheeger-Gromoll \cite{cheeger1980}, we know that the injective radius of $(Y,\lambda h)$ is at least $\pi$, which in turn shows that the injective radius of $(Y,h)$ is strictly larger than $\frac{\pi}{2}$.

Applying the argument above in the first case, we still have a loop formed as in \eqref{eq: s0,loop}.
Now we consider $(y',\hat{\mu}')=\exp(\frac{s_0}{2}\mathsf{H}_p)(y,\hat{\mu}_2)$. Then there are two geodesics of length $\frac{s_0}{2}$ from $\hat{y}$ to $y'$ with initial (co)velocity $\hat{\mu}_2$ and $-\hat{\mu}_1$ respectively. So we have $\frac{s_0}{2} \geq \mathrm{inj}(Y)>\frac{\pi}{2}$, which gives $s_0 > \pi$, contradiction.

\end{proof}

%%%%%%%%%%%%%%%%%%%%%%%%%%%%%%%%%%%%%%%
We record a lemma about the property of the Bessel function. 
\begin{lemma}\label{lem:bessel}
  For all $r,\nu\in \mathbb{R}^+$, there exist  constants $C_{\nu}$ and $C_{\nu,N}$
  depending only on $\nu$ and $\nu,N$ respectively such that
    \begin{equation}\label{eq:bess1}
    |J_{\nu}(r)|\le C_{\nu} r^{\nu}(1+r)^{-\nu-\frac 12},
  \end{equation}
  \begin{equation}\label{eq:bess2}
    |J'_{\nu}(r)|=
    |J_{\nu-1}(r)-\nu J_{\nu}(r)/r|
    \le C_{\nu}r^{\nu-1}(1+r)^{-\nu+\frac 12}.
  \end{equation}
Moreover we can write
  \begin{equation}\label{eq:bess3}
    J_{\nu}(r)=r^{-1/2}(e^{ir}j_{+}(\nu, r)+e^{-ir}j_{-}(\nu, r))
  \end{equation}
  for two functions $j_{\pm}$ depending on $\nu,r$ and satisfying for all $N\ge1$ and $r\ge1$
  \begin{equation}\label{eq:bess4}
    |j_{\pm}(\nu, r)|\le C_{\nu,0},
    \qquad
    |\partial_{r}^{N}j_{\pm}(\nu, r)|\le C_{\nu,N}
    r^{-N-1}.
  \end{equation}
  Furthermore, if $r\geq 2\nu\geq 1$,  for all $N\ge0$, there exist constant $C_N$ independent of $\nu$ such that
  \begin{equation}\label{eq:bess5}
    |\partial_{r}^{N}j_{\pm}(\nu, r)|\le C_{N}
   \, 2^\nu \big(\frac\nu r\big)^{N}.
  \end{equation}
\end{lemma}

\begin{remark} The estimations \eqref{eq:bess1}- \eqref{eq:bess4} are mainly used for $\nu$ in a fixed range, and the last one \eqref{eq:bess5} is used for $2\nu\le r$ as $\nu\to\infty$.
\end{remark}

\begin{proof} The estimations \eqref{eq:bess1}- \eqref{eq:bess4} have been proved in \cite[Lemma 5.1]{CDYZ}. 
Now we only prove \eqref{eq:bess5} by detecting the parameter $\nu$.

 We recall the integral representation of the
  Macdonald function $K_{\nu}(z)$:
  \begin{equation}\label{eq:macd}
    K_{\nu}(z)=b_{\nu}z^{-1/2}e^{-z}
    \int_{0}^{\infty}
    e^{-t}t^{\nu-1/2}
    \left(1+\frac{t}{2z}\right)^{\nu-1/2}dt
  \end{equation}
  where
  $b_{\nu}=\sqrt{\frac \pi2}/\Gamma(\nu+\frac 12)$.
  The representation \eqref{eq:macd} is valid for
  all $z\not\in(-\infty,0]$ and $\Re \nu>-\frac 12$
  (see e.g.~\cite{EMO} page 19).
  By the standard connection formulas
  \begin{equation*} 
    \textstyle
    K_{\nu}(z)=
    \frac{i\pi}{2}e^{i\nu\pi/2}H^{(1)}_{\nu}(iz)=
    -\frac{i\pi}{2}e^{-i\nu\pi/2}H^{(2)}_{\nu}(-iz)
  \end{equation*}
  we deduce the analogous representations, for some
  constants $b'_{\nu},b''_{\nu}$ and

  \begin{equation}\label{eq:hank1}
    H_{\nu}^{(1)}(z)=b'_{\nu}
    z^{-1/2}e^{iz}
    \int_{0}^{\infty}
    e^{-t}t^{\nu-1/2}
    \left(1+\frac{it}{2z}\right)^{\nu-1/2}dt,
  \end{equation}
  \begin{equation}\label{eq:hank2}
    H_{\nu}^{(2)}(z)=b''_{\nu}
    z^{-1/2}e^{-iz}
    \int_{0}^{\infty}
    e^{-t}t^{\nu-1/2}
    \left(1+\frac{t}{2iz}\right)^{\nu-1/2}dt.
  \end{equation}
  where $\Re{\nu}>-1/2$ and
  $$b'_{\nu}=\frac{2e^{-i\nu\pi/2}}{\sqrt{i}\pi} b_{\nu}, \quad b''_{\nu}=\frac{2e^{i\nu\pi/2}}{i\sqrt{i}\pi} b_{\nu}.$$
  Both \eqref{eq:hank1}--\eqref{eq:hank2} are valid at
  least for $\Re z>0$, and we shall use them for
  $z=r\in(0,\infty)$ and $\nu>-1/2$.
  Consider the function on $r\in [1,\infty)$
  \begin{equation*}
    j_+(\nu,r)=
    b'_{\nu}\int_{0}^{\infty}
    e^{-t}t^{\nu-1/2}
    \left(1+\frac{it}{2r}\right)^{\nu-1/2}dt.
  \end{equation*}
  Since 
  $$\Big|\left(1+\frac{it}{2r}\right)^{\nu-1/2}\Big|=
    \big(1+\frac{t^{2}}{4r^{2}}\big)^{\frac \nu2-\frac 14}\le
    \big(1+\frac{t}{2r}\big)^{\nu-\frac12}\leq  \big(1+\frac{\nu}{r}\frac{t}{2\nu}\big)^{\nu}\leq \big(1+\frac{t}{2\nu}\big)^{\nu} \leq e^{\frac{t}2}$$
    provided $r\geq 2\nu$,
  we have (by a change of variable and the definition of the $\Gamma$ function):
  \begin{equation}\label{eq:esta}
    |  j_+(\nu,r)|\le C b_{\nu} \int_{0}^{\infty}
    e^{-\frac t2}t^{\nu-1/2} dt\le C b_{\nu} \Gamma(\nu+\frac12) 2^{\nu}\le C 2^{\nu}
    \quad\text{for}\quad r\geq 2\nu\ge1.
  \end{equation}
 For $r$-derivatives, we have
  \begin{equation*}
    \left|
    \partial_{r}\left(1+\frac{it}{2r}\right)^{\nu-1/2}
    \right|=
    \left|
    \frac{\nu-1/2}{r}
    \left(1+\frac{it}{2r}\right)^{\nu-3/2}\frac{it}{2r}
    \right|
    \le\frac{\nu}r
    (1+\frac{t}{2r})^{\nu-\frac12}
  \end{equation*}
  and more generally 
  \begin{equation*}
    \left|
    \partial_{r}^{N}\left(1+\frac{it}{2r}\right)^{\nu-1/2}
    \right|\le
    \big(\frac{\nu}r\big)^N (1+\frac{t}{2r})^{\nu-\frac12}\leq  \big(\frac{\nu}r\big)^N e^{\frac t2}, 
    \qquad
    r\ge 2\nu\geq 1.
  \end{equation*}
  This implies
  \begin{equation}\label{eq:estan}
    |\partial_{r}^{N}  j_+(\nu,r)|\le
   C 2^{\nu}  \big(\frac{\nu}r\big)^N
    \qquad
    r\ge 2\nu\geq 1.
  \end{equation}
  Recalling \eqref{eq:hank1} we have thus proved,
  for $r\geq 2\nu\geq1$,
  \begin{equation*}
    H^{(1)}_{\nu}(r)=r^{-1/2}e^{ir}  j_+(\nu,r)
    \quad\text{with $ j_+(\nu,r)$ satisfying }\quad 
    \eqref{eq:esta}, \eqref{eq:estan}.
  \end{equation*}
  In a similar way we prove, for $r\geq 2\nu\geq1$,
  \begin{equation*}
    H^{(2)}_{\nu}(x)=r^{-1/2}e^{-ir}  j_-(\nu,r)
    \quad\text{with $ j_-(\nu,r)$ satisfying }\quad 
    \eqref{eq:esta}, \eqref{eq:estan}.
  \end{equation*}
  Since $2J_{\nu}=H^{(1)}_{\nu}+H^{(2)}_{\nu}$,
  we have proved \eqref{eq:bess5}
  for all $r\geq 2\nu\geq1$.
\end{proof}

\begin{center}

\end{center}

\end{document}